\documentclass[11pt,a4paper,reqno]{amsart} 
\usepackage{amsmath}
\usepackage{amssymb}
\usepackage{euscript}
\usepackage{multirow}
\usepackage{graphicx}
\usepackage{xcolor}
\usepackage[colorlinks=true,linktocpage=true,pagebackref=false,citecolor=black,linkcolor=black]{hyperref}
\usepackage{pifont}
\usepackage{mathabx}
\usepackage{tikz,pgfplots}
\pgfplotsset{compat=newest,compat/show suggested version=false}
\usepgfplotslibrary{fillbetween}
\usetikzlibrary{cd,arrows,decorations.pathmorphing,backgrounds,automata,positioning,fit,matrix}

\pgfplotsset{soldot/.style={color=black,only marks,mark=*}} \pgfplotsset{holdot/.style={color=black,fill=white,only marks,mark=*}}

\newtheorem{thm}{Theorem}[section]
\newtheorem{mthm}[thm]{Main Theorem}
\newtheorem{lem}[thm]{Lemma}

\newtheorem{cor}[thm]{Corollary}

\theoremstyle{definition}

\theoremstyle{remark}
\newtheorem{remark}[thm]{Remark}
\newtheorem{remarks}[thm]{Remarks}

\newtheorem{examples}[thm]{Examples}

\newtheorem{prob}[thm]{Problem}

\numberwithin{equation}{section}

 \newcommand{\R}{{\mathbb R}}
 \newcommand{\C}{{\mathbb C}}

\newcommand{\sph}{{\mathbb S}} 
 \newcommand{\PP}{{\mathbb P}}

 \newcommand{\Cont}{{\mathcal C}}

\newcommand{\Ss}{{\EuScript S}}
\newcommand{\Tt}{{\EuScript T}}

\newcommand{\Bb}{{\EuScript B}}
\newcommand{\Cc}{{\EuScript C}}
\newcommand{\pol}{{\EuScript K}}

\newcommand{\Dd}{{\EuScript D}}

\newcommand{\Uu}{{\EuScript U}}

\newcommand{\Reg}{\operatorname{Reg}}
\newcommand{\Sing}{\operatorname{Sing}}

\newcommand{\Int}{\operatorname{Int}}
\newcommand{\cl}{\operatorname{Cl}}

\newcommand{\dist}{\operatorname{dist}}

\newcommand{\id}{\operatorname{id}}

\newcommand{\x}{{\tt x}} \newcommand{\y}{{\tt y}} 
\newcommand{\z}{{\tt z}} \renewcommand{\t}{{\tt t}}
\newcommand{\s}{{\tt s}}

\newcommand{\veps}{\varepsilon}

\newcommand{\ol}{\overline}

\begin{document}

\title[On a Nash curve selection lemma through finitely many points]{On a Nash curve selection lemma\\ through finitely many points}

\author{Jos\'e F. Fernando}
\address{Departamento de \'Algebra, Geometr\'\i a y Topolog\'\i a, Facultad de Ciencias Matem\'aticas, Universidad Complutense de Madrid, Plaza de Ciencias 3, 28040 MADRID (SPAIN)}
\email{josefer@mat.ucm.es}
\thanks{Author is supported by Spanish STRANO PID2021-122752NB-I00 and Grupos UCM 910444.}

\date{31/03/2025}

\begin{abstract}
A celebrated theorem in Real Algebraic and Analytic Geometry (originally due to Bruhat-Cartan and Wallace and stated later in its current form by Milnor) is the (Nash) curve selection lemma, which has wide applications also in Complex Algebraic and Analytic Geometry. It states that each point in the closure of a semialgebraic set $\Ss\subset\R^n$ can be reached by a Nash arc of $\R^n$ such that at least one of its branches is contained in $\Ss$. 

The purpose of this work is to generalize the previous result to finitely many points. More precisely, let $\Ss\subset\R^n$ be a semialgebraic set, let $x_1,\ldots,x_r\in\Ss$ be $r$ points (that we call `control points') and $0=:t_1<\ldots<t_r:=1$ be $r$ values (that we call `control times'). A natural `logistic' question concerns the existence of a smooth and semialgebraic (Nash) path $\alpha:[0,1]\to\Ss$ that passes through the control points at the control times, that is, $\alpha(t_k)=x_k$ for $k=1,\ldots,r$. The necessary and sufficient condition to guarantee the existence of $\alpha$ when the number of control points is large enough and they are in general position is that $\Ss$ is connected by analytic paths. The existence of generic real algebraic sets that do not contain rational curves confirms that the analogous result involving polynomial paths (instead of Nash paths) is only possible under additional restrictions. A sufficient condition is that $\Ss\subset\R^n$ has in addition dimension $n$. 

A related problem concerns the approximation by a Nash path of an existing continuous semialgebraic path $\beta:[0,1]\to\Ss$ with control points $x_1,\ldots,x_r\in\Ss$ and control times $0=:t_1<\ldots<t_r:=1$. As one can expect, apart from the restrictions on $\Ss$, some restrictions on $\beta$ are needed. A sufficient condition is that the (finite) set of values $\eta(\beta)$ at which $\beta$ is not smooth is contained in the set of regular points of $\Ss$ and $\eta(\beta)$ does not meet the set of control times. 

If $\Ss\subset\R^n$ is a finite union (connected by analytic paths) of $n$-dimensional convex polyhedra, we can even `estimate' (using Bernstein's polynomials) the degree of the involved polynomial path. This requires: (1) a polynomial double curve selection lemma for convex polyhedra involving only degree $3$ cuspidal curves; (2) to find the simplest polynomial paths that connect two convex polyhedra (whose union is connected by analytic paths), and (3) some improvements concerning well-known bounds for Bernstein's polynomials (and their high order derivatives) to approximate continuous functions that are not differentiable on their whole domain.
\end{abstract}

\maketitle

\section{Introduction}\label{s1}

A natural `logistic' problem in Real Geometry, whose affirmative solution would generalize the curve selection lemma \cite[p.989]{bc}, \cite[\S3]{m}, \cite[Lem.18.3]{w}, is the following (see also \cite[\S4.C]{fgu}). Let $X$ be a connected topological space of certain type, let $x_1,\ldots,x_r\in X$ be finitely many points (control points) and $0=:t_1<\cdots<t_r:=1$ be finitely many values (control times).

\begin{prob}[Curve selection lemma through finitely many points]\label{1}
{\em Is there a (continuous) path $\alpha:[0,1]\to X$ of `certain prefixed type' such that $\alpha(t_i)=x_i$ for $i=1,\ldots,r$?} 
\end{prob}

Suppose we already have a continuous path $\beta:[0,1]\to X$ such that $\beta(t_i)=x_i$ for $i=1,\ldots,r$, that $X$ is a metric space and fix $\veps>0$.

\begin{prob}[Approximation of curves through finitely many points]\label{2}
{\em Is there a (continuous) path $\alpha:[0,1]\to X$ of `certain prefixed type' such that $\alpha(t_i)=x_i$ for $i=1,\ldots,r$ and $\dist(\alpha(t),\beta(t))<\veps$ for each $t\in[0,1]$?}
\end{prob}

\subsection{Semialgebraic setting}\label{sst}
A subset $\Ss\subset\R^n$ is \em semialgebraic \em when it admits a description by a finite boolean combination of polynomial equalities and inequalities. The category of semialgebraic sets is closed under basic boolean operations, but also under usual topological operations: taking closures (denoted by $\cl(\cdot)$), interiors (denoted by $\Int(\cdot)$), connected components, etc. If $\Ss\subset\R^m$ and $\Tt\subset\R^n$ are semialgebraic sets, a map $f:\Ss\to\Tt$ is \em semialgebraic \em if its graph is a semialgebraic set. 

In the following smooth means $\Cont^\infty$. A map $f:U\to\R^m$ on an open semialgebraic set $U\subset\R^n$ is \em Nash \em if it is smooth and semialgebraic. Recall that Nash maps are analytic maps \cite[Prop.8.1.8]{bcr}. If $\Ss\subset\R^n$ is a semialgebraic set, a map $f:\Ss\to\R^m$ is {\em Nash} if there exist an open semialgebraic neighborhood $U\subset\R^n$ of $\Ss$ and a Nash extension $F:U\to\R^m$ of $f$ to $U$. Analogously, a Nash manifold is a semialgebraic subset $\Ss\subset\R^n$ that is a smooth submanifold of $\R^n$. As an application of \cite[Prop.8.1.8]{bcr} one deduces that Nash manifolds are analytic manifolds. Recall that open semialgebraic subsets of $\R^n$ admit by the Finiteness Theorem \cite[Th.2.7.2]{bcr} a description as a finite union of {\em basic open semialgebraic sets}, that is, semialgebraic sets of the type $\{f_i>0,\ldots,f_r>0\}$ where $f_i\in\R[\x]:=\R[\x_1,\ldots,\x_n]$. Along the article we will use typewriter symbols $\x,\y,\z,\t$ to denote variables or tuples of variables, whereas we use the symbols $x,y,z,t$ to denote values or points that we substitute in variables or tuples of variables $\x,\y,\z,\t$.

\subsection{State of the art for semialgebraic sets and Nash paths}
In this work we study Problems \ref{1} and \ref{2} when $X=\Ss\subset\R^n$ is a semialgebraic set and $\alpha:[0,1]\to\Ss$ is a Nash path. We prove results that involve a tight control of the behavior of the obtained Nash/polynomial path (Theorem \ref{smart} (polynomial case) and Main Theorems \ref{nashsmart} (Nash case) and \ref{plcase} (PL case)). Using these results, we deduce that a sufficient condition to solve Problems \ref{1} and \ref{2} is that $\Ss$ is connected by analytic paths. In fact, if the number of points $x_i$ is large enough and they are in general position, the connexion by analytic paths is a necessary condition. A `theoretical' (but not constructive) solution to Problem \ref{1} follows straightforwardly from \cite[Main Thm.1.4]{f1}, where we characterize the semialgebraic subsets of $\R^n$ of dimension $d$ that are images of $\R^d$ under a Nash map. Namely,

\begin{thm}[Nash images of affine spaces, {\cite[Main Thm.1.4]{f1}}]\label{ni}
Let $\Ss\subset\R^n$ be a semialgebraic set of dimension $d$. The following conditions are equivalent:
\begin{itemize}
\item[(i)] There exists a Nash map $f:\R^d\to\R^n$ such that $f(\R^d)=\Ss$.
\item[(ii)] $\Ss$ is connected by analytic paths.
\end{itemize}
\end{thm}

\subsubsection{Nash curve selection lemma through finitely many points.}
The announced `theoretical' (but not constructive) consequence of Theorem \ref{ni} is the following.

\begin{cor}[Nash curve selection lemma through finitely many points]\label{main}
Let $\Ss\subset\R^n$ be a semialgebraic set connected by analytic paths. Fix control points $x_1,\ldots,x_r\in\Ss$ and control values $0=:t_1<\cdots<t_r:=1$. Then there exists a Nash path $\alpha:[0,1]\to\Ss$ such that $\alpha(t_i)=x_i$ for $i=1,\ldots,r$.
\end{cor}
\begin{proof}
Let $f:\R^d\to\R^n$ be a Nash map such that $f(\R^d)=\Ss$ and let $z_1,\ldots,z_r\in\R^d$ be such that $f(z_i)=x_i$ for $i=1,\ldots,r$. Using for instance Lagrange's interpolation, we find a polynomial path $\beta:[0,1]\to\R^d$ (of degree $\leq r-1$) such that $\beta(t_i)=z_i$ for $i=1,\ldots,r$. Thus, $\alpha:=f\circ\beta:[0,1]\to\Ss$ is a Nash path that satisfies the required conditions.
\end{proof}
\begin{remark}[Classical curve selection lemma]
In \cite[\S9]{f1} it is proved that each semialgebraic set $\Ss\subset\R^n$ is the union of its connected components by analytic paths, which are finitely many semialgebraic sets $\Ss_1,\ldots,\Ss_r$. If $x\in\cl(\Ss)$, we may assume $x\in\cl(\Ss_1)$. Thus, $\Ss_1\cup\{x\}$ is again connected by analytic paths \cite[Main Thm.1.4 \& Lem.7.4]{f1} and by Corollary \ref{main} there exists a Nash path $\alpha:[0,1]\to\Ss_1\cup\{x\}\subset\Ss\cup\{x\}$ such that $\alpha(0)=x$ and $\alpha((0,1])\subset\Ss_1\subset\Ss$. Thus, Corollary \ref{main} provides the classical curve selection lemma as a straightforward consequence.\hfill$\sqbullet$
\end{remark}

The main results of this article provide a different proof of Corollary \ref{main} (Problem \ref{1}) with a more constructive flavor, which is not based on the existential use of \cite[Main Thm.1.4]{f1}. We will simultaneously face the problem of approximating some existing continuous semialgebraic path passing through the control points at the control times (Problem \ref{2}). As the reader can expect, the previous continuous semialgebraic path shall satisfy some additional restrictions.

\subsubsection{Polynomial curve selection lemma through finitely many points}
Let $\alpha:=(\alpha_1,\ldots,\alpha_n):[a,b]\to\R^n$ be a continuous semialgebraic path. We claim: {\em There exists a minimal finite set $\eta(\alpha)\subset[a,b]$ such that $\alpha|_{[a,b]\setminus\eta(\alpha)}$ is a Nash map}. 
\begin{proof}
Consider the continuous semialgebraic map $\beta_k:[a,b]\to\R^2,\ t\mapsto(t,\alpha_k(t))$ for $k=1,\ldots,n$. By \cite[Prop.2.9.10]{bcr} there exist finitely many points $t_1,\ldots,t_r\in[a,b]$ such that $M_{ki}:=\beta_k((t_i,t_{i+1}))$ is a Nash submanifold of $\R^2$ for $i=1,\ldots,r-1$ and $k=1,\ldots,n$. For each $p\in M_{ki}$ denote the tangent line to $M_{ki}$ at $p$ with $T_pM_{ki}$. Consider the projection $\pi_1:\R^2\to\R$ onto the first coordinate and let $R_{ki}:=\{p\in M_{ki}:\ \dim(\pi_1(T_pM_{ki}))=0\}$. We claim: {\em The semialgebraic set $R_{ki}$ is finite for each $i=1,\ldots,r-1$ and each $k=1,\ldots,n$}. 

If $p\in R_{ki}$ (for some $i=1,\ldots,r-1$ and $k=1,\ldots,n$), then $\alpha_k$ is not differentiable at $\pi_1(p)$. By \cite[Ch.7.Thm.(3.2) (II$_m$), p.115]{dries} the non-differentiability locus of $\alpha_k$ is a semialgebraic set of dimension $\leq0$, that is, it is a finite set. Consequently, $R_{ki}$ is a finite set, as claimed.

We conclude that $\eta(\alpha)\subset\{t_1,\ldots,t_r\}\cup\bigcup_{k=1}^n\bigcup_{i=1}^rR_{ik}$ is a finite set, as required.
\end{proof}

By \cite[Prop.8.1.12]{bcr} and after reparameterizing (locally at $a$ and $b$ if necessary) we may assume that $\alpha$ is analytic at the points $a,b$ and consequently that $\eta(\alpha)\subset(a,b)$. Let $\Ss_1,\Ss_2\subset\R^n$ be two Nash manifolds. A \em (Nash) bridge between $\Ss_1$ and $\Ss_2$ \em is the image $\Gamma$ of a Nash arc $\alpha:[-1,1]\to\R^n$ such that $\alpha([-1,0))\subset\Ss_1$ and $\alpha((0,1])\subset\Ss_2$. The point $\alpha(0)$ is called the \em base point \em of $\Gamma$. In case $\Ss_1,\Ss_2\subset\R^n$ are open semialgebraic sets and there exists a Nash bridge $\alpha:[-1,1]\to\R^n$ between $\Ss_1$ and $\Ss_2$, we can modify $\alpha$ to have a polynomial arc $\alpha:[-1,1]\to\R^n$ such that $\alpha([-1,0))\subset\Ss_1$ and $\alpha((0,1])\subset\Ss_2$ (see \cite[Lem.4.1]{fu}).

In \cite{fu} we study the images of the closed unit ball under polynomial maps. As a main tool, we prove there the following result \cite[Lem.3.1]{fu}, which is stronger than only a solution to Problems \ref{1} and \ref{2}. The main difficulty focuses on guaranteeing that the approximating polynomial paths have their images inside the chosen semialgebraic set. These types of problems of keeping the same target space after approximation are analyzed carefully in \cite{fgh1,fgh2}.

\begin{thm}[{Smart polynomial curve, \cite[Lem.3.1]{fu}}]\label{smart}
Let $\Ss_1,\ldots,\Ss_r\subset\R^n$ be connected open semialgebraic sets (non-necessarily pairwise different) and denote $\Ss:=\bigcup_{i=1}^r\Ss_i$. Pick control points $p_i\in\cl(\Ss_i)$ and assume there exists a polynomial bridge $\Gamma_i$ between $\Ss_i$ and $\Ss_{i+1}$. Denote the base point of $\Gamma_i$ with $q_i\in\cl(\Ss_i)\cap\cl(\Ss_{i+1})$. Fix control times $s_0:=0<t_1<\cdots<t_r<1=:s_r$ and $s_i\in(t_i,t_{i+1})$ for $i=1,\ldots,r-1$. Then there exists a polynomial path $\alpha:\R\to\R^n$ that satisfies: 
\begin{itemize}
\item[(i)] $\alpha([0,1])\subset\Ss\cup\{p_1,\ldots,p_r,q_1,\ldots,q_{r-1}\}$.
\item[(ii)] $\alpha(t_i)=p_i$ for $i=1,\ldots,r$.
\item[(iii)] $\alpha((t_i,s_i))\subset\Ss_i$, $\alpha((s_i,t_{i+1}))\subset\Ss_{i+1}$ and $\alpha(s_i)=q_i$ for $i=1,\ldots,r-1$.
\end{itemize}
In addition, if $\veps>0$ and $\beta:[0,1]\to\R^n$ is a continuous semialgebraic path such that $\eta(\beta)\subset(0,1)\setminus\{t_1,\ldots,t_r,s_1,\ldots,s_{r-1}\}$, $\beta(\eta(\beta))\subset\Ss$ and $\beta$ satisfies conditions \em (i)\em, \em (ii) \em and \em (iii) \em above, we may assume that $\|\alpha-\beta\|<\veps$.
\end{thm}
\begin{remark}
Contrary to what we have stated in Problems \ref{1} and \ref{2} above, here the control times are inside the interval $(0,1)$. This is done to simplify the proof (and it will happen again in Main Theorem \ref{nashsmart}), but it is not limiting. As we have commented, if $\beta:[0,1]\to\R^n$ is a continuous semialgebraic path, we can reparameterize $\beta$ locally at $0$ and $1$ in order to have $\eta(\beta)\subset(0,1)$. This means that we can analytically extend $\beta$ around $0$ and $1$ to an interval $[-\delta,1+\delta]$ for some $\delta>0$ and after rescaling (to work in the interval $[0,1]$), we may assume that the control times $t_i\in(0,1)$.\hfill$\sqbullet$
\end{remark}

\subsection{Main results}
The first part of Theorem \ref{smart} concerns Problem \ref{1}, whereas its second part concerns Problem \ref{2}. We cannot expect a general result (that is, without the assumption that the $\Ss_i$ are open semialgebraic subsets of $\R^n$) of similar nature involving polynomial paths instead of Nash paths. In general, semialgebraic sets do not contain rational paths. By \cite{c,v}, a generic complex hypersurface $Z$ of $\C\PP^n$ of degree $d\geq 2n-2$ for $n\geq 4$ and of degree $d\geq 2n-1$ for $n=2,3$ does not contain rational curves. If $\Ss\subset\R^n$ is a semialgebraic set whose Zariski closure $X$ in $\R\PP^n$ is a generic hypersurface of $\R\PP^n$ of high enough degree, then its Zariski closure $Z$ in $\C\PP^n$ does not contain rational curves, so $\Ss$ cannot contain rational paths. This means in particular (as general real algebraic sets are birational to real hypersurfaces) that general semialgebraic sets do not contain polynomial paths. 

\subsubsection{General case}
In this article we prove Main Theorem \ref{nashsmart} (Figure \ref{fig1}) and we provide a somehow constructive proof. This requires to improve some results \cite{dl,f,vo} concerning the convergence at compact sets of the derivatives of Bernstein's polynomials to the derivatives of the function $f:[0,1]\to\R$ we want to approximate, even if $f$ only admits derivatives on an open strict subset of the interval $[0,1]$ (Theorem \ref{cotasi}). More precisely, we need to estimate the derivatives of the Bernstein's polynomials of a continuous semialgebraic function on the closed interval $[0,1]$. Such function $f$ is analytic on $[0,1]\setminus{\mathfrak F}$, where ${\mathfrak F}$ is a finite subset of $[0,1]$. A possibility would be to smoothen $f$ until certain order $\ell$ around the points of ${\mathfrak F}$, but this requires to modify $f$ and supposes an increase on the complexity of the construction. To avoid this smoothening of $f$, we prove Theorem \ref{cotasi} to provide bounds about the convergence of the derivatives of the Bernstein's polynomials of $f$ on the compact subsets of $[0,1]\setminus{\mathfrak F}$. 

\begin{center}
\begin{figure}[ht]
\begin{tikzpicture}[scale=1.25]
\draw[dashed,line width=1.5pt] (0,0) -- (0,2) -- (2,2) -- (0,0);
\draw[dashed,line width=1.5pt] (2,2) -- (4,2) -- (4,0) -- (2,2);
\draw[dashed,line width=1.5pt] (5,1) circle (1cm);
\draw[dashed,line width=1.5pt] (6,1) arc (270:360:1cm) arc (180:270:1cm) arc (90:180:1cm) arc (0:90:1cm);
\draw[dashed,line width=1.5pt] (8,1) -- (9,1) -- (9,2) -- (8,1);
\draw[dashed,line width=1.5pt] (9,1) -- (10,1) -- (10,0) -- (9,1);

\draw[fill=gray!100,opacity=0.5,draw=none] (0,0) -- (0,2) -- (2,2) -- (0,0);
\draw[fill=gray!100,opacity=0.5,draw=none] (2,2) -- (4,2) -- (4,0) -- (2,2);
\draw[fill=gray!100,opacity=0.5,draw=none] (5,1) circle (1cm);
\draw[fill=gray!100,opacity=0.5,draw=none] (6,1) arc (270:360:1cm) arc (180:270:1cm) arc (90:180:1cm) arc (0:90:1cm);
\draw[fill=gray!100,opacity=0.5,draw=none] (8,1) -- (9,1) -- (9,2) -- (8,1);
\draw[fill=gray!100,opacity=0.5,draw=none] (9,1) -- (10,1) -- (10,0) -- (9,1);

\draw[color=red,line width=1.5pt] (1.3,1.875) arc (110:70:2cm);
\draw[color=red,line width=1.5pt] (3.5,1) -- (4.5,1);
\draw[color=red,line width=1.5pt] (5.5,1) -- (6.5,1);
\draw[color=red,line width=1.5pt] (7.25,1.125) arc (250:290:2cm);

\draw[color=red,line width=1.5pt] (8.5,1.25) -- (9.5,0.75);

\draw (2,2) node{$\bullet$};
\draw (4,1) node{$\bullet$};
\draw (6,1) node{$\bullet$};
\draw (8,1) node{$\bullet$};
\draw (9,1) node{$\bullet$};

\draw (2,2.3) node{$q_1$};
\draw (3.6,1.3) node{$q_2$};
\draw (6.2,1.3) node{$q_3$};
\draw (8,1.3) node{$q_4$};
\draw (9,0.7) node{$q_5$};

\draw (0,1) node{$\bullet$};
\draw (4,2) node{$\bullet$};
\draw (5,0) node{$\bullet$};
\draw (7,0) node{$\bullet$};
\draw (9,2) node{$\bullet$};
\draw (10,0) node{$\bullet$};

\draw (-0.3,1) node{$p_1$};
\draw (4.3,2.3) node{$p_2$};
\draw (5,-0.3) node{$p_3$};
\draw (7,-0.3) node{$p_4$};
\draw (9,2.3) node{$p_5$};
\draw (10,-0.3) node{$p_6$};


\draw[dashed] (0,-0.75) -- (10,-0.75);

\end{tikzpicture}

\vspace*{1mm}
\begin{tikzpicture}[scale=1.25]
\draw[dashed,line width=1.5pt] (0,0) -- (0,2) -- (2,2) -- (0,0);
\draw[dashed,line width=1.5pt] (2,2) -- (4,2) -- (4,0) -- (2,2);
\draw[dashed,line width=1.5pt] (5,1) circle (1cm);
\draw[dashed,line width=1.5pt] (6,1) arc (270:360:1cm) arc (180:270:1cm) arc (90:180:1cm) arc (0:90:1cm);
\draw[dashed,line width=1.5pt] (8,1) -- (9,1) -- (9,2) -- (8,1);
\draw[dashed,line width=1.5pt] (9,1) -- (10,1) -- (10,0) -- (9,1);

\draw[fill=gray!100,opacity=0.5,draw=none] (0,0) -- (0,2) -- (2,2) -- (0,0);
\draw[fill=gray!100,opacity=0.5,draw=none] (2,2) -- (4,2) -- (4,0) -- (2,2);
\draw[fill=gray!100,opacity=0.5,draw=none] (5,1) circle (1cm);
\draw[fill=gray!100,opacity=0.5,draw=none] (6,1) arc (270:360:1cm) arc (180:270:1cm) arc (90:180:1cm) arc (0:90:1cm);
\draw[fill=gray!100,opacity=0.5,draw=none] (8,1) -- (9,1) -- (9,2) -- (8,1);
\draw[fill=gray!100,opacity=0.5,draw=none] (9,1) -- (10,1) -- (10,0) -- (9,1);


\draw[color=blue,line width=1.5pt] (0,1) .. controls (0,2) and (1.25,2) .. (2,2) .. controls (2.5,2) and (3.5,1.5) .. (4,2) .. controls (3.75,1.5) and (3.5,1) .. (4,1) .. controls (4,1) and (5,1) .. (5,0) .. controls (5,1) and (6,1) .. (6,1) .. controls (6,1) and (7,1.1) .. (7,0) .. controls (7,1.1) and (8,1) .. (8,1) .. controls (8,1) and (9,1.25) .. (9,2) .. controls (9,1.5) and (8.5,1) .. (9,1) .. controls (9,1) and (10,1) .. (10,0);

\draw (2,2) node{$\bullet$};
\draw (4,1) node{$\bullet$};
\draw (6,1) node{$\bullet$};
\draw (8,1) node{$\bullet$};
\draw (9,1) node{$\bullet$};

\draw (2,2.3) node{$q_1$};
\draw (3.8,0.8) node{$q_2$};
\draw (6.2,1.3) node{$q_3$};
\draw (8,1.3) node{$q_4$};
\draw (9,0.7) node{$q_5$};

\draw (0,1) node{$\bullet$};
\draw (4,2) node{$\bullet$};
\draw (5,0) node{$\bullet$};
\draw (7,0) node{$\bullet$};
\draw (9,2) node{$\bullet$};
\draw (10,0) node{$\bullet$};

\draw (-0.3,1) node{$p_1$};
\draw (4.3,2.3) node{$p_2$};
\draw (5,-0.3) node{$p_3$};
\draw (7,-0.3) node{$p_4$};
\draw (9,2.3) node{$p_5$};
\draw (10,-0.3) node{$p_6$};

\end{tikzpicture}
\caption{Statement of Main Theorem \ref{nashsmart}.\label{fig1}}
\end{figure}
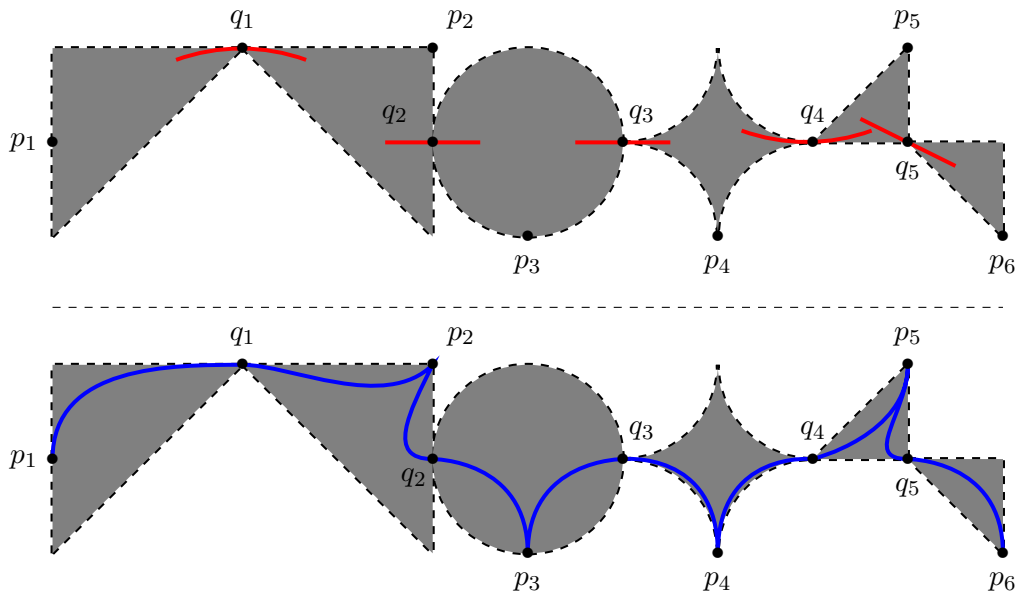
\end{center}

\begin{mthm}[Smart Nash curve]\label{nashsmart}
Let $\Ss\subset\R^n$ be a pure dimensional semialgebraic set and $\Ss_1,\ldots,\Ss_r$ open connected semialgebraic subsets of $\Reg(\Ss)$ (non-necessarily pairwise different). Pick control points $p_i\in\cl(\Ss_i)$ for $i=1,\ldots,r$ and assume there exists a Nash bridge $\Gamma_i$ between $\Ss_i$ and $\Ss_{i+1}$ for $i=1,\ldots,r-1$. Denote the base point of $\Gamma_i$ with $q_i\in\cl(\Ss_i)\cap\cl(\Ss_{i+1})$. Fix control times $s_0:=0<t_1<\cdots<t_r<1=:s_r$ and $s_i\in(t_i,t_{i+1})$ for $i=1,\ldots,r-1$. Then there exists a Nash path $\alpha:[0,1]\to\R^n$ that satisfies: 
\begin{itemize}
\item[(i)] $\alpha([0,1])\subset\bigcup_{i=1}^r\Ss_i\cup\{p_1,\ldots,p_r,q_1,\ldots,q_{r-1}\}$.
\item[(ii)] $\alpha(t_i)=p_i$ for $i=1,\ldots,r$.
\item[(iii)] $\alpha((t_i,s_i))\subset\Ss_i$, $\alpha((s_i,t_{i+1}))\subset\Ss_{i+1}$ and $\alpha(s_i)=q_i$ for $i=1,\ldots,r-1$.
\end{itemize}
In addition, if $\veps>0$ and $\beta:[0,1]\to\R^n$ is a continuous semialgebraic path such that $\eta(\beta)\subset(0,1)\setminus\{t_1,\ldots,t_r,s_1,\ldots,s_{r-1}\}$, $\beta(\eta(\beta))\subset\bigcup_{i=1}^r\Ss_i$ and $\beta$ satisfies conditions \em (i)\em, \em (ii) \em and \em (iii) \em above, we may assume that $\|\alpha-\beta\|<\veps$.
\end{mthm}

As a consequence of Main Theorem \ref{nashsmart}, we provide in \S\ref{alternative} an alternative proof of Corollary \ref{main}. Following the proof of Main Theorem \ref{nashsmart} the reader realizes that, up to resolution of singularities and the use of a Nash tubular neighborhood, the proof of Main Theorem \ref{nashsmart} is reduced to show Theorem \ref{smart}, which is constructive up to polynomial approximation (controlling the behavior of a large enough number of derivatives) of continuous semialgebraic paths (which are analytic outside a finite set) combined with Hermite's interpolation. In the proof of Theorem \ref{smart} provided in \cite{fu} we smoothen corners of continuous semialgebraic paths, whereas in this article we use the announced Theorem \ref{cotasi}. In \cite{cf1,cf2} we make an extended use of Main Theorem \ref{nashsmart} to represent compact semialgebraic sets connected by analytic paths as images of closed unit balls under Nash maps. 

\subsubsection{Piecewise linear semialgebraic sets.}
In Section \ref{s4} we simplify the proof of Main Theorem \ref{nashsmart} for piecewise linear semialgebraic sets (PL case), that is, when the involved semialgebraic sets are the interiors of convex polyhedra of dimension $n$. In this case, we approximate the polygonal path that connects the control points (and base points of the polynomial bridges) at the prescribed control times. In order to get better bounds for the degrees of the polynomial paths provided by Main Theorem \ref{nashsmart} (see \S\ref{bound}): (1) we state a (polynomial) curve selection lemma for convex polyhedra that involves degree $3$ cuspidal curves (Lemma \ref{cuspidal}), and (2) we prove that the simplest polynomial paths that connect two convex polyhedra (whose union is connected by analytic paths) are moment curves (Theorem \ref{mc}).

\begin{mthm}[PL case]\label{plcase}
Let $\Ss_1,\ldots,\Ss_r\subset\R^n$ be the interiors of $n$-dimensional convex polyhedra (non-necessarily pairwise different) and denote $\Ss:=\bigcup_{i=1}^r\Ss_i$. Pick control points $p_i\in\cl(\Ss_i)$ for $i=1,\ldots,r$ and suppose that there exists a Nash bridge $\Gamma_i$ between $\Ss_i$ and $\Ss_{i+1}$ for $i=1,\ldots,r-1$. Denote the base point of $\Gamma_i$ with $q_i\in\cl(\Ss_i)\cap\cl(\Ss_{i+1})$. Fix control times $s_0:=0<t_1<\cdots<t_r<1=:s_{r}$ and $s_i\in(t_i,t_{i+1})$ for $i=1,\ldots,r-1$. Then there exists a polynomial map $\alpha:\R\to\R^n$ that satisfies:
\begin{itemize}
\item[(i)] $\alpha([0,1])\subset\Ss\cup\{p_1,\ldots,p_r,q_1,\ldots,q_{r-1}\}$.
\item[(ii)] $\alpha(t_i)=p_i$ for $i=1,\ldots,r$.
\item[(iii)] $\alpha((t_i,s_i))\subset\Ss_i$, $\alpha((s_i,t_{i+1}))\subset\Ss_{i+1}$ and $\alpha(s_i)=q_i$ for $i=1,\ldots,r-1$.
\item[(iv)] The restriction $\alpha|_{[t_1,t_r]}$ is as close as wanted to the piecewise linear parameterization $\beta:[t_1,t_r]\to\cl(\Ss)$ of the polygonal path that connects the points $p_1,q_1,p_2,\cdots,p_{r-1},q_{r-1},p_r$, passes through these points at the control times $t_1<s_1<t_2<\cdots<t_{r-1}<s_{r-1}<t_r$ and satisfies $\eta(\beta)\subset\{t_2,\ldots,t_{r-1},s_1,\ldots,s_{r-1}\}$.
\end{itemize}
\end{mthm}

\subsubsection{Graph}\label{graph}
Let $\Ss\subset\R^n$ be a $d$-dimensional semialgebraic set and let $\Ss_1,\ldots,\Ss_r\subset\R^n$ be connected open semialgebraic subsets of $\Reg(\Ss)$ of dimension $d$. Observe that $\Ss_i$ is a Nash manifold for $i=1,\ldots,r$. Assume $\bigcup_{i=1}^r\Ss_i\subset\Ss\subset\cl(\bigcup_{i=1}^r\Ss_i)$ and $\Ss$ is connected by analytic paths. We construct a graph $\Lambda$ to approach (Nash) logistic problems in $\Ss$ in the following way. The vertices of the graph are $\Ss_1,\ldots,\Ss_r$ and we have an edge between the vertices $\Ss_i$ and $\Ss_j$ if there exists a Nash bridge inside $\Ss$ between the Nash manifolds $\Ss_i$ and $\Ss_j$. By the following lemma (see also \cite[Main Thm.1.4 \& Cor.7.6]{f1} and \cite[Lem.4.2]{fu}) the previous graph is connected (because $\Ss$ is connected by analytic paths) and one can approach with the help of Main Theorem \ref{nashsmart} (Nash) logistic problems between the `regions' $\Ss_k$ using the existing Nash bridges between them (see below the alternative proof of Corollary \ref{main} as an example of application of this strategy).

\begin{lem}\label{graphconnected}
The graph $\Lambda$ is connected.
\end{lem}
\begin{proof}
It is enough: {\em to reorder recursively the indices $i=1,\ldots,r$ in such a way that for each $i=2,\ldots,r$ there exists a Nash bridge inside $\Ss$ between $\Ss_i$ and some $\Ss_j$ with $1\leq j\leq i-1$}. 

Suppose we have chosen $\Ss_1,\ldots,\Ss_k$ satisfying the previous conditions and let us choose a suitable $\Ss_{k+1}$. Denote $\Tt_1:=\bigcup_{j=1}^k\Ss_j$ and $\Tt_2:=\bigcup_{\ell=k+1}^r\Ss_\ell$. If $\Tt_1\cap\Tt_2\neq\varnothing$, there exists an index $\ell\in\{k+1,\ldots,r\}$ such that $\Ss_\ell\cap\Ss_j\neq\varnothing$ for some $j\in\{1,\ldots,k\}$. We interchange $k+1$ and $\ell$ in order to have $\Ss_{k+1}=\Ss_\ell$. Pick a point $x\in\Ss_j\cap\Ss_{k+1}$ and any Nash arc $\alpha:[-1,1]\to\Ss_j\cap\Ss_{k+1}$ such that $\alpha(0)=x$. Observe that $\alpha$ provides a Nash bridge inside $\Ss$ between some $\Ss_j$ with $1\leq j\leq k$ and $\Ss_{k+1}$. 

Assume next $\Tt_1$ and $\Tt_2$ are disjoint (open semialgebraic subsets of $\Reg(\Ss)$). Let $Y$ be the Zariski closure of $(\cl(\Tt_1)\setminus\Tt_1)\cup(\cl(\Tt_2)\setminus\Tt_2)$, which by \cite[Prop.2.8.13]{bcr} has dimension $\leq d-1$. We have 
\begin{multline}\label{finitey}
\Ss\subset\cl(\Ss)=\cl\Big(\bigcup_{i=1}^r\Ss_i\Big)=\cl(\Tt_1)\cup\cl(\Tt_2)\\
=\Tt_1\cup\Tt_2\cup(\cl(\Tt_1)\setminus\Tt_1)\cup(\cl(\Tt_2)\setminus\Tt_2)\subset\Tt_1\cup\Tt_2\cup Y.
\end{multline}

As $\dim(\Tt_1)=\dim(\Tt_2)=d$, the differences $\Tt_1\setminus Y$ and $\Tt_2\setminus Y$ are non-empty semialgebraic sets. Pick points $x\in\Tt_1\setminus Y=\Tt_1\setminus(\cl(\Tt_2)\cup Y)$ and $y\in\Tt_2\setminus Y=\Tt_2\setminus(\cl(\Tt_1)\cup Y)$ (recall that $\Tt_1$ and $\Tt_2$ are disjoint and $Y$ is the Zariski closure of $(\cl(\Tt_1)\setminus\Tt_1)\cup(\cl(\Tt_2)\setminus\Tt_2)$). As $\Ss$ is connected by analytic paths, there exists a Nash path $\alpha:[0,1]\to\Ss$ such that $\alpha(0)=x$ and $\alpha(1)=y$. As $\alpha^{-1}(Y)$ is both a closed subset of $[0,1]$ and the zero set in $(0,1)$ of a Nash function defined on $[0,1]$ (because $x,y\not\in Y$), we deduce by the Identity Principle that $\alpha^{-1}(Y)$ has dimension $0$, so it is a finite subset of $[0,1]$. By \eqref{finitey} we deduce $[0,1]\setminus(\alpha^{-1}(\Tt_1\setminus Y)\cup\alpha^{-1}(\Tt_2\setminus Y))\subset\alpha^{-1}(Y)$ is a finite set. As $0\in\alpha^{-1}(\Tt_1\setminus Y)=\alpha^{-1}(\Tt_1)\setminus\alpha^{-1}(\cl(\Tt_2)\cup Y)$ and $1\in\alpha^{-1}(\Tt_2\setminus Y)$, we have $0<t_0:=\inf(\alpha^{-1}(\Tt_2))<1$. Observe that $[0,t_0)\setminus\alpha^{-1}(Y)\subset\alpha^{-1}(\Tt_1\setminus Y)$ and $t_0\in\cl(\alpha^{-1}(\Tt_2))$. As $\alpha^{-1}(Y)$ is a finite set and $\alpha^{-1}(\Tt_2)$ is a non-empty open semialgebraic subset of $[0,1]$, there exists a small enough $\veps>0$ such that $\alpha([t_0-\veps,t_0))\subset\Tt_1=\bigcup_{j=1}^k\Ss_j$ and $\alpha((t_0,t_0+\veps])\subset\Tt_2=\bigcup_{\ell=k+1}^r\Ss_\ell$. Shrinking $\veps>0$ if necessary, we may assume $\alpha([t_0-\veps,t_0))\subset\Ss_j$ for some $1\leq j\leq k$ and $\alpha((t_0,t_0+\veps])\subset\Ss_{\ell}$ for some $k+1\leq \ell\leq r$. As $\alpha$ is a (non-constant) Nash path, we may assume (shrinking $\veps>0$ again if necessary) by semialgebraic triviality \cite[Thm.9.3.2]{bcr} that the restrictions $\alpha|_{[t_0-\veps,t_0)}$ and $\alpha|_{(t_0,t_0+\veps]}$ are injective. As $\Ss_j\cap\Ss_\ell=\varnothing$, we deduce $\alpha|_{[t_0-\veps,t_0+\veps]}$ is a Nash arc. We interchange $k+1$ and $\ell$ in order to have $\Ss_{k+1}=\Ss_{\ell}$. Thus, there exists a Nash bridge inside $\Ss$ between $\Ss_{k+1}$ and some $\Ss_j$ with $1\leq j\leq k$, as required.
\end{proof}

In case $\Ss_1,\ldots,\Ss_r\subset\R^n$ are open semialgebraic subsets of $\R^n$, we can study the previous problems from the polynomial point of view. Using Bernstein's polynomials (Theorem \ref{cotasi}), we can estimate the degree of the constructed polynomial paths, especially if each $\Ss_k$ is in addition the interior of an $n$-dimensional convex polyhedron (proof of Main Theorem \ref{plcase} in Section \ref{s4} and \S\ref{bound}). This also allows to estimate in Remark \ref{plcasec} the degree of the polynomial maps that appear in \cite[Thm.1.2 \& Thm.1.3]{fu} to represent compact semialgebraic sets that are connected by analytic paths as the image of closed unit balls under polynomials maps.

\subsubsection{Alternative proof of Corollary {\em\ref{main}}}\label{alternative}
Let $\Tt_1,\ldots,\Tt_s$ be the connected components of $\Reg(\Ss)$, which are connected Nash manifolds. Let $\Uu_k\in\{\Tt_1,\ldots,\Tt_s\}$ be such that $x_i\in\cl(\Uu_{i})$ for $i=1,\ldots,r$. Consider the graph $\Lambda$ whose vertices are the connected Nash manifolds $\Tt_i$ and such that there exists an edge between a pair of vertices $\Tt_i$ and $\Tt_j$ if and only if there exists a Nash bridge $\Gamma$ inside $\Ss$ between the Nash manifolds $\Tt_i$ and $\Tt_j$. As $\Ss$ is connected by analytic paths, the graph $\Lambda$ is by Lemma \ref{graphconnected} connected. Thus, given the sequence of vertices $\Uu_{1},\ldots,\Uu_{r}$, there exists a path $P$ in the graph $\Lambda$ that passes through $\Uu_{1},\ldots,\Uu_{r}$ in this order. We collected all the ordered vertices of $P$ (including repetitions if needed) and denote them by $\Ss_1,\ldots,\Ss_\ell$ in such a way that there exists a Nash path $\Gamma_i$ between $\Ss_i$ and $\Ss_{i+1}$ for $i=1,\ldots,\ell-1$. In addition, there exist indices $1=:j_1<\ldots<j_r:=\ell$ such that $\Ss_{j_k}=\Uu_k$ for $k=1,\ldots,r$. For each $i\in\{1,\ldots,\ell\}\setminus\{j_1,\ldots,j_r\}$ we pick a point $p_i\in\Ss_i$ and denote $p_{j_k}:=x_k$ for $k=1,\ldots,r$. Denote the base point of $\Gamma_i$ with $q_i\in\cl(\Ss_i)\cap\cl(\Ss_{i+1})\subset\Ss$ for $i=1,\ldots,\ell-1$. Take times $0=:w_1<\ldots<w_\ell:=1$ such that $w_{j_k}=t_k$ for $k=1,\ldots,r$, $s_i\in(w_i,w_{i+1})$ for $i=1,\ldots,\ell-1$, $s_0<0$ and $s_{\ell}>1$. By Main Theorem \ref{nashsmart} there exists a Nash path $\alpha:[s_0,s_\ell]\to\R^n$ that satisfies: 
\begin{itemize}
\item[(i)] $\alpha([s_0,s_\ell])\subset\bigcup_{i=1}^r\Ss_i\cup\{p_1,\ldots,p_r,q_1,\ldots,q_{r-1}\}\subset\Ss$.
\item[(ii)] $\alpha(w_i)=p_i$ for $i=1,\ldots,r$.
\item[(iii)] $\alpha((w_i,s_i))\subset\Ss_i$, $\alpha((s_i,w_{i+1}))\subset\Ss_{i+1}$ and $\alpha(s_i)=q_i$ for $i=1,\ldots,r-1$.
\end{itemize}
Consequently, $\alpha|_{[0,1]}:[0,1]\to\Ss$ is a Nash path such that $\alpha(t_k)=x_k$ for $k=1,\ldots,r$, as required.
\qed

\subsection{Structure of the article}
The article is organized as follows. In Section \ref{s2} we present some preliminary concepts and tools. We would like to mention some results concerning Stone-Weierstrass' polynomial approximation using Bernstein's polynomials (Theorem \ref{cotasi}, that follows the ideas developed in \cite{f} and whose proof is postponed until Section \ref{s5}) and some of its main consequences (Lemmas \ref{swdp} and \ref{clue}). In Section \ref{s3} we prove the main result (Main Theorem \ref{nashsmart}), whereas in Section \ref{s4} we estimate the degree of the polynomial paths provided by Theorem \ref{smart} when the involved semialgebraic sets are piecewise linear (Main Theorem \ref{plcase}). Consequently, one can provide bounds for the degrees of the polynomial maps that appear in \cite[Thm.1.3 \& Thm.1.4]{fu} (see Remark \ref{plcasec}). We postpone some of the technicalities of the proof of Main Theorem \ref{nashsmart} until Appendix \ref{A} in order to make its proof more discurse and intuitive.

\subsection{Acknowledgements}
The author is deeply indebted with Prof. Safey el Din and Dr. Ueno for very helpful and enlightening comments and inspiring talks during the preparation of this work. The author is very grateful to S. Schramm for a careful reading of the final version and for the suggestions to refine its redaction. The author also thanks the anonymous referees for very valuable suggestions to correct several inaccuracies and incomplete arguments. These suggestions have notably improved and made clearer and precise the final version of this article. 

\section{Basic concepts and preliminary results}\label{s2}

In this section we recall and present some preliminary concepts and results that will be the key to prove Main Theorem \ref{nashsmart}. 

\subsection{Regular and singular points of a semialgebraic set.}
Recall that the set of regular points of a semialgebraic set $\Ss\subset\R^n$ is defined as follows. Let $X$ be the Zariski closure of $\Ss$ in $\R^n$ and $\widetilde{X}$ the complexification of $X$, that is, the smallest complex algebraic subset of $\C^n$ that contains $X$. The set $\Sing(\widetilde{X})$ of singular points of $\widetilde{X}$ corresponds to the collection of those points of $\widetilde{X}$ that do not admit a neighborhood diffeomorphic to a complex manifold. Define $\Reg(X):=X\setminus\Sing(\widetilde{X})$ and let $\Reg(\Ss)$ be the interior of $\Ss\setminus\Sing(\widetilde{X})$ in $\Reg(X)$. Observe that $\Reg(\Ss)$ is a finite union of disjoint Nash manifolds maybe of different dimensions. We refer the reader to \cite[\S2.A]{f1} for further details concerning the set of regular points of a semialgebraic set. 

\subsection{Hironaka's desingularization}
A rational map $f:=(f_1,\ldots,f_n):Z\to\R^n$ on an algebraic set $Z\subset\R^m$ is \em regular \em if its components are quotients of polynomials $f_k:=\frac{g_k}{h_k}$ such that $Z\cap\{h_k=0\}=\varnothing$. Hironaka's desingularization results \cite{hi} are powerful tools and we recall here the one we need.

\begin{thm}[Desingularization]\label{hi1}
Let $X\subset\R^n$ be an algebraic set. Then there exist a non-singular algebraic set $X'\subset\R^m$ and a proper regular map $f:X'\to X$ such that
$$
f|_{X'\setminus f^{-1}(\Sing(X))}:X'\setminus f^{-1}(\Sing(X))\rightarrow X\setminus\Sing(X)
$$
is a diffeomorphism whose inverse map is also regular.
\end{thm}

\begin{remark}
If $X$ is pure dimensional, $X\setminus\Sing X$ is dense in $X$. As $f$ is proper, it is surjective.\hfill$\sqbullet$
\end{remark}

\subsection{Topology of spaces of continuous functions}
Let $[a,b]\subset\R$ be a compact interval and $\Omega\subset[a,b]$ an open set. For each $\ell\geq1$ consider the space $\Cont^\ell_\Omega([a,b],\R)$ of continuous functions on $[a,b]$ that are $\Cont^\ell$ on $\Omega$. We endow $\Cont^\ell_\Omega([a,b],\R)$ with the $\Cont^\ell_\Omega$ topology that has as basis of open neighborhoods of $g\in\Cont^\ell_\Omega([a,b],\R)$ the family of sets of the type:
$$
{\mathcal U}^\ell_{g,K,\veps}:=\{f\in\Cont_\Omega^\ell([a,b],\R):\ \|f-g\|_{[a,b]}<\veps,\ \|f^{(k)}-g^{(k)}\|_{K}<\veps:\ k=1,\ldots,\ell\}
$$
where $K\subset\Omega$ is a compact set, $\veps>0$ and $\|h\|_T:=\max\{h(x):\ x\in T\}$ for each compact subset $T\subset[a,b]$. Sometimes we will omit the subindex $T$ when it is clear from the context. If $\Omega=[a,b]$, the previous topology is the usual $\Cont^\ell$ topology of $\Cont^\ell([a,b])$. Observe that $\Cont_\Omega^\ell([a,b],\R^n)=\Cont_\Omega^\ell([a,b],\R)\times\cdots\times\Cont_\Omega^\ell([a,b],\R)$ and we consider the product topology in this space. In particular if $f:=(f_1,\ldots,f_n)\in\Cont_\Omega^\ell([a,b],\R^n)$ and $T\subset[a,b]$ is a compact set, we denote to lighten notation $\|f\|_T:=\|\sqrt{f_1^2+\cdots+f_n^2}\|_T=\max\{\sqrt{f_1^2(x)+\cdots+f_n^2(x)}:\ x\in T\}$. If $X\subset[a,b]$, one defines analogously the $\Cont^{\ell}_{\Omega\cap X}$-topology of the space $\Cont^\ell_{\Omega\cap X}(X,\R)$. 

The following result follows from \cite[\S2.5. Ex.10, pp. 64-65]{hir3} using standard arguments.
\begin{lem}\label{cont}
Let $U\subset\R^n$ be an open set and $\varphi:U\to\R^m$ a $\Cont^k$ map for some $0\leq k\leq\ell$. Consider the map $\varphi^*:\Cont^\ell_{\Omega}([a,b],U)\to\Cont^k_{\Omega}([a,b],\R^m),\ f\mapsto\varphi\circ f$, where both spaces are endowed with their $\Cont^k_\Omega$-topologies. Then $\varphi^*$ is continuous. 
\end{lem}

In addition, one has the following.
\begin{lem}\label{cont2}
Let $X\subset[a,b]$ and consider the restriction map
$$
\rho:\Cont_\Omega^\ell([a,b],\R^n)\to\Cont_{\Omega\cap X}^\ell(X,\R^n),\ f\mapsto f|_{X},
$$ 
where the spaces are endowed with their respective $\Cont^\ell_{\Omega}$ and $\Cont^\ell_{\Omega\cap X}$ topologies. Then $\rho$ is continuous and if in addition $X\subset[a,b]$ is closed, then $\rho$ is surjective. 
\end{lem}

\subsection{Stone-Weierstrass' approximation and Bernstein's polynomials}
The proof of Theorem \ref{smart} provided in \cite{fu} involves Stone-Weierstrass' approximation (controlling the behavior of a large enough number of derivatives). We want to analyze the crucial role that Stone-Weierstrass' approximation plays. There are many constructive results in this direction and we refer the reader to \cite[Ch.7.\S.2]{dl} where estimations of the approximation errors are available. In this article we will use \em Bernstein's polynomials\em, which provided a pioneer constructive proof of Stone-Weierstrass' approximation theorem \cite{b}. We suggest the reader \cite[Ch.10]{dl} and \cite[\S1]{l} for further references. Although Bernstein's polynomials converge slowly to the approximated function, they have shape preserving properties \cite[Ch.10.Thm.3.3]{dl} and a `good local behavior' \cite[Ch.10.(3.3)]{dl}: \em If two continuous functions coincide on a subinterval of their common domain, their Bernstein's polynomials of high degree are very similar in (the compact subsets of) such subinterval\em, even when we compare their high order derivatives (Lemma \ref{comparison}). 

We recall some properties of the celebrated Bernstein's polynomials and we present some improvements in this work to fit our requirements (Theorem \ref{cotasi}). The \em Bernstein's approximation polynomial \em (of degree $\nu$) of a real function $f:[0,1]\to\R$ is
$$
B_\nu(f):=\sum_{k=0}^\nu f\Big(\frac{k}{\nu}\Big)B_{k,\nu}(\x),\quad\text{where $B_{k,\nu}(\x):=\binom{\nu}{k}\x^k(1-\x)^{\nu-k}\quad\text{for $k=0,\ldots,\nu$}.$}
$$

\subsubsection{Basic properties of Bernstein's polynomials}\label{bpbp}
Each Bernstein's basis polynomial $B_{k,\nu}(\x)$ of degree $\nu$ is strictly positive on the interval $(0,1)$. In fact, for each $x\in(0,1)$ the values $\{B_{k,\nu}(x)\}_{k=0}^\nu$ constitute the probability mass function of the binomial distribution $\Bb(\nu,x)$. This means (see \cite[p. 6]{l}):
\begin{itemize}
\item[$\bullet$] $\sum_{k=0}^\nu B_{k,\nu}(x)=1$,
\item[$\bullet$] $\sum_{k=0}^\nu k B_{k,\nu}(x)=\nu x$ as it is the mean of $\Bb(\nu,x)$,
\item[$\bullet$] $\sum_{k=0}^\nu (\nu x-k)^2B_{k,\nu}(x)=\nu x(1-x)$ as it is the variance of $\Bb(\nu,x)$.
\end{itemize}
In addition,
\begin{itemize}
\item[(i)] If $m\leq f\leq M$, then $m\leq B_\nu(f)\leq M$ (see \cite[(2) p. 5]{l}).
\item[(ii)] $B_{k,\nu}(\x)=(1-\x)B_{k,\nu-1}(\x)+\x B_{k-1,\nu-1}(\x)$ (which follows from the properties of binomial numbers). 
\item[(iii)] For each $h\in\R$ denote $\Delta_hf(x):=f(x+h)-f(x)$ and 
$$
\Delta_h^kf(x):=\Delta_h(\Delta_h^{k-1}f(x))=\sum_{j=0}^k(-1)^{k-j}\binom{k}{j}f(x+jh).
$$
If $B_\nu^{(k)}(f)$ denotes {\em the $k^{\text{th}}$ derivative of $B_\nu(f)$}, we have by \cite[\S1.4(2), p.12]{l}
$$
B_\nu^{(k)}(f)(\x)=\frac{\nu!}{(\nu-k)!}\sum_{i=0}^{\nu-k}\Delta_{\frac{1}{\nu}}^kf\Big(\frac{i}{\nu}\Big)B_{i,\nu-k}(\x)
$$
for $k=0,\ldots,\nu$.
\end{itemize}

\begin{remark}\label{ab}
If $f\in\Cont([a,b])$, write $f^*:[0,1]\to\R,\ t\mapsto f(a+t(b-a))$. Observe that $f=f^*(\frac{x-a}{b-a})$ and define $B_\nu^{\bullet}(f):=B_\nu(f^*)(\frac{x-a}{b-a})$ the \em Bernstein's polynomial of $f$ of degree $\nu$ for the interval $[a,b]$\em. The changes one makes in subsequent formulas for the interval $[0,1]$ to obtain the corresponding ones for the interval $[a,b]$ are of the following type: \em the polynomial $\x$ is changed by $\frac{\x-a}{b-a}$, so the polynomial $1-\x$ is changed by $\frac{b-\x}{b-a}$\em. For instance, the polynomial $\x(1-\x)\leadsto\frac{(\x-a)(b-\x)}{(b-a)^2}$ and the polynomial $(1-2\x)\leadsto\frac{((b-a)-2(\x-a))}{b-a}$.\hfill$\sqbullet$
\end{remark}

\subsubsection{Derivatives of Bernstein's polynomials}
One of the most remarkable properties of Bernstein's approximation, which is very useful for our constructions, is that derivatives $B_{\nu}^{(\ell)}(f)$ of $B_\nu(f)$ of each order $\ell$ converge to corresponding derivatives of $f$, see Lorentz \cite{l2}: \em If $f\in\Cont^\ell([0,1])$ for some $\ell\geq0$, then $\lim_{\nu\to\infty}B_{\nu}^{(\ell)}(f)=f^{(\ell)}$ uniformly on the interval $[0,1]$\em. This property can be viewed as a compensation for the `slow' convergence of $B_\nu(f)$ to $f$. If $\|\cdot\|_{[0,1]}$ denotes the maximum norm on $[0,1]$, the error bound
\begin{equation}\label{12}
|B_\nu(f)(x)-f(x)|\leq\frac{1}{2\nu}x(1-x)\|f''\|_{[0,1]}
\end{equation}
provided in \cite[Ch.10, (3.4), p.308]{dl} shows that the rate of convergence is at least $1/\nu$ for $f \in\Cont^2([0,1])$. On the other hand, Voronovskaya's asymptotic formula \cite{vo} (or \cite[\S.1.6.1]{l})
\begin{equation}\label{13}
\lim_{\nu\to\infty}\nu(B_{\nu}(f)(x)-f(x))=\frac{1}{2}x(1-x)f''(x)
\end{equation}
shows that for $x\in(0,1)$ with $f''(x)\neq0$, the asymptotic rate of convergence is precisely $1/\nu$. In \cite{f} it is shown that all derivatives of the operator $B_\nu$ converge at essentially the same rate by extending both the error bound \eqref{12} and Voronovskaya's formula \eqref{13}. The error bound is generalized in \cite{f} to the following:

\begin{thm}[Error bound, {\cite[Thm.1]{f}}]\label{cota1}
If $f\in\Cont^{\ell+2}([0,1])$ for some $\ell\ge0$, then
\begin{equation}\label{cota11}
|B_{\nu}^{(\ell)}(f)(x)-f^{(\ell)}(x)|\leq\frac{1}{2\nu}(\ell(\ell-1)\|f^{(\ell)}\|_{[0,1]}+\ell|1-2x|\|f^{(\ell+1)}\|_{[0,1]}+x(1-x)\|f^{(\ell+2)}\|_{[0,1]})
\end{equation}
for each $x\in[0,1]$.
\end{thm}
\begin{remark}\label{derf}
The reader can prove inductively that 
$$
\frac{d^\ell}{dx^\ell}(x(1-x)f''(x))=-\ell(\ell-1)f^{(\ell)}+\ell(1-2x)f^{(\ell+1)}+x(1-x)f^{(\ell+2)}
$$
is the $\ell^{\text{th}}$ derivative of $x(1-x)f''(x)$.\hfill$\sqbullet$
\end{remark}

In addition, Voronovskaya's formula \eqref{13} can be `differentiated' to determine the asymptotic behavior of the error for the high order derivatives of the Bernstein's polynomial:

\begin{thm}[Asymptotic behavior, {\cite[Thm.2]{f}}]\label{cota2}
If $f\in\Cont^{\ell+2}([0,1])$ for some $\ell\ge0$, then
$$
\lim_{\nu\to\infty}\nu(B_{\nu}^{(\ell)}(f)(x)-f^{(\ell)}(x))=\frac{1}{2}\frac{d^\ell}{dx^\ell}(x(1-x)f''(x))
$$
uniformly in the interval $[0,1]$.
\end{thm}

Thus, the $\ell$th derivative of $B_{\nu}(f)$ converges to $f^{(\ell)}(x)$ at the rate of $1/\nu$ when the $\ell$th derivative of $x(1-x)f''(x)$ is non-zero.

\subsubsection{Control of the derivatives of Bernstein's polynomials on compact subsets.}
In this paper we deal with continuous functions $f:[0,1]\to\R$ that are $\Cont^\ell$ only on an open subset $\Omega\subset[0,1]$ and we need to control the behavior of a large enough number of derivatives of the Bernstein's polynomials of $f$ on a compact subset $K\subset\Omega$. A first attempt is to smooth our function $f$ on $[0,1]\setminus\Omega$ and to make use of Lemma \ref{comparison} together with \cite[Thm.1 \& Thm.2]{f}. To avoid an increase of complexity when smoothing the initial data, we amalgamate in Theorem \ref{cotasi} the quoted results \cite[Thm.1 \& Thm.2]{f} and \cite[Ch.10.\S2-3]{dl} to approach the situation we need. Summarizing we provide a bound for the error of each derivative on the chosen compact set and show how the error behaves asymptotically. To make the presentation of the article more discursive, we postpone the proof of the following result until Section \ref{s5}.

\begin{thm}[Convergence of derivatives on compact subsets]\label{cotasi}
Let $f:[0,1]\to\R$ be a continuous function that is $\Cont^{\ell+4}$ on an open subset $\Omega\subset(0,1)$ for some $\ell\geq0$. Let $K\subset\Omega$ be a compact set. Then there exists a constant $C_{f,K,\ell}>0$ such that
\begin{multline*}
|B_{\nu}(f)^{(\ell)}(x)-f^{(\ell)}(x)|\leq\frac{1}{2\nu}\Big(\ell(\ell-1)\sum_{k=\ell}^{\ell+3}\frac{\|f^{(k)}\|_K}{(k-\ell)!}\\
+\ell|1-2x|\sum_{k=\ell+1}^{\ell+3}\frac{\|f^{(k)}\|_K}{(k-\ell-1)!}+x(1-x)\sum_{k=\ell+2}^{\ell+3}\frac{\|f^{(k)}\|_K}{(k-\ell-2)!}\Big)+\frac{C_{f,K,\ell}}{\nu^2}
\end{multline*}
for each $x\in K$ (error bound). In addition, for each $\veps>0$ there exists a constant $C^*_{f,K,\ell,\veps}>0$ such that
$$
\Big|\nu(B_{\nu}(f)^{(\ell)}(x)-f^{(\ell)}(x))-\frac{1}{2}\frac{d^\ell}{d\x^\ell}(x(1-x)f''(x))\Big|<\veps+\frac{C^*_{f,K,\ell,\veps}}{\nu}
$$
for each $x\in K$ and $\nu>\ell$ (control of the asymptotic behavior).
\end{thm}

\subsection{Polynomial approximation combined with interpolation}

We adapt \cite{bo} to prove the following result that combines Bernstein's polynomial approximation (controlling the behavior of a large enough number of derivatives on a compact subset) with interpolation on a finite set. We include full details for the sake of completeness.

\begin{lem}\label{swdp}
Let $[a,b]\subset\R$ and $\Omega\subset[a,b]$ an open set. Let $a<t_1<\cdots<t_r<b$ be real numbers such that each $t_i\in\Omega$ and $f:[a,b]\to\R$ a $\Cont^{\ell+4}_{\Omega}$-function for some $\ell\geq0$. Fix $\veps>0$ and let $K\subset\Omega$ be a compact set. Then there exists a polynomial $g\in\R[\t]$ such that: 
\begin{itemize}
\item[(i)] $\|f-g\|_{[a,b]}<\veps$.
\item[(ii)] $\|f^{(k)}-g^{(k)}\|_{K}<\veps$ for $k=1,\ldots,\ell$.
\item[(iii)] $g^{(k)}(t_i)=f^{(k)}(t_i)$ for $i=1,\ldots,r$ and $k=0,\ldots,\ell$. 
\end{itemize}
\end{lem}
\begin{proof}
Take polynomials $P_{ik}$ such that 
$$
P_{ik}^{(m)}(t_j)=\begin{cases}
0&\text{if $i\neq j$ or $k\ne m$},\\
1&\text{if $i=j$ and $k=m$,}
\end{cases}
$$ 
for $i=1,\ldots,r$ and $0\leq k,m\leq\ell$. For instance, {\em we may choose
\begin{equation}\label{poly}
P_{ik}:=c_{ik}(\t-t_i)^k\prod_{j\neq i}((\t-t_i)^{\ell+1}-(t_j-t_i)^{\ell+1})^{\ell+1}
\end{equation} 
where $c_{ik}:=\frac{1}{k!}\frac{(-1)^{(\ell+1)(r-1)}}{\prod_{j\neq i}(t_j-t_i)^{(\ell+1)^2}}$.} 

The Taylor expansion of $P_{ik}$ at $t_i$ has the form
$$
P_{ik}=\frac{1}{k!}(\t-t_i)^k+d_{ik}(\t-t_i)^{\ell+1}+\cdots
$$
for some $d_{ik}\in\R$, whereas the Taylor expansion of $P_{ik}$ at $t_j$ (for $j\neq i$) has the form
\begin{multline*}
P_{ik}=e_{ijk}(\t-t_j)^{\ell+1}+\cdots\\
\text{where } e_{ijk}:=c_{ik}(t_j-t_i)^k((\ell+1)(t_j-t_i)^{\ell})^{\ell+1}\prod_{\lambda\neq i,j}((t_j-t_i)^{\ell+1}-(t_\lambda-t_i)^{\ell+1})^{\ell+1}.
\end{multline*}
In both cases above the symbol $+\cdots$ means `plus terms of higher degree' with respect to either $\t-t_i$ or $\t-t_j$ depending on each case. To compute $e_{ijk}$ it is enough to figure out the first non-zero monomial of the Taylor expansion at $t_j$ of each factor of the product $P_{ik}$ and then to multiply them.

Define $K':=K\cup\{t_1,\ldots,t_r\}$ and
\begin{align}
\label{boundm}M&:=\max\{\|P_{ik}\|_{[a,b]},\|P_{ik}^{(m)}\|_{K'}:\ 1\leq i\leq r,\ 0\leq k\leq\ell,\ 1\leq m\leq\ell\},\\
\label{bounddelta}\delta&:=\frac{\veps}{1+r(\ell+1)M}.
\end{align}
By Theorem \ref{cotasi} there exists a Bernstein's polynomial $h\in\R[\t]$ of $f$ (in the interval $[a,b]$) such that $\|h-f\|_{[a,b]}<\delta$ and $\|h^{(k)}-f^{(k)}\|_{K'}<\delta$ for $k=1,\ldots,\ell$. Define
$$
g:=h+\sum_{i=1}^r\sum_{k=0}^{\ell}b_{ik}P_{ik}
$$
where $b_{ik}:=f^{(k)}(t_i)-h^{(k)}(t_i)$ for $i=1,\ldots,r$ and $k=0,\ldots,\ell$. Thus, 
$$
g^{(m)}(t_j)=h^{(m)}(t_j)+\sum_{i=1}^r\sum_{k=0}^{\ell}b_{ik}P_{ik}^{(m)}(t_j)=h^{(m)}(t_j)+b_{jm}=f^{(m)}(t_j)
$$
for $j=1,\ldots,r$ and $m=0,\ldots,\ell$.

As $|b_{ik}|=|f^{(k)}(t_i)-h^{(k)}(t_i)|<\delta$ for $i=1,\ldots,r$ and $k=0,\ldots,\ell$, we have
\begin{align*}
\|g-f\|_{[a,b]}&\leq\|h-f\|_{[a,b]}+\sum_{i=1}^r\sum_{k=0}^{\ell}|b_{ik}|\|P_{ik}\|_{[a,b]}<\delta+r(\ell+1)M\delta=\veps,\\
\|g^{(m)}-f^{(m)}\|_{K}&\leq\|h^{(m)}-f^{(m)}\|_{K}+\sum_{i=1}^r\sum_{k=0}^{\ell}|b_{ik}|\|P_{ik}^{(m)}\|_{K}<\delta+r(\ell+1)M\delta=\veps 
\end{align*}
for each $m=1,\ldots,\ell$, as required.
\end{proof}

\begin{remark}\label{light}
In the previous result we have chosen the same number $\ell$ of known derivatives for all the values $t_i$ in order to simplify the presentation, but it is possible to choose different numbers of known derivatives for each value $t_i$. The proof is quite similar but the notation is more intricate and the concrete details more cumbersome.\hfill$\sqbullet$
\end{remark}

The proof of Main Theorem \ref{nashsmart} still requires some preliminary work that we approach next.

\subsection{Polynomial paths with prescribed behavior at points and intervals}
We prove next (as a consequence of Lemma \ref{swdp}) a key result to prove Main Theorem \ref{nashsmart}. When we write a series in the form $h:=a_k\t^k+\cdots$, we mean that the lowest order term is $a_k\t^k$ (with $a_k\neq0$) and the remaining terms have higher order and are not relevant for our computation. Recall that $\R[\x]:=\R[\x_1,\ldots,\x_n]$.

\begin{lem}\label{clue}
Let $\Ss_0,\ldots,\Ss_r\subset\R^n$ be connected open semialgebraic sets (non-necessarily pairwise different) and pick points $x_i\in\cl(\Ss_{i-1})\cap\cl(\Ss_i)$ for $i=1,\ldots,r$. Assume that there exist a continuous path $\beta:[a,b]\to\bigcup_{k=0}^r\Ss_k\cup\{x_1,\ldots,x_r\}$ and values $a:=t_0<t_1<\cdots<t_r<t_{r+1}:=b$ satisfying the following properties:
\begin{itemize}
\item[(i)] $\beta([t_0,t_1))\subset\Ss_0$, $\beta((t_k,t_{k+1}))\subset\Ss_k$ for $k=1,\ldots,r-1$ and $\beta((t_r,t_{r+1}])\subset\Ss_r$,
\item[(ii)] $\beta(t_i)=x_i$ and $\beta$ is an analytic path on a neighborhood of $t_i$ for $i=1,\ldots,r$,
\item[(iii)] there exist polynomials $f_{ij}\in\R[\x]$ such that $\{f_{i1}>0,\ldots,f_{is}>0\}\subset\Ss_{i-1}$ is adherent to $x_i$ and the analytic series $(f_{ij}\circ\beta)(t_i-\t)=a_{ij}\t^{n_{ij}}+\cdots$ satisfies $a_{ij}>0$,
\item[(iv)] there exist polynomials $g_{ij}\in\R[\x]$ such that $\{g_{i1}>0,\ldots,g_{is}>0\}\subset\Ss_i$ is adherent to $x_i$ and the analytic series $(g_{ij}\circ\beta)(t_i+\t)=b_{ij}\t^{p_{ij}}+\cdots$ satisfies $b_{ij}>0$,
\end{itemize}
Let $\ell:=\max\{n_{ij},p_{ij}: 1\leq i\leq r,\ 1\leq j\leq s\}$ and $\Omega\subset[a,b]$ be an open neighborhood of $\{t_1,\ldots,t_r\}$ such that $\beta|_\Omega$ is analytic. 
\begin{itemize}
\item[(1)] There exists an open neighborhood ${\mathcal U}$ of $\beta\in\Cont^{\ell+4}_{\Omega}([a,b])$ in the $\Cont^{\ell}_{\Omega}$-topology such that if $\alpha\in{\mathcal U}$ and $\alpha^{(m)}(t_i)=\beta^{(m)}(t_i)$ for each $i=1,\ldots,r$ and each $m=0,\ldots,\ell$, then $\alpha((t_k,t_{k+1}))\subset\Ss_k$ for $k=0,\ldots,r$. 
\item[(2)] There exists a polynomial path $\alpha:[a,b]\to\bigcup_{k=0}^r\Ss_k\cup\{x_1,\ldots,x_r\}$ close to $\beta$ in the $\Cont^{\ell}_{\Omega}$-topology such that $\alpha(t_i)=x_i$ for $i=1,\ldots,r$ and $\alpha((t_k,t_{k+1}))\subset\Ss_k$ for $k=0,\ldots,r$.
\end{itemize}
\end{lem}
\begin{proof}
We prove this result as an application of Lemma \ref{swdp}. Observe that $(-1)^{n_{ij}}(f_{ij}\circ\beta)^{(n_{ij})}(t_i)>0$ and $(g_{ij}\circ\beta)^{(p_{ij})}(t_i)>0$ for each pair $i,j$. Thus, there exists $\delta>0$ such that for the compact interval $I_i:=[t_i-\delta,t_i+\delta]\subset\Omega$, $(-1)^{n_{ij}}(f_{ij}\circ\beta|_{I_i})^{(n_{ij})}>0$ and $(g_{ij}\circ\beta|_{I_i})^{(p_{ij})}>0$ for $i=1,\ldots,r$ and $j=1,\ldots,s$. Denote $J_0:=[a,t_1-\delta]$, $J_k:=[t_k+\delta,t_{k+1}-\delta]$ for $k=1,\ldots,r-1$ and $J_r:=[t_r+\delta,b]$. By Lemmas \ref{cont} and \ref{cont2} the maps
\begin{align*}
&\varphi_{ij}:\Cont^{\ell+4}_{\Omega}([a,b],\R^n)\to\Cont^{\ell+4}(I_i,\R), \gamma\mapsto f_{ij}\circ\gamma|_{I_i},\\
&\phi_{ij}:\Cont^{\ell+4}_{\Omega}([a,b],\R^n)\to\Cont^{\ell+4}(I_i,\R), \gamma\mapsto g_{ij}\circ\gamma|_{I_i},\\
&\psi_k:\Cont^{\ell+4}_{\Omega}([a,b],\R^n)\to\Cont^0(J_k,\R),\ \gamma\mapsto\dist(\gamma|_{J_k}(\t),\R^n\setminus\Ss_k)
\end{align*}
are continuous. In addition, as $\beta(J_k)\subset\Ss_k$, each function $\psi_k(\beta)$ is strictly positive for $k=0,\ldots,r$. Define 
$$
\veps:=\min_{i,j,k}\{\min\{(-1)^{n_{ij}}(f_{ij}\circ\beta|_{I_i})^{(n_{ij})}\},\min\{(g_{ij}\circ\beta|_{I_i})^{(p_{ij})}\},\min\{\psi_k(\beta)\}\}>0
$$ 
and consider
\begin{equation*}
\begin{split}
{\mathcal U}_0&:=\bigcap_{i=1}^r\bigcap_{j=1}^s\{\gamma\in\Cont^{\ell+4}_{\Omega}([a,b],\R^n):\ \|\varphi_{ij}(\gamma)^{(n_{ij})}-\varphi_{ij}(\beta)^{(n_{ij})}\|_{I_i}<\veps\}\\
&\cap\bigcap_{i=1}^r\bigcap_{j=1}^s\{\gamma\in\Cont^{\ell+4}_{\Omega}([a,b],\R^n):\ \|\phi_{ij}(\gamma)^{(p_{ij})}-\phi_{ij}(\beta)^{(p_{ij})}\|_{I_i}<\veps\}\\ 
&\cap\bigcap_{k=0}^r\{\gamma\in\Cont^{\ell+4}_{\Omega}([a,b],\R^n):\ \|\psi_k(\gamma)-\psi_k(\beta)\|_{J_k}<\veps\},
\end{split}
\end{equation*}
which is an open subset of $\Cont^{\ell+4}_{\Omega}([a,b],\R^n)$ in the $\Cont^{\ell}_{\Omega}$-topology. Consider the compact set $K:=\bigcup_{i=1}^rI_i\subset\Omega$. There exists $\rho>0$ such that 
$$
{\mathcal U}:=\{\gamma\in\Cont_{\Omega}^{\ell+4}([a,b],\R^n):\ \|\gamma-\beta\|_{[a,b]}<\rho,\ \|\gamma^{(m)}-\beta^{(m)}\|_{K}<\rho,\ m=1,\ldots,\ell\}\subset{\mathcal U}_0.
$$

We are ready to prove the assertions in the statement:

(1) We claim: \em If $\alpha\in{\mathcal U}$ and $\alpha^{(m)}(t_i)=\beta^{(m)}(t_i)$ for $i=1,\ldots,r$ and $m=0,\ldots,\ell$, then $\alpha((t_k,t_{k+1}))\subset\Ss_k$ for $k=0,\ldots,r$\em.

It holds $\alpha(J_k)\subset\Ss_k$ for $k=0,\ldots,r$, because $\alpha\in\{\gamma\in\Cont^{\ell+4}_\Omega([a,b]):\ |\psi_k(\gamma|_{J_k})-\psi_k(\beta|_{J_k})|<\veps\}$ for $k=0,\ldots,r$. Thus, to prove the claim it is enough to check:
\begin{align}
\label{2161}&\alpha([t_i-\delta,t_i))\subset\{f_{i1}>0,\ldots,f_{is}>0\}\subset\Ss_{i-1},\\
\label{2162}&\alpha((t_i,t_i+\delta])\subset\{g_{i1}>0,\ldots,g_{is}>0\}\subset\Ss_i
\end{align}
for $i=1,\ldots,r$. We show only \eqref{2161} because the proof of \eqref{2162} is analogous.

Using Taylor's expansion, we know that $\alpha$ around $t_i$ has the form
$$
\alpha(\t)=\sum_{m=0}^\ell\frac{1}{m!}\alpha^{(m)}(t_i)(\t-t_i)^m+(\t-t_i)^{\ell+1}\eta(\t-t_i)=\sum_{m=0}^\ell\frac{1}{m!}\beta^{(m)}(t_i)(\t-t_i)^m+(\t-t_i)^{\ell+1}\eta(\t-t_i)
$$
where $\eta$ is a continuous map defined on an interval around $0$ (recall that $\alpha\in\Cont_{\Omega}^{\ell+4}([a,b],\R^n$). As $\beta$ is analytic in a neighborhood of $t_i$, there exists a tuple of analytic series $\tau\in\R\{\t\}^n$ such that
$$
\beta(\t)=\sum_{m=0}^\ell\frac{1}{m!}\beta^{(m)}(t_i)(\t-t_i)^m+(\t-t_i)^{\ell+1}\tau(\t-t_i).
$$
Thus, if $\zeta:=\eta-\tau$, which is a continuous function around $0$, we deduce 
$$
\alpha(\t)-\beta(\t)=(\t-t_i)^{\ell+1}\zeta(\t-t_i)\quad\leadsto\quad\alpha(t_i-\t)-\beta(t_i-\t)=(-\t)^{\ell+1}\zeta(-\t).
$$
Recall that $\x:=(\x_1,\ldots,\x_n)$, write $\y:=(\y_1,\ldots,\y_n)$ and let $\z$ be a single variable. As the polynomial $f_{ij}(\x+\z\y)-f_{ij}(\x)$ vanishes on the real algebraic set $\{\z=0\}$, there exists a polynomial $F_{ij}\in\R[\x,\y,\z]$ such that
$$
f_{ij}(\x+\z\y)=f_{ij}(\x)+\z F_{ij}(\x,\y,\z).
$$
As $\ell\geq n_{ij}$, we deduce
\begin{multline*}
f_{ij}(\alpha(t_i-\t))=f_{ij}(\beta(t_i-\t)+\alpha(t_i-\t)-\beta(t_i-\t))=f_{ij}(\beta(t_i-\t)+(-\t)^{\ell+1}\zeta(-\t))\\
=f_{ij}(\beta(t_i-\t))+(-1)^{\ell+1}\t^{\ell+1}F_{ij}(\beta(t_i-\t),\zeta(-\t),(-1)^{\ell+1}\t^{\ell+1})
=a_{ij}\t^{n_{ij}}+\cdots.
\end{multline*}
Consequently, $(f_{ij}\circ\alpha)^{(m)}(t_i)=0$ for $m=0,\ldots,n_{ij}-1$ and $(-1)^{(n_{ij})}(f_{ij}\circ\alpha)^{(n_{ij})}(t_i)=n_{ij}!\,a_{ij}>0$. In addition, $\alpha(t_i-t)\in\{f_{i1}>0,\ldots,f_{is}>0\}$ for $t\in(0,\delta)$ close to $0$.

As $(-1)^{(n_{ij})}(f_{ij}\circ\beta|_{I_i})^{(n_{ij})}(t_i-\t)>\veps>0$ on $[-\delta,\delta]$ and $|(f_{ij}\circ\beta|_{I_i})^{(n_{ij})}-(f_{ij}\circ\alpha|_{I_i})^{(n_{ij})}|<\veps$, we conclude that $(-1)^{(n_{ij})}(f_{ij}\circ\alpha|_{I_i})^{(n_{ij})}(t_i-\t)>0$ on $[-\delta,\delta]$ for each $j=1,\ldots,s$. Suppose there exists a point $t^*\in[t_i-\delta,t_i)$ such that $\alpha(t^*)\not\in\{f_{i1}>0,\ldots,f_{is}>0\}$ and assume $(f_{i1}\circ\alpha)(t^*)\leq0$. As $\alpha(t_i-t)\in\{f_{i1}>0,\ldots,f_{is}>0\}$ for $t\in(0,\delta)$ close to $0$, there exists $\xi_0\in(0,\delta)$ such that $(f_{i1}\circ\alpha)(t_i-\xi_0)=0$. Assume by induction on $m\leq n_{i1}-1$ that there exist values $0<\xi_m<\cdots<\xi_1<\xi_0<\delta$ such that $(f_{i1}\circ\alpha)^{(j)}(t_i-\xi_j)=0$ for $j=0,\ldots,m$. As $(f_{i1}\circ\alpha)^{(m)}(t_i)=0$ and $(f_{i1}\circ\alpha)^{(m)}(t_i-\xi_m)=0$, there exists by Rolle's theorem $\xi_{m+1}\in(0,\xi_m)$ such that $(f_{i1}\circ\alpha)^{(m+1)}(t_i-\xi_{m+1})=0$. In particular, $(f_{i1}\circ\alpha)^{(n_{i1})}(t_i-\xi_{n_{i1}})=0$ and $\xi_{n_{i1}}\in(0,\delta)$, which contradicts the fact that $(-1)^{(n_{i1})}(f_{i1}\circ\alpha|_{I_i})^{(n_{i1})}(t_i-\t)>0$ on $[-\delta,\delta]$. Consequently, $\alpha(t)\in\{f_{i1}>0,\ldots,f_{is}>0\}$ for each $t\in[t_i-\delta,t_i)$. Observe that to prove the latter assertion we have only used that $|(f_{ij}\circ\beta|_{I_i})^{(n_{ij})}-(f_{ij}\circ\alpha|_{I_i})^{(n_{ij})}|<\veps$ and not that $|(f_{ij}\circ\beta|_{I_i})^{(m)}-(f_{ij}\circ\alpha|_{I_i})^{(m)}|<\veps$ for $m=1,\ldots,n_{ij}-1$. We will go deeper into this fact in Remark \ref{sharpr}(i).

(2) Let $K'\subset\Omega$ be a compact set that contains $K$ and let $0<\kappa<\rho$. By Lemma \ref{swdp} there exists a polynomial tuple $\alpha\in\R[\t]^n$ such that $\|\alpha-\beta\|_{[a,b]}<\kappa$, $\|\alpha^{(m)}-\beta^{(m)}\|_{K'}<\kappa$ for $m=1,\ldots,\ell$ (so $\alpha\in{\mathcal U}$) and $\alpha^{(m)}(t_i)=\beta^{(m)}(t_i)$ for $i=1,\ldots,r$ and $m=0,\ldots,\ell$. By (1) we deduce $\alpha((t_i,t_{i+1}))\subset\Ss_i$ for $i=0,\ldots,r$. In addition, $\alpha$ is close to $\beta$ in the $\Cont^\ell_\Omega$-topology of $\Cont^{\ell+4}_\Omega[a,b]$, as required.
\end{proof}

\begin{remarks}\label{sharpr}
(i) Suppose that in the statement of Lemma \ref{clue} each semialgebraic set $\Ss_i$ is the interior of an $n$-dimensional convex polyhedra. Then we may assume that each $\Ss_i:=\{h_{i1}>0,\ldots,h_{is}>0\}$ where $h_{ij}\in\R[\x]$ is a polynomial of degree $1$ for $i=0,\ldots,r$. Recall that $J_k:=[t_k+\delta,t_{k+1}-\delta]$ for $k=0,\ldots,r$ and $I_i:=[t_i-\delta,t_i+\delta]$ for $i=1,\ldots,r$. We keep the notations introduced in the statement and the proof of Lemma \ref{clue}(1) and we analyze how we can simplify the conditions that appear in the statement and the proof of Lemma \ref{clue}(1) to guarantee that $\alpha((t_k,t_{k+1}))\subset\Ss_k$ for $k=0,\ldots,r-1$. We consider $f_{ij}:=h_{i-1,j}$ and $g_{ij}:=h_{ij}$ for $i=1,\ldots,r$ and $j=1,\ldots,s$.

First, to have $\alpha(J_k)\subset\Ss_k$, it is enough that 
$$
\|\dist(\alpha|_{J_k},\R^n\setminus\Ss_k)-\dist(\beta|_{J_k},\R^n\setminus\Ss_k)\|_{J_k}<\min\{\dist(\beta|_{J_k},\R^n\setminus\Ss_k)\}
$$
for $k=0,\ldots,r-1$. By hypothesis the Taylor polynomials of $\alpha$ and $\beta$ at $t_i$ coincide until degree $\ell$. To guarantee that
\begin{align}
\label{2163}&\alpha([t_i-\delta,t_i))\subset\Ss_{i-1}=\{h_{i-1,1}>0,\ldots,h_{i-1,s}>0\},\\
\label{2164}&\alpha((t_i,t_i+\delta])\subset\Ss_i=\{h_{i1}>0,\ldots,h_{is}>0\}
\end{align}
for $i=1,\ldots,r$ it is enough to have, in view of the proof of Lemma \ref{clue}(1), the following properties:
\begin{align*}
&\|(h_{i-1,j}\circ\beta|_{I_i})^{(n_{ij})}-(h_{i-1,j}\circ\alpha|_{I_i})^{(n_{ij})}\|_{I_i}<\min\{(-1)^{(n_{ij})}(h_{i-1,j}\circ\beta|_{I_i})^{(n_{ij})}\},\\
&\|(h_{ij}\circ\beta|_{I_i})^{(p_{ij})}-(h_{ij}\circ\alpha|_{I_i})^{(p_{ij})}\|_{I_i}<\min\{(h_{ij}\circ\beta|_{I_i})^{(p_{ij})}\}
\end{align*}
for $i=1,\ldots,r$. Thus, we do not have to care about the derivatives of order strictly smaller than $n_{ij}$ or $p_{ij}$ (depending on the case). This reduction will be used in the proof of Main Theorem \ref{plcase} in order to simplify the estimations provided in \S\ref{bound}.

(ii) In view of Remark \ref{light} it is not necessary to use in Lemma \ref{clue} (1) that the derivatives of $\alpha$ and $\beta$ at $t_i$ coincide for $m=0,\ldots,\ell$, but only for $m=0,\ldots,\max\{n_{ij},p_{ij}:\ j=1,\ldots,s\}$.

(iii) If $\Ss_{i-1}=\Ss_i$ for some $i=1,\ldots,r$ in the statement of Lemma \ref{clue}, the condition $x_i\in\cl(\Ss_{i-1})\cap\cl(\Ss_i)$ means $x_i\in\cl(\Ss_i)$ and condition (i) reads as $\beta((t_{i-1},t_{i+1})\setminus\{t_i\})\subset\Ss_i$. The reader has to take this into account when applying Lemma \ref{clue} to prove Main Theorem \ref{nashsmart}.\hfill$\sqbullet$
\end{remarks}

\section{Drawing Nash paths inside semialgebraic sets}\label{s3}

In this section we prove Main Theorem \ref{nashsmart}. Before that we need a preliminary result. Again, if we write a series in the form $h:=a_k\t^k+\cdots$, we mean that the lowest order term is $a_k\t^k$ (with $a_k\neq0$) and the remaining terms have higher order and are not relevant for our computation.

\subsection{Double Nash curve selection lemma.}
The following result is an amalgamated modification of the classical (Nash) curve selection lemma \cite[Prop.8.1.13]{bcr} and double polynomial curve selection lemma \cite[Lem.3.8]{fu}. 

\begin{lem}[Double Nash curve selection lemma]\label{doublecurve}
Let $\Ss\subset\R^n$ be a semialgebraic set of dimension $d\geq2$ and $\Ss_d$ the set of points of $\Ss$ of dimension $d$. Pick a point $p\in\cl(\Ss_d)$. Then there exists a Nash arc $\alpha:[-1,1]\to\R^n$ such that $\alpha(0)=p$, $\alpha([-1,1]\setminus\{0\})\subset\Ss_d$ and $\alpha([-1,0))\cap\alpha((0,1])=\varnothing$. If $\Ss$ has dimension $n$, we may assume $\alpha$ is a polynomial arc.
\end{lem}
\begin{proof}
Let $X$ be the Zariski closure of $\Ss$ in $\R^n$, which is an algebraic set of dimension $d$. By Theorem \ref{hi1} there exist a non-singular algebraic set $X'\subset\R^m$ and a proper regular map $f:X'\to X$ such that $f|_{X'\setminus f^{-1}(\Sing(X))}:X'\setminus f^{-1}(\Sing(X))\rightarrow X\setminus\Sing(X)$ is a Nash diffeomorphism whose inverse map is also regular. As $\dim(\Sing(X))\leq d-1$, we have $\Ss_d\setminus\Sing(X)$ is dense in $\Ss_d$. As $p\in\cl(\Ss_d)=\cl(\Ss_d\setminus\Sing(X))$ and $f$ is proper, there exists a point $p'\in\cl(f^{-1}(\Ss_d\setminus\Sing(X)))$ such that $f(p')=p$. Assume that we find a Nash arc $\beta:[-1,1]\to\R^m$ such that $\beta(0)=p'$, $\beta([-1,1]\setminus\{0\})\subset f^{-1}(\Ss_d\setminus\Sing(X))$ and $\beta([-1,0))\cap\beta((0,1])=\varnothing$. As $f$ is a regular map and in particular a Nash map, if we define $\alpha:=f\circ\beta$, we will be done.

So let us assume: {\em the Zariski closure $X$ of $\Ss$ in $\R^n$ is non-singular (and consequently $X$ is a disjoint union of finitely many Nash manifolds maybe of different dimensions) and we have an algebraic set $Y\subset X$ of dimension strictly smaller than $d$ `to be avoided'}. Let $U\subset\R^n$ be an open semialgebraic neighborhood of $p$ in $X$ endowed with a Nash diffeomorphism $\varphi:U\to\R^d$ such that $\varphi(p)=0$. Let $\Ss'':=\varphi((\Ss_d\setminus Y)\cap U)$ and assume that we find a Nash arc $\gamma:[-1,1]\to\R^d$ such that $\gamma(0)=0$, $\gamma([-1,1]\setminus\{0\})\subset\Ss''$ and $\gamma([-1,0))\cap\gamma((0,1])=\varnothing$. If we define $\beta:=\varphi^{-1}\circ\gamma$, we will be done.

Thus, we can suppose: {\em $\Ss$ is pure dimensional of dimension $n\geq2$, the Zariski closure of $\Ss$ in $\R^n$ is $\R^n$ and $p\in\cl(\Ss)$ is the origin}. As $\Int(\Ss)$ is dense in $\Ss$ (because $\Ss$ is pure dimensional), there exists by \cite[Prop.8.1.13]{bcr} a Nash arc $\eta:=(\eta_1,\ldots,\eta_n):[-1,1]\to\R^n$ such that $\eta(0)=p$ and $\eta((0,1])\subset\Int(\Ss)$. After shrinking the domain of $\eta$, we may assume that each $\eta_i\in\R[[\t]]_{\rm alg}$ is an algebraic analytic series. After a linear change of coordinates and a reparameterization of $\eta$, we may assume that $\eta_2:=t^{\ell_2}$ for some $\ell_2\geq1$ (recall that $n\geq2$). As $\Int(\Ss)$ is an open semialgebraic subset of $\R^n$ and $p\in\cl(\Ss)=\cl(\Int(\Ss))$, there exist polynomials $f_1,\ldots,f_r\in\R[\x]$ such that $f_i(p)=0$ for $i=1,\ldots,r$ and
$$
\eta((0,\veps])\subset\{f_1>0,\ldots,f_r>0\}\subset\Int(\Ss)
$$
for some $0<\veps<1$ (because $\Int(\Ss)$ can be written by \cite[Thm.2.7.2]{bcr} as a finite union of basic open semialgebraic sets, see \S\ref{sst}). Consider the algebraic series $f_j(\eta)\in\R[[\t]]_{\rm alg}$, which satisfies $f_j(\eta)=a_j\t^{k_j}+\cdots$ for some $a_j>0$ and $k_j\geq1$. Define $m:=\max\{k_j:\ j=1,\ldots,r\}+\ell_2+1$ and let $q>2m$ be an odd positive integer. Let $\zeta_j\in\R[[\t]]_{\rm alg}$ be an algebraic series such that $\xi_j:=\eta_j+\t^m\zeta_j\in\R[\t]$ is a univariate polynomial for $j=1,\ldots,n$ and $\xi_2=\eta_2=\t^{\ell_2}$ (that is, $\zeta_2=0$). Denote $\xi:=(\xi_1,\ldots,\xi_n)$ and $\zeta:=(\zeta_1,\ldots,\zeta_n)$. Define $\gamma:=\xi(\t^2)+\t^qe_1\in\R[\t]^n$, where $e_1:=(1,0,\ldots,0)$. As the exponent $q$ is odd, all the exponents of the non-zero monomials (if any) of the polynomial $\xi_1(\t^2)$ are even and $\xi_2(\t^2)=\t^{2\ell_2}$, we deduce $\gamma([-\veps,0))\cap\gamma((0,\veps])=\varnothing$ for each $\veps>0$.

Let $\x:=(\x_1,\ldots,\x_n)$, $\y:=(\y_1,\ldots,\y_n)$ and $\z$ be a single variable. Write
$$
f_j(\x+\z\y)=f_j(\x)+\z h_j(\x,\y,\z)
$$
where $h_j\in\R[\x,\y,\z]$. Then
\begin{multline*}
f_j(\gamma(\t))=f_j(\xi(\t^2)+\t^qe_1)=f_j(\eta(\t^2)+\t^{2m}(\zeta(\t^2)+\t^{q-2m}e_1))\\
=f_j(\eta(\t^2))+\t^{2m}h_j(\eta(\t^2),\zeta(\t^2)+t^{q-2m}e_1,\t^{2m})=a_j\t^{2k_j}+\cdots,
\end{multline*}
so for $\veps>0$ small enough $\gamma:[-\veps,\veps]\to\R^n$ is a polynomial arc such that in addition 
$$
\gamma([-\veps,\veps]\setminus\{0\})\subset\{f_1>0,\ldots,f_r>0\}\subset\Int(\Ss)
$$ 
and $\gamma(0)=0=p$. After an affine reparameterization in order to have the interval $[-1,1]$ as the domain of $\gamma$, we deduce $\gamma$ is the searched polynomial path.
\end{proof}

\subsection{Smart Nash curve selection lemma}
Recall that a $d$-dimensional Nash manifold $M\subset\R^n$ with boundary is a $d$-dimensional smooth submanifold with boundary of $\R^n$ that is in addition a semialgebraic set. We are ready to prove Main Theorem \ref{nashsmart} (although we postpone some technicalities until Appendix \ref{A} for the sake of clearness).

\begin{proof}[Proof of Main Theorem \em\ref{nashsmart}]
Let $X\subset\R^n$ be the Zariski closure of $\Ss$ in $\R^n$, $\Tt:=\cl(\Ss)\setminus\Reg(\Ss)$ and $Y\subset X$ the Zariski closure of $\Tt\cup\Sing(X)$. If $d:=\dim(\Ss)$, then $\dim(X)=d$ and $\dim(Y)\leq d-1$, so $\Ss\setminus Y\neq\varnothing$ is dense in $\Ss$, because $\Ss$ is pure dimensional. The proof is conducted in several steps: 

\noindent{\sc Step 0. Reduction of the $1$ dimensional case to the $2$-dimensional case.} To avoid a misleading use of some preliminary results that only work for dimension $\geq2$, we study this case separately. Assume that $\dim(X)=1$. Define $\Ss^\bullet:=\Ss\cup\{p_1,\ldots,p_r,q_1,\ldots,q_{r-1}\}$, which is by \cite[Lem.7.3 \& Cor.7.6]{f1} irreducible. Observe that $X$ is also the Zariski closure of $\Ss^\bullet$, because $\Ss^\bullet\subset\cl(\Ss)$. Let $\widetilde{X}\subset\C^n$ be the Zariski closure of $X$ in $\C^n$ and let $(\widetilde{X}',\pi)$ be the normalization of $X$. We endow $(\widetilde{X}',\pi)$ with an involution $\widetilde{\sigma}:\widetilde{X}'\to \widetilde{X}'$ induced by the involution $\sigma:\widetilde{X}\to\widetilde{X}$ that arises from the restriction to $\widetilde{X}$ of the complex conjugation in $\C^n$. We may assume $\widetilde{X}'\subset\C^m$ and $\widetilde{\sigma}$ is the restriction to $\widetilde{X}'$ of the complex conjugation in $\C^m$ (see \cite[Prop.3.11]{fg1}). By \cite[Thm.3.15]{fg1} and as $\Ss^\bullet$ is irreducible, $\pi^{-1}(\Ss^\bullet)$ has a (unique) $1$-dimensional connected component $\Ss'_0$ such that $\pi(\Ss'_0)=\Ss^\bullet$. As $X$ has dimension $1$, it is a coherent analytic set, so $\Ss'_0\subset Z:=\widetilde{X}'\cap\R^m$. As $\widetilde{X}'$ is a normal curve, $Z$ is a non-singular real algebraic curve. We claim: \em the connected components of $Z$ are Nash diffeomorphic either to $\sph^1$ or to the real line $\R$\em.

By \cite[Thm.VI.2.1]{sh} there exist a compact affine non-singular real algebraic curve $Z^*$, a finite set $F$, which is empty if $Z$ is compact, and a union $Z'$ of some connected components of $Z^*\setminus F$ such that $Z$ is Nash diffeomorphic to $Z'$ and $\cl(Z')$ is a compact $1$-dimensional Nash manifold with boundary $F$. As $Z^*$ is a compact affine non-singular real algebraic curve, its connected components are diffeomorphic to $\sph^1$, so by \cite[Thm.VI.2.2]{sh} the connected components of $Z^*$ are in fact Nash diffeomorphic to $\sph^1$. Now, each connected component of $Z$ is Nash diffeomorphic to an open connected (semialgebraic) subset of $\sph^1$, as claimed. 

Consequently, $\Ss'_0$ is Nash diffeomorphic to a $1$-dimensional connected (semialgebraic) subset $\Ss'$ of $\sph^1$. Thus, there exists a generically $1$-$1$ surjective Nash map $\varphi$ from a connected (semialgebraic) subset $\Ss'$ of $\sph^1$ to $\Ss^\bullet$. By \cite[Thm.3.15]{fg1} and as each $\Ss_i$ is irreducible (because it is a connected Nash manifold \cite[(3.1)(i)]{fg1}), $\varphi^{-1}(\Ss_i)$ has a (unique) $1$-dimensional connected component $\Ss'_i$ such that $\varphi(\Ss'_i)=\Ss_i$, which is an open connected (semialgebraic) subset of $\sph^1$. As there exists a Nash bridge $\Gamma_i$ between $\Ss_i$ and $\Ss_{i+1}$ with base point $q_i$, there exists by \cite[Lem.B.2]{f1} a Nash bridge $\Gamma_i'$ between $\Ss_i'$ and $\Ss_{i+1}'$ with base point $q_i'\in\sph^1$ such that $\varphi(q_i')=q_i$ for $i=1,\ldots,r-1$. Pick points $p_i'\in\cl(\Ss'_i)$ such that $\varphi(p_i')=p_i$ for $i=1,\ldots,r$. Observe that: {\em If $\beta:[0,1]\to\Ss^\bullet$ is a continuous semialgebraic path satisfying the conditions of the statement of Main Theorem {\em\ref{nashsmart}} with respect to $\Ss^\bullet$, there exists} by \cite[Lem.B.1 \& B.2]{f1} {\em a continuous semialgebraic path $\gamma:[0,1]\to\Ss'$ satisfying the conditions of such statement with respect to $\Ss'$ such that $\varphi\circ\gamma=\beta$}. In this case we take $p_i':=\gamma(t_i)$, which fulfills $\varphi(p_i')=p_i$, for $i=1,\ldots,r$.

Consider the Nash retraction $\psi:\R^2\setminus\{0\}\to\sph^1,\ (x,y)\mapsto\frac{(x,y)}{\sqrt{x^2+y^2}}$, which satisfies $\psi|_{\sph^1}=\id_{\sph^1}$, and define $\Ss_i'':=\psi^{-1}(\Ss_i')$, which contains $\Ss_i'$, for $i=1,\ldots,r$. We have: 
\begin{itemize}
\item $\Ss_i''$ is an open connected semialgebraic subset of $\R^2\setminus\{0\}$, which is a Nash manifold. 
\item $p_i'\in\cl(\Ss_i')\subset\cl(\Ss_i'')$ for $i=1,\ldots,r$.
\item $q_i'\in\cl(\Ss_i')\cap\cl(\Ss_{i+1}')\subset\cl(\Ss_i'')\cap\cl(\Ss_{i+1}'')$ for $i=1,\ldots,r-1$.
\item $\Gamma_i'$ is a Nash bridge between $\Ss_i'\subset\Ss_i''$ and $\Ss_{i+1}'\subset\Ss_{i+1}''$ with base point $q_i'$ for $i=1,\ldots,r-1$. 
\end{itemize}
Thus, if we find a Nash path $\alpha_0:[0,1]\to\bigcup_{i=1}^r\Ss_i''\cup\{p_1',\ldots,p_r',q_1',\ldots,q_{r-1}'\}$ satisfying the required conditions of the statement of Main Theorem \ref{nashsmart} for the new setting, then $\alpha:=\varphi\circ\psi\circ\alpha_0:[0,1]\to\Ss^\bullet=\bigcup_{i=1}^r\Ss_i\cup\{p_1,\ldots,p_r,q_1,\ldots,q_{r-1}\}$ is a Nash path satisfying the required conditions in the statement.

Consequently, to prove Main Theorem \ref{nashsmart} we assume in the following that $d\geq2$. To lighten notations, we reset all the notations used in {\sc Step 0}.

\noindent{\sc Step 1. Construction of a suitable continuous semialgebraic path $\beta$.}
We show first: {\em There exists a continuous semialgebraic path $\beta:[0,1]\to\R^n$ such that 
$$
\eta(\beta)\subset(0,1)\setminus\{t_1,\ldots,t_r,s_1,\ldots,s_{r-1}\},
$$ 
$\beta(\eta(\beta))\subset\bigcup_{i=1}^r\Ss_i$ and $\beta$ satisfies conditions {\em(i)}, {\em(ii)} and {\em(iii)} in the statement of Main Theorem {\em\ref{nashsmart}}}. Recall that $\Tt:=\cl(\Ss)\setminus\Reg(\Ss)$ and $Y\subset X$ the Zariski closure of $\Tt\cup\Sing(X)$.

Let us check: {\em For each $i=1,\ldots,r-1$ we may modify the Nash bridges $\Gamma_i$ in order to have in addition $\Gamma_i\cap Y\subset\{q_i\}$ and $(\Gamma_i\setminus\{q_i\})\cap(\Gamma_j\setminus\{q_j\})=\varnothing$ if $i\neq j$}. 

Pick any index $i=1,\ldots,r-1$ and suppose we have constructed the Nash bridges $\Gamma_j$ for $1\leq j\leq i-1$ satisfying the required conditions. Denote the Zariski closure of $\bigcup_{j=1}^{i-1}\Gamma_j$ with $Y_i'$. We distinguish two cases:

\noindent{\sc Case 1.} Suppose first $q_i\in\cl(\Ss_i\cap\Ss_{i+1})$. Observe that $\Ss_i\cap\Ss_{i+1}\neq\varnothing$ is pure dimensional and $\dim(\Ss_i\cap\Ss_{i+1})=d$. As $\dim(Y)<\dim(\Ss_i\cap\Ss_{i+1})$ and $\dim(Y_i')\leq1<2\leq\dim(\Ss_i\cap\Ss_{i+1})$, we have $q_i\in\cl((\Ss_i\cap\Ss_{i+1})\setminus(Y\cup Y_i'))$. By Lemma \ref{doublecurve} there exists a Nash arc $\alpha:[-1,1]\to\R^n$ such that $\alpha(0)=q_i$, $\alpha([-1,1]\setminus\{0\})\subset(\Ss_i\cap\Ss_{i+1})\setminus(Y\cup Y_i'))$ and $\alpha([-1,0))\cap\alpha((0,1])=\varnothing$. We substitute the old $\Gamma_i$ by the new $\Gamma_i:=\alpha([-1,1])$ and observe $\Gamma_i\cap Y\subset\{q_i\}$ and $(\Gamma_i\setminus\{q_i\})\cap(\Gamma_j\setminus\{q_j\})=\varnothing$ if $1\leq j\leq i-1$. 

\begin{center}
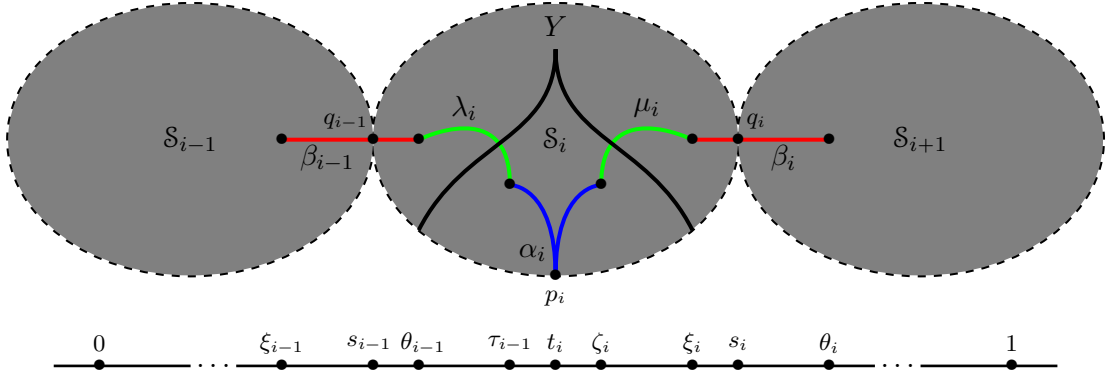
\begin{figure}[ht]
\begin{tikzpicture}[scale=1.2]

\draw[dashed,line width=1.5pt] (1,1) ellipse (2cm and 1.5cm);
\draw[dashed,line width=1.5pt] (5,1) ellipse (2cm and 1.5cm);
\draw[dashed,line width=1.5pt] (9,1) ellipse (2cm and 1.5cm);

\draw[fill=gray!100,opacity=0.5,draw=none] (1,1) ellipse (2cm and 1.5cm); 
\draw[fill=gray!100,opacity=0.5,draw=none] (5,1) ellipse (2cm and 1.5cm); 
\draw[fill=gray!100,opacity=0.5,draw=none] (9,1) ellipse (2cm and 1.5cm); 

\draw[color=blue,line width=1.5pt] (4.5,0.5) .. controls (4.5,0.5) and (5,0.5) .. (5,-0.5) .. controls (5,0.5) and (5.5,0.5) .. (5.5,0.5);

\draw[color=red,line width=1.5pt] (2,1) -- (3.5,1);
\draw[color=red,line width=1.5pt] (6.5,1) -- (8,1);

\draw[color=green,line width=1.5pt] (3.5,1) .. controls (3.5,1) and (4.5,1.5) .. (4.5,0.5);
\draw[color=green,line width=1.5pt] (6.5,1) .. controls (6.5,1) and (5.5,1.5) .. (5.5,0.5);

\draw (3,1) node{$\bullet$};
\draw (7,1) node{$\bullet$};
\draw (5,-0.5) node{$\bullet$};
\draw (3.5,1) node{$\bullet$};
\draw (2,1) node{$\bullet$};
\draw (6.5,1) node{$\bullet$};
\draw (8,1) node{$\bullet$};
\draw (4.5,0.5) node{$\bullet$};
\draw (5.5,0.5) node{$\bullet$};

\draw (2.7,1.2) node{\footnotesize$q_{i-1}$};
\draw (7.2,1.2) node{\footnotesize$q_i$};
\draw (5,-0.75) node{\footnotesize$p_i$};

\draw (2.5,0.8) node{$\beta_{i-1}$};
\draw (7.5,0.8) node{$\beta_i$};
\draw (4.75,-0.25) node{$\alpha_i$};
\draw (4,1.35) node{$\lambda_i$};
\draw (6,1.35) node{$\mu_i$};

\draw[color=black,line width=1pt] (1.5,-1.5) -- (8.5,-1.5);
\draw (1.25,-1.5) node{$\ldots$};
\draw (8.75,-1.5) node{$\ldots$};

\draw[color=black,line width=1pt](-0.5,-1.5) -- (1,-1.5);
\draw (0,-1.5) node{$\bullet$};
\draw (0,-1.25) node{\footnotesize$0$};

\draw[color=black,line width=1pt](9,-1.5) -- (10.5,-1.5);
\draw (10,-1.5) node{$\bullet$};
\draw (10,-1.25) node{\footnotesize$1$};

\draw[color=black,line width=1.5pt] (3.5,0) .. controls (4,1) and (5,1) .. (5,2) .. controls (5,1) and (6,1) .. (6.5,0);
\draw (5,2.25) node{$Y$};

\draw (2,-1.5) node{$\bullet$};
\draw (3,-1.5) node{$\bullet$};
\draw (3.5,-1.5) node{$\bullet$};
\draw (5,-1.5) node{$\bullet$};
\draw (4.5,-1.5) node{$\bullet$};
\draw (5.5,-1.5) node{$\bullet$};
\draw (6.5,-1.5) node{$\bullet$};
\draw (7,-1.5) node{$\bullet$};
\draw (8,-1.5) node{$\bullet$};

\draw (2,-1.25) node{\footnotesize$\xi_{i-1}$};
\draw (3,-1.25) node{\footnotesize$s_{i-1}\ $};
\draw (3.5,-1.25) node{\footnotesize$\ \theta_{i-1}$};
\draw (4.5,-1.25) node{\footnotesize$\tau_{i-1}$};
\draw (5,-1.25) node{\footnotesize$t_i$};
\draw (5.5,-1.25) node{\footnotesize$\zeta_i$};
\draw (6.5,-1.25) node{\footnotesize$\xi_i$};
\draw (7,-1.25) node{\footnotesize$s_i$};
\draw (8,-1.25) node{\footnotesize$\theta_i$};

\draw (1,1) node{$\Ss_{i-1}$};
\draw (5,1) node{$\Ss_i$};
\draw (9,1) node{$\Ss_{i+1}$};

\end{tikzpicture}
\caption{Construction of the Nash paths $\lambda_i$ and $\mu_i$.\label{fig2}}
\vspace*{-1.75em}
\end{figure}
\end{center}

\noindent{\sc Case 2.} Suppose next $q_i\not\in\cl(\Ss_i\cap\Ss_{i+1})$. Then there exists an open semialgebraic neighborhood $U\subset X$ of $q_i$ such that $\Ss_i\cap\Ss_{i+1}\cap U=\varnothing$. As $q_i\in\cl(\Ss_i)\cap\cl(\Ss_{i+1})$, we also have $q_i\in\cl(\Ss_i\cap U)\cap\cl(\Ss_{i+1}\cap U)$. We shrink $U$ to have in addition that $\Ss_i\cap U$ and $\Ss_{i+1}\cap U$ are connected Nash manifolds. Shrinking $\Gamma_i$ if necessary we have that it is a Nash bridge between $\Ss_i\cap U$ and $\Ss_{i+1}\cap U$ with base point $q_i$. By \cite[Main Thm.1.1 \& Prop.7.6]{f1} the union $(\Ss_i\cap U)\cup(\Ss_{i+1}\cap U)\cup\{q_i\}$ is a semialgebraic set connected by analytic paths. By \cite[Prop.7.8]{f1} we may assume that $\Gamma_i\cap (Y\cup Y_i')\subset\{q_i\}$. In particular, $(\Gamma_i\setminus\{q_i\})\cap(\Gamma_j\setminus\{q_j\})=\varnothing$ if $1\leq j\leq i-1$. 

Next, let $\beta_i:[-1,1]\to\Gamma_i\subset\Ss\cup\{q_i\}$ be a Nash parameterization of the Nash bridge $\Gamma_i$ such that $\beta_i(0)=q_i$, $\beta_i([-1,0))\subset\Ss_i$ and $\beta_i((0,1])\subset\Ss_{i+1}$. Let $Y'$ be the Zariski closure of $\bigcup_{i=1}^{r-1}\Gamma_i$. Using Lemma \ref{doublecurve} recursively we find Nash arcs $\alpha_i:[-1,1]\to\Ss_i\cup\{p_i\}$ such that $\alpha_i(0)=p_i$, $\alpha_i([-1,1]\setminus\{0\})\subset\Ss_i\setminus (Y\cup Y')$, $\alpha_i([-1,0))\cap\alpha_j((0,1])=\varnothing$ and if we denote $\Lambda_i:=\alpha_i([-1,1])$, then $(\Lambda_i\setminus\{p_i\})\cap(\Lambda_j\setminus\{p_j\})=\varnothing$ for $1\leq j<i\leq r$. In addition, $(\Gamma_i\setminus\{q_i\})\cap(\Lambda_j\setminus\{p_j\})=\varnothing$ for $i=1,\ldots,r-1$ and $j=1,\ldots,r$.

Thus, the collection of semialgebraic sets 
$$
\{\Gamma_i\setminus\{q_i\}:\ i=1,\ldots,r-1\}\cup\{\Lambda_j\setminus\{p_j\}: j=1,\ldots,r\}
$$ 
is a pairwise disjoint family. We affinely reparameterize the domains of $\beta_i$ and $\alpha_j$ and shrink them if necessary in such a way that there exist values
\begin{multline*}
\tau_0:=s_0=0<t_1<\zeta_1<\xi_1<s_1<\theta_1<\tau_1<t_2<\zeta_2<\cdots\\
\tau_{r-2}<t_{r-1}<\zeta_{r-1}<\xi_{r-1}<s_{r-1}<\theta_{r-1}<\tau_{r-1}<t_r<1=s_r=:\zeta_r
\end{multline*}
such that:
\begin{itemize}
\item[$\bullet$] $\alpha_i:[\tau_{i-1},\zeta_i]\to\Ss_i\cup\{p_i\}$ and $\alpha_i(t_i)=p_i$.
\item[$\bullet$] $\beta_i:[\xi_i,\theta_i]\to\Gamma_i$ and $\beta_i(s_i)=q_i$.
\end{itemize}

The points $\alpha_i(\tau_{i-1}),\alpha_i(\zeta_i),\beta_{i-1}(\theta_{i-1}),\beta_i(\xi_i)$ belong to $\Ss_i\setminus Y$, which is an open semialgebraic subset of the connected Nash manifold $\Ss_i$, and they are pairwise different. By \cite[Thm.1.5]{f1} there exist: 
\begin{itemize}
\item a Nash path $\lambda_i:[\theta_{i-1},\tau_{i-1}]\to\Ss_i$ such that $\lambda_i(\theta_{i-1})=\beta_{i-1}(\theta_{i-1})$ and $\lambda_i(\tau_{i-1})=\alpha_i(\tau_{i-1})$,
\item a Nash path $\mu_i:[\zeta_i,\xi_i]\to\Ss_i$ such that $\mu_i(\zeta_i)=\alpha_i(\zeta_i)$ and $\mu_i(\xi_i)=\beta_i(\xi_i)$.
\end{itemize}
By \cite[Lem.7.7]{f1} we have $\lambda_i^{-1}(Y)$ and $\mu_i^{-1}(Y)$ are finite sets (Figure \ref{fig2}). 

Denote $Z:=\{\tau_0,\ldots,\tau_{r-1},\zeta_1,\ldots,\zeta_r,\xi_1,\ldots,\xi_{r-1},\theta_1,\ldots,\theta_{r-1}\}$. Thus, concatenating all the previous Nash paths and arcs we construct a piecewise Nash path $\beta:[0,1]\to\R^n$ such that
\begin{itemize}
\item[(1)] $\beta([0,1])\subset\bigcup_{i=1}^r\Ss_i\cup\{p_1,\ldots,p_r,q_1,\ldots,q_{r-1}\}$.
\item[(2)] $\beta(t_i)=p_i$ for $i=1,\ldots,r$.
\item[(3)] $\beta((t_i,s_i))\subset\Ss_i$, $\beta((s_i,t_{i+1}))\subset\Ss_{i+1}$ and $\beta(s_i)=q_i$.
\item[(4)] $\eta(\beta)\subset Z\subset(0,1)\setminus\{t_1,\ldots,t_r,s_1,\ldots,s_{r-1}\}$ (because $\beta|_{[0,1]\setminus Z}$ is a Nash map).
\item[(5)] $\beta^{-1}(Y)$ is a finite set and $\eta(\beta)\cap\beta^{-1}(Y)=\varnothing$ (because $\eta(\beta)\subset Z$ and $\beta(Z)\cap Y=\varnothing$).
\end{itemize}

Thus, we have provided a procedure to construct a continuous semialgebraic path $\beta:[0,1]\to\R^n$ such that $\eta(\beta)\subset(0,1)\setminus\{t_1,\ldots,t_r,s_1,\ldots,s_{r-1}\}$, $\beta^{-1}(Y)$ is a finite set, $\eta(\beta)\cap\beta^{-1}(Y)=\varnothing$ and $\beta$ satisfies conditions (i), (ii) and (iii) in the statement.  

\noindent{\sc Step 2. Modification of a given continuous semialgebraic path $\beta$.}
Fix in this step any continuous semialgebraic path $\beta:[0,1]\to\R^n$ satisfying the required conditions (i), (ii) and (iii) in the statement. By Lemma \ref{modification} (below) we may assume in addition (perturbing $\beta$ slightly if necessary) that $\beta^{-1}(Y)$ is a finite set and $\eta(\beta)\cap\beta^{-1}(Y)=\varnothing$. For the sake of clearness and to make the proof more discursive we have postponed this technical part of the proof until Appendix \ref{A}.

\noindent{\sc Step 3. Reduction to the open semialgebraic setting.} 
By Theorem \ref{hi1} there exist a non-singular algebraic set $X'\subset\R^m$ and a proper regular map $f:X'\to X$ such that the restriction $f|_{X'\setminus f^{-1}(\Sing(X))}:X'\setminus f^{-1}(\Sing(X))\to X\setminus\Sing(X)$ is a Nash diffeomorphism whose inverse map is also regular. If $A\subset X$, the strict transform of $A$ under $f$ is $A':=\cl(f^{-1}(A\setminus\Sing(X))\cap f^{-1}(A)$. As $f$ is proper, $f(A')=\cl(A\setminus\Sing(X))\cap A$. Thus, if $A\setminus\Sing(X)$ is dense in $A$, one has $f(A')=A$. This happens for instance if $A$ is a pure dimensional semialgebraic set of dimension $d$.

Let $\Ss'$ be the strict transform of $\Reg(\Ss)$ under $f$ and $\Ss_i'$ the strict transform of $\Ss_i$, which is a connected Nash submanifold of $\R^m$, because $\Reg(\Ss)\subset X\setminus\Sing(X)$. By \cite[Lem.B.1 \& B.2]{f1} the strict transform under $f$ of $\beta$ is a continuous semialgebraic path $\gamma:[0,1]\to\cl(\Ss')$, which satisfies $f\circ\gamma=\beta$. Denote $p_i':=\gamma(t_i)$ and $q_i':=\gamma(s_i)$. Observe that $f(p_i')=p_i$ for $i=1,\ldots,r$ and $f(q_i')=q_i$ for $i=1,\ldots,r-1$. We have:
\begin{itemize}
\item[(i)] $\gamma([0,1])\subset\bigcup_{i=1}^r\Ss_i'\cup\{p_1',\ldots,p_r',q_1',\ldots,q'_{r-1}\}$.
\item[(ii)] $\gamma(t_i)=p_i'$ for $i=1,\ldots,r$.
\item[(iii)] $\gamma((t_i,s_i))\subset\Ss_i'$, $\gamma((s_i,t_{i+1}))\subset\Ss_{i+1}'$ and $\gamma(s_i)=q_i'$.
\end{itemize}

By \cite[Cor.8.9.5]{bcr} there exists a Nash tubular neighborhood $(U,\rho)$ of $X'$ in $\R^m$ where $\rho:U\to X'$ is a Nash retraction. Define $\Ss'':=\rho^{-1}(\Ss')$ and $\Ss''_k:=\rho^{-1}(\Ss'_k)$ for $k=1,\ldots,r$, which are open semialgebraic subsets of $\R^m$. As each Nash manifold $\Ss_k'$ is connected, shrinking $U$ if necessary, we may assume in addition that each $\Ss''_k$ is connected. Observe that $\gamma([0,1])\subset\bigcup_{i=1}^r\Ss''_k\cup\{p_1',\ldots,p_r',q_1',\ldots,q'_{r-1}\}$. There exists $\kappa>0$ small enough such that $\gamma|_{[s_i-\kappa,s_i+\kappa]}$ supplies by \cite[Lem.B.1 \& Lem. B.2]{f1} a Nash bridge between $\Ss_i''$ and $\Ss_{i+1}''$ for $i=1,\ldots,r-1$.

\noindent{\sc Step 4. Computing the order of differentiability.}
We need to compute certain positive integer $\ell$ in order to apply Lemma \ref{clue}(2). Recall that each $\Ss_i''$ is an open semialgebraic set and $\gamma$ is a Nash path in a neighborhood of the finite set $\{t_1,\ldots,t_r,s_1,\ldots,s_{r-1}\}$ such that $\gamma$ is a non-trivial Nash arc inside $\Ss_i''\cup\{p_i'\}$ around $t_i$ and $\gamma$ provides a Nash bridge between $\Ss_i''$ and $\Ss_{i+1}''$ with base point $q_i'$ around $s_i$. As each $\Ss_i''$ is an open semialgebraic set, it is by \cite[Thm.2.7.2]{bcr} a finite union of basic open semialgebraic sets, see \S\ref{sst}. As $\gamma$ is a non-trivial Nash arc (around $t_i$) inside $\Ss_i''\cup\{p_i'\}$, both (open) branches around $t_i$ are contained in one of these basic open semialgebraic sets. Thus, there exist polynomials $f_{ij},g_{ij}\in\R[\x]$ such that: 
\begin{itemize}
\item[$\bullet$] $\{f_{i1}>0,\ldots,f_{is}>0\}\subset\Ss''_i$ is adherent to $p_i'$ and $(f_{ij}\circ\gamma)(t_i-\t)=a_{ij}\t^{e_{ij}}+\cdots$, where $a_{ij}>0$ and $e_{ij}$ is a positive integer.
\item[$\bullet$] $\{g_{i1}>0,\ldots,g_{is}>0\}\subset\Ss''_i$ is adherent to $p_i'$ and $(g_{ij}\circ\gamma)(t_i+\t)=b_{ij}\t^{u_{ij}}+\cdots$, where $b_{ij}>0$ and $u_{ij}$ is a positive integer.
\end{itemize}
Analogously, as $\gamma$ provides (around $s_i$) a Nash bridge between $\Ss_i''$ and $\Ss_{i+1}''$ with base point $q_i'$, one of its two (open) branches around $t_i$ is contained in a basic open semialgebraic subset of $\Ss_i''$ and its other (open) branch around $t_i$ is contained in a basic open semialgebraic subset of $\Ss_{i+1}''$. Thus, there exist polynomials $h_{ij},m_{ij}\in\R[\x]$ such that: 
\begin{itemize}
\item[$\bullet$] $\{h_{i1}>0,\ldots,h_{is}>0\}\subset\Ss''_i$ is adherent to $q_i'$ and $(h_{ij}\circ\gamma)(s_i-\t)=c_{ij}\t^{v_{ij}}+\cdots$, where $c_{ij}>0$ and $v_{ij}$ is a positive integer.
\item[$\bullet$] $\{m_{i1}>0,\ldots,m_{is}>0\}\subset\Ss''_{i+1}$ is adherent to $q_i'$ and $(m_{ij}\circ\gamma)(s_i+\t)=d_{ij}\t^{w_{ij}}+\cdots$, where $d_{ij}>0$ and $w_{ij}$ is a positive integer.
\end{itemize}
Define $\ell:=\max\{e_{ij},u_{ij},v_{ij},w_{ij}:\ 1\leq i\leq r,1\leq j\leq s\}$.

\noindent{\sc Conclusion.} By Lemma \ref{clue}(2) there exists a polynomial path $\alpha_0:\R\to\R^m$ that satisfies: 
\begin{itemize}
\item[(i)] $\alpha_0([0,1])\subset\bigcup_{i=1}^r\Ss_i''\cup\{p_1',\ldots,p_r',q_1',\ldots,q_{r-1}'\}$.
\item[(ii)] $\alpha_0(t_i)=p_i'$ for $i=1,\ldots,r$.
\item[(iii)] $\alpha_0((t_i,s_i))\subset\Ss_i''$, $\alpha_0((s_i,t_{i+1}))\subset\Ss_{i+1}''$ and $\alpha_0(s_i)=q_i'$ for $i=1,\ldots,r-1$.
\item[(iv)] $\alpha_0|_{[0,1]}$ is close to $\gamma$ in the $\Cont^0$ topology.
\end{itemize}
Define $\alpha_1:=\rho\circ\alpha_0:\R\to\R^m$ (where $\rho$ is the Nash retraction provided in {\sc Step 3}), which is a Nash path that satisfies:
\begin{itemize}
\item[(i)] $\alpha_1([0,1])\subset\bigcup_{i=1}^r\Ss_i'\cup\{p_1',\ldots,p_r',q_1',\ldots,q_{r-1}'\}$.
\item[(ii)] $\alpha_1(t_i)=p_i'$ for $i=1,\ldots,r$.
\item[(iii)] $\alpha_1((t_i,s_i))\subset\Ss_i'$, $\alpha_1((s_i,t_{i+1}))\subset\Ss_{i+1}'$ and $\alpha_1(s_i)=q_i'$ for $i=1,\ldots,r-1$.
\item[(iv)] $\alpha_1|_{[0,1]}$ is close to $\rho\circ\gamma=\gamma$ in the $\Cont^0$ topology (Lemma \ref{cont}).
\end{itemize}
Next define $\alpha:=f\circ\alpha_1:\R\to\R^n$, which is a Nash path that satisfies:
\begin{itemize}
\item[(i)] $\alpha([0,1])\subset\bigcup_{i=1}^r\Ss_i\cup\{p_1,\ldots,p_r,q_1,\ldots,q_{r-1}\}$.
\item[(ii)] $\alpha(t_i)=p_i$ for $i=1,\ldots,r$.
\item[(iii)] $\alpha((t_i,s_i))\subset\Ss_i$, $\alpha((s_i,t_{i+1}))\subset\Ss_{i+1}$ and $\alpha(s_i)=q_i$ for $i=1,\ldots,r-1$.
\item[(iv)] $\alpha|_{[0,1]}$ is close to $f\circ\gamma=\beta$ in the $\Cont^0$ topology (Lemma \ref{cont}),
\end{itemize}
as required.
\end{proof}

We revisit next a well-known characterization of the connexion by analytic paths for semialgebraic sets. This result was proved indirectly in \cite[Main Thm.1.4]{f1} showing that the corresponding two properties are both equivalent to the fact that the involved semialgebraic set is the image of some $\R^d$ under a Nash map.

\begin{cor}\label{newproof}
Let $\Ss\subset\R^n$ be a semialgebraic set of dimension $d$. The following conditions are equivalent:
\begin{itemize}
\item[(i)] $\Ss$ is connected by analytic paths.
\item[(ii)] $\Ss$ is pure dimensional and there exists an analytic path $\alpha:[0,1]\to\Ss$ whose image meets all the connected components of $\Reg(\Ss)$.
\end{itemize}
\end{cor}
\begin{proof}
Let $\Ss_1,\ldots,\Ss_\ell$ be the connected components of $\Reg(\Ss)$, which are pairwise disjoint. Let $\Lambda$ be the graph proposed in Remark \ref{graph} whose vertices are the Nash manifolds $\Ss_1,\ldots,\Ss_\ell$ and such that there exists an edge between the vertices $\Ss_i$ and $\Ss_j$ if and only if there exists a Nash bridge inside $\Ss$ between $\Ss_i$ and $\Ss_j$. When $\Lambda$ is a connected graph, there exists a sequence of semialgebraic sets $\Tt_1,\ldots,\Tt_r$ such that $\{\Ss_1,\ldots,\Ss_\ell\}=\{\Tt_1,\ldots,\Tt_r\}$ and for each index $i=1,\ldots,r-1$ there exists a Nash bridge inside $\Ss$ between $\Tt_i$ and $\Tt_{i+1}$.

(i) $\Longrightarrow$ (ii) We prove first that $\Ss$ is pure dimensional. Otherwise, there exists a point $x\in\Ss$ and an open semialgebraic neighborhood $U\subset\R^n$ of $x$ such that $\dim(\Ss\cap U)<\dim(\Ss)$. Let $Y$ be the Zariski closure of $\Ss\cap U$ and pick a point $y\in\Ss\setminus Y$, which is non-empty because $\dim(Y)<\dim(\Ss)$. As $\Ss$ is connected by analytic paths, there exists an analytic path $\alpha:[0,1]\to\Ss$ such that $\alpha(0)=x$ and $\alpha(1)=y$. The inverse image $V:=\alpha^{-1}(\Ss\cap U)$ is an open semialgebraic subset of $[0,1]$ that contains $0$. Let $f\in\R[\x]$ be a polynomial equation of $Y$. As $(f\circ\alpha)|_V=0$ and $[0,1]$ is connected, the identity principle for analytic functions implies that $f\circ\alpha=0$, so $f(y)=0$ and $y\in Y$, which is a contradiction. Thus, $\Ss$ is pure dimensional.

By Lemma \ref{graphconnected} we know that $\Lambda$ is a connected graph. Pick points $x_i\in\Tt_i$ for $i=1,\ldots,r$. By Main Theorem \ref{nashsmart} there exists a Nash path $\alpha:[0,1]\to\Ss$ such that $\alpha(\frac{k}{r+1})=x_k$ for $k=1,\ldots,r$. Thus $\alpha:[0,1]\to\Ss$ is an analytic path that meets all the connected components of $\Reg(\Ss)$.

(ii) $\Longrightarrow$ (i) We prove next recursively that: {\em $\Lambda$ is a connected graph}. It is enough: {\em to reorder recursively the indices $i=1,\ldots,\ell$ in such a way that for each $i=2,\ldots,\ell$ there exists a Nash bridge inside $\Ss$ between $\Ss_i$ and some $\Ss_j$ with $1\leq j\leq i-1$}. 

Define $t_i:=\inf(\alpha^{-1}(\Ss_i))$ for $i=1,\ldots,\ell$. As each $\alpha^{-1}(\Ss_i)$ is an open semialgebraic subset of $[0,1]$ and $\Ss_i\cap\Ss_j=\varnothing$ if $i\neq j$, we deduce $t_i\neq t_j$ if $i\neq j$. We reorder the indices $i=1,\ldots,\ell$ in such a way that $i<j$ if $t_i<t_j$. There exists $\veps>0$ such that $\alpha((t_i-\veps,t_i))\subset\Ss_j$ for some $1\leq j<i$ and $\alpha((t_i,t_i+\veps))\subset\Ss_i$ for each $i=2,\ldots,\ell$. Consequently, there exists a Nash bridge inside $\Ss$ between $\Ss_i$ and some $\Ss_j$ with $1\leq j\leq i-1$ for $i=2,\ldots,\ell$.

Choose a sequence of semialgebraic sets $\Tt_1,\ldots,\Tt_r$ such that $\{\Ss_1,\ldots,\Ss_\ell\}=\{\Tt_1,\ldots,\Tt_r\}$ and for each index $i=1,\ldots,r-1$ there exists a Nash bridge between $\Tt_i$ and $\Tt_{i+1}$. As $\Ss$ is pure dimensional, $\Ss=\cl(\Reg(\Ss))\cap\Ss=\bigcup_{i=1}^r\cl(\Tt_i)\cap\Ss$. If $x,y\in\Ss$, there exist indices $i,j$ such that $x\in\cl(\Tt_i)$ and $y\in\cl(\Tt_j)$. We may assume $i<j$ and we pick points $x_k\in\Tt_k$ for $k=i+1,\ldots,j-1$ and write $x_i:=x$ and $x_j:=y$. By Main Theorem \ref{nashsmart} there exists a Nash path $\alpha:[0,1]\to\Ss$ such that $\alpha(0)=x$ and $\alpha(1)=y$. Thus, $\Ss$ is connected by Nash paths and consequently by analytic paths, as required.
\end{proof}

\section{Polynomial paths inside piecewise linear semialgebraic sets}\label{s4}

In this section we prove Main Theorem \ref{plcase}, that is, we revisit Main Theorem \ref{nashsmart} for the piecewise linear (PL) case: {\em the involved semialgebraic sets are the interiors of convex polyhedra of dimension $n$}. Due to the maximality of the dimension of the convex polyhedra, we are under the hypothesis of Theorem \ref{smart} and the obtained `smart' path can be chosen polynomial. In order to get better bounds for the degrees of these polynomial paths: (1) we state a (polynomial) curve selection lemma for convex polyhedra that involves degree $3$ cuspidal curves (Lemma \ref{cuspidal}), and (2) we prove that the simplex polynomial paths that connect two convex polyhedra (whose union is connected by analytic paths) are moment curves (Theorem \ref{mc}).

\subsection{Double Nash curve selection lemma for PL semialgebraic sets.}
In order to lighten the presentation we first find a simplified version of Lemma \ref{doublecurve} for convex polyhedra (Figure \ref{fig3}). Denote $\R[\x]:=\R[\x_1,\ldots,\x_n]$. Given a polynomial $h\in\R[\x]$ of degree $1$, denote $\vec{h}:=h-h(0)$, which is a linear form.

\begin{lem}[Cuspidal curve]\label{cuspidal}
Let $\pol\subset\R^n$ be an $n$-dimensional convex polyhedron and let $p\in\pol$. Assume that $p$ is the origin and the point $e_1:=(1,0,\ldots,0)\in\Int(\pol)$. Consider the polynomial map $\alpha:\R\to\R^n,\ t\mapsto(t^2,t^3,0,\ldots,0)$. Then there exists $\veps>0$ such that $\alpha([-\veps,\veps])\subset\Int(\pol)\cup\{p\}$.
\end{lem}
\begin{center}
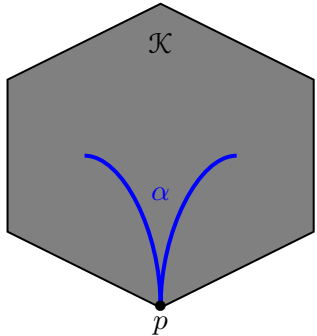
\begin{figure}[ht]
\begin{tikzpicture}[scale=1]
\draw[color=black,line width=1.5pt] (2,0) -- (4,1) -- (4,3) -- (2,4) -- (0,3) -- (0,1) -- (2,0);
\draw[fill=gray!100,opacity=0.5,draw=none] (2,0) -- (4,1) -- (4,3) -- (2,4) -- (0,3) -- (0,1) -- (2,0);

\draw[color=blue,line width=1.5pt] (1,2) .. controls (1.5,2) and (2,1) .. (2,0) .. controls (2,1) and (2.5,2) .. (3,2);
\draw (2,0) node{$\bullet$};
\draw (2,3.5) node{$\pol$};
\draw (2,1.5) node{{\color{blue}$\alpha$}};
\draw (2,-0.25) node{$p$};

\end{tikzpicture}
\caption{Cuspidal curve of Lemma \ref{cuspidal}.\label{fig3}}
\vspace{-1.75em}
\end{figure}
\end{center}

\begin{proof}
Let $h_1,\ldots, h_m\in\R[\x]$ be polynomials of degree $1$ such that $\pol:=\{h_1\geq0,\ldots,h_m\geq0\}$. As $e_1\in\Int(\pol)$, we have $h_k(e_1)>0$ for $k=1,\ldots,m$. Write $\vec{h}_k:=h_k(\x)-h_k(0)$, which is a linear form. Observe that
\begin{align*}
h_k(e_1)&=h_k(0)+\vec{h}_k(e_1)>0,\\
h_k(t^2,t^3,0,\ldots,0)&=h_k(0)+t^2\vec{h}_k(1,t,0,\ldots,0).
\end{align*}
We distinguish two cases:

\noindent{\sc Case 1.} $\vec{h}_k(e_1)>0$ (and $h_k(0)\geq0$). As $f_k:\R\to\R,\ t\mapsto\vec{h}_k(1,t,0,\ldots,0)=\vec{h}_k(e_1)+t\vec{h}_k(0,1,0,\ldots,0)$ is continuous and $f_k(0)=\vec{h}_k(e_1)>0$, there exists $\veps_k>0$ such that if $|t|<\veps_k$, then $f_k(t)>0$. As $h_k(0)\geq0$,
$$
h_k(t^2,t^3,0,\ldots,0)=h_k(0)+t^2\vec{h}_k(1,t,0,\ldots,0)>0\quad\text{if $0<|t|<\veps_k$.}
$$

\noindent{\sc Case 2.} $\vec{h}_k(e_1)\leq0$. Then $h_k(0)>0$. As $g_k:\R\to\R,\ t\mapsto h_k(0)+t^2\vec{h}_k(1,t,0,\ldots,0)$ is continuous and $h_k(0)>0$, there exists $\veps_k>0$ such that if $|t|<\veps_k$, then $h_k(t^2,t^3,0,\ldots,0)>0$.

To finish it is enough to take $\veps:=\min\{\veps_1,\ldots,\veps_m\}>0$.
\end{proof}

\subsection{Moment bridges between convex polyhedra.}
We analyze next the structure of the simplest possible Nash bridges between two convex polyhedra such that their union is a semialgebraic set connected by analytic paths and, surprisingly, moment curves appear (Figure \ref{fig4}). 

\begin{thm}[Moment curves]\label{mc}
Let $\pol_1,\pol_2\subset\R^n$ be $n$-dimensional convex polyhedra such that $0\in\pol_1\cap\pol_2$ and $\Int(\pol_1)\cap\Int(\pol_2)=\varnothing$. Assume that there exists a Nash arc $\alpha:[-1,1]\to\pol_1\cup\pol_2$ such that $\alpha(0)=0$, $\alpha([-1,0))\subset\Int(\pol_1)$ and $\alpha((0,1])\subset\Int(\pol_2)$. Then there exist $e=1,2$, an integer $e\leq d\leq n$ and $\veps>0$ such that after an affine change of coordinates in $\R^n$ the polynomial arc $\beta:=(\beta_1,\ldots,\beta_n):[-\veps,\veps]\to\pol_1\cup\pol_2$ satisfies $\beta(0)=0$, 
$$
\beta_k(t)=\begin{cases}
t^{e+k-1}&\text{if $k=1,\ldots,d$,}\\ 
0&\text{if $k=d+1,\ldots,n$,}
\end{cases}
$$
$\beta([-\veps,0))\subset\Int(\pol_1)$ and $\beta((0,\veps])\subset\Int(\pol_2)$. 
\end{thm}

To prove Theorem \ref{mc} we need a preliminary result. Given a non-zero power series $\zeta:=\sum_{k\geq0}a_k\t^k\in\R[[\t]]$, we denote its {\em order with respecto to $\t$} with $\omega(\zeta):=\min\{k\geq0:\ a_k\neq0\}$. For completeness $\omega(0):=+\infty$.

\begin{lem}\label{nashpol}
Let $\pol\subset\R^n$ be an $n$-dimensional convex polyhedron that contains the origin and let $\alpha:=(\alpha_1,\ldots,\alpha_n):[-1,1]\to\R^n$ be a Nash arc such that $\alpha(0)=0$ and $\alpha((0,1])\subset\Int(\pol)$. Assume $k_i:=\omega(\alpha_i)<\omega(\alpha_{i+1})=:k_{i+1}$ for $i=1,\ldots,n$ and write $\alpha_i:=\t^{k_i}(a_i+\t\gamma_i)$ where $a_i\in\R\setminus\{0\}$ and $\gamma_i$ is a Nash series. Then the monomial map $\beta:=(\beta_1,\ldots,\beta_n):\R\to\R^n,\ t\mapsto(a_1t^{k_1},\ldots,a_nt^{k_n})$ satisfies $\beta((0,\veps])\subset\Int(\pol)$ for some $\veps>0$.
\end{lem}

\begin{center}
\begin{figure}[ht]
\begin{tikzpicture}[scale=1]
\draw[color=black,line width=1.5pt] (0,0) -- (2,0) -- (0,2) -- (1,2) -- (2,3) -- (1,4) -- (0,4) -- (0,0);
\draw[color=black,line width=1.5pt] (4,0) -- (2,0) -- (4,2) -- (3,2) -- (2,3) -- (3,4) -- (4,4) -- (4,0);

\draw[fill=gray!100,opacity=0.5,draw=none] (0,0) -- (2,0) -- (0,2) -- (1,2) -- (2,3) -- (1,4) -- (0,4) -- (0,0);
\draw[fill=gray!100,opacity=0.5,draw=none](4,0) -- (2,0) -- (4,2) -- (3,2) -- (2,3) -- (3,4) -- (4,4) -- (4,0);

\draw[color=red,line width=1.5pt] (1,0.5) parabola bend (2,0) (3,0.5);
\draw[color=red,line width=1.5pt] (0.5,1) .. controls (-0.15,2) and (-0.15,2) .. (0.5,3);
\draw[color=red,line width=1.5pt] (3.5,1) .. controls (4.15,2) and (4.15,2) .. (3.5,3);
\draw[color=red,line width=1.5pt] (1.5,3) -- (2.5,3);

\draw[dashed,color=black,line width=0.75pt] (5,0) -- (5,4);

\draw (2,0) node{$\bullet$};
\draw (0,2) node{$\bullet$};
\draw (4,2) node{$\bullet$};
\draw (2,3) node{$\bullet$};
\draw (0.5,0.25) node{$\pol_1$};
\draw (3.5,0.25) node{$\pol_2$};
\draw (2,0.5) node{{\color{red}$\beta$}};

\draw[fill=gray!100,opacity=0.5,draw=none] (6,2) -- (8,2) -- (6,3.5) -- (6,2);
\draw[fill=gray!100,opacity=0.5,draw=none] (8,2) -- (10,1) -- (10.5,2.5) -- (8,2);

\draw[color=black,line width=1pt] (6,2) -- (6.5,2.5) -- (8,2);
\draw[color=black,line width=1pt] (6,3.5) -- (6.5,2.5); 
\draw[color=black,line width=1pt] (6,2) -- (6,3.5) -- (8,2) -- (6,2);
\draw[color=black,line width=1pt] (8,2) -- (10,2) -- (10.5,2.5) -- (8,2); 
\draw[color=black,line width=1pt] (8,2) -- (10,1) -- (10.5,2.5); 
\draw[dashed,color=black,line width=1pt] (10,2) -- (10,1);

\draw[color=red,line width=1.5pt] (7,2.25) parabola bend (8,2) (8,2);
\draw[color=red,line width=1.5pt] (9,1.65) parabola bend (8,2) (8,2);

\draw (8,2) node{$\bullet$};
\draw (7,3.25) node{$\pol_1$};
\draw (9.25,1) node{$\pol_2$};
\draw (8,2.5) node{{\color{red}$\beta$}};

\draw[dashed,color=black,line width=1pt] (11.5,0) -- (11.5,4);

\draw[color=black,line width=1.5pt] (14.5,2) -- (12.5,2) -- (14.5,0) -- (14.5,2);
\draw[color=black,line width=1.5pt] (14.5,2) -- (12.5,2) -- (14.5,4) -- (14.5,2);

\draw[fill=gray!100,opacity=0.5,draw=none] (14.5,2) -- (12.5,2) -- (14.5,0) -- (14.5,2);
\draw[fill=gray!100,opacity=0.5,draw=none] (14.5,2) -- (12.5,2) -- (14.5,4) -- (14.5,2);

\draw[color=red,line width=1.5pt] (12.5,2) parabola bend (12.5,2) (13.5,2.5);
\draw[color=red,line width=1.5pt] (12.5,2) parabola bend (12.5,2) (13.5,1.5);

\draw (12.5,2) node{$\bullet$};
\draw (14,1.25) node{$\pol_{21}$};
\draw (14,2.75) node{$\pol_1$};
\draw (12.5,2.5) node{{\color{red}$\beta$}};

\end{tikzpicture}
\caption{Moment curves of Theorem \ref{mc}.\label{fig4}}
\vspace{-1.75em}
\end{figure}
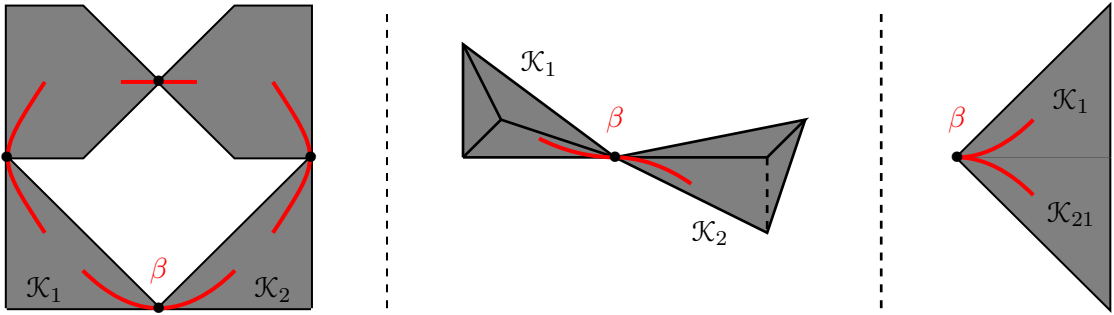
\end{center}

\begin{proof}
Write $\pol:=\{h_1\geq0,\ldots,h_m\geq0\}$ where $h_j\in\R[\x]$ are polynomials of degree one. As the origin belongs to $\pol$, we have $h_j(0)\geq0$. Write $\vec{h}_j:=h_j-h_j(0)$, where $\vec{h}_j$ is a linear form. Thus,
$$
h_j(\alpha_1,\ldots,\alpha_n)=h_j(0)+\vec{h}_j(\alpha_1,\ldots,\alpha_n).
$$
If $h_j(0)>0$, there exists $\veps_j>0$ such that $h_j(\beta_1(t),\ldots,\beta_n(t))>0$ if $0<t<\veps_j$, because 
$$
\vec{h}_j(\beta_1,\ldots,\beta_n)=\t^{k_1}\eta_j(\t)
$$
for some Nash series $\eta_j\in\R[[\t]]_{\rm alg}$. If $h_j(0)=0$, then $h_j=\vec{h}_j$. Write $h_j:=b_{jp_j}\x_{p_j}+\cdots+b_{jn}\x_n$ where $b_{jp_j}\neq0$. Then
\begin{align*}
h_j(\alpha_1,\ldots,\alpha_n)=h_j(\alpha_{p_j},\ldots,\alpha_n)=b_{jp_j}a_{p_j}\t^{k_{p_j}}(1+\t\tau_j),\\
h_j(\beta_1,\ldots,\beta_n)=h_j(\beta_{p_j},\ldots,\beta_n)=b_{jp_j}a_{p_j}\t^{k_{p_j}}(1+\t\theta_j),
\end{align*}
where $\tau_j,\theta_j\in\R[[\t]]_{\rm alg}$ are Nash series. As $h_j(\alpha_1,\ldots,\alpha_n)(t)>0$ for $0<t<1$, we deduce $b_{jp_j}a_{p_j}>0$, so there exists $\veps_j>0$ such that $h_j(\beta_1,\ldots,\beta_n)(t)>0$ if $0<t<\veps_j$.

To finish it is enough to take $\veps:=\min\{\veps_1,\ldots,\veps_m\}>0$.
\end{proof}

We are ready to prove Theorem \ref{mc}.

\begin{proof}[Proof of Theorem \em \ref{mc}]
The proof is conducted in several steps:

\noindent{\sc Step 0. Initial preparation.} As $\alpha:=(\alpha_1,\ldots,\alpha_n)$ is a Nash arc such that $\alpha(0)=0$, we may assume (after a linear change of coordinates) $\omega(\alpha_\ell)\leq\omega(\alpha_{\ell+1})$ for $\ell=1,\ldots,n-1$ and the previous inequality is strict if $\alpha_\ell\neq0$ for $\ell=1,\ldots,n-1$. Assume $\alpha_\ell=0$ exactly for $\ell=s+1,\ldots,n$. The tangent to $\alpha$ at $t=0$ is the line $\{\x_2=0,\ldots,\x_n=0\}$. Consider the intersections $\pol_i':=\pol_i\cap\{\x_{s+1}=0,\ldots,\x_n=0\}$, which are non-empty $s$-dimensional convex polyhedra for $i=1,2$ such that $\alpha([-\veps,0))\subset\Int(\pol_1')$ and $\alpha((0,\veps])\subset\Int(\pol_2')$. The previous assertion holds because $\alpha([-\veps,0))\subset\Int(\pol_1)$, $\alpha((0,\veps])\subset\Int(\pol_2)$ and $\alpha([-\veps,\veps])\subset\{\x_{s+1}=0,\ldots,\x_n=0\}$. By Lemma \ref{nashpol} and after a new linear change of coordinates we may assume $\alpha_\ell:=t^{k_\ell}$ for $\ell=1,\ldots,s$, $k_\ell<k_{\ell+1}$ for $\ell=1,\ldots,s-1$ and $\alpha_\ell=0$ for $\ell=s+1,\ldots,n$. 

Write $\pol_i:=\{h_{i1}\geq0,\ldots,h_{ir}\geq0\}$ where $h_{ij}\in\R[\x]$ are polynomials of degree $1$ and recall that $\Int(\pol_i)=\{h_{i1}>0,\ldots,h_{ir}>0\}$. 

\noindent{\sc Step 1. First modification of the Nash arc $\alpha$.}
We claim: \em we may assume $\omega(\alpha_1)$ is either $1$ (if $k_1$ is odd) or $2$ (if $k_1$ is even)\em.

As $\alpha(t)\subset\Int(\pol_i)$ for $(-1)^it>0$ small enough, each $h_{ij}(\alpha(t))>0$ if $(-1)^it>0$ is small enough for $i=1,2$. Write 
$$
e:=\begin{cases}
1&\text{if $k_1$ is odd,}\\
2&\text{if $k_1$ is even.}
\end{cases}
$$
We claim: {\em $\alpha^*:\R\to\R^n,\ t\mapsto(t^{e},t^{k_2-k_1+e},\ldots,t^{k_s-k_1+e},0,\ldots,0)$ is the monomial map we are looking for in this step.} Let us check: {\em $\alpha^*$ satisfies the inequalities defining $\Int(\pol_i)$ for $(-1)^it>0$ small enough and $i=1,2$. In addition, $h_{ij}(t^e,0,\ldots,0)\geq0$ for $(-1)^it>0$ if $i=1,2$ and $j=1,\ldots,r$.}

Fix any pair $(i,j)$. If $h_{ij}(0)>0$, there is nothing to prove, so we assume $h_{ij}(0)=0$. We have
$$
h_{ij}(\t^{k_1},\ldots,\t^{k_s},0,\ldots,0)=\t^{k_1-e}h_{ij}(\t^{e},\t^{k_2-k_1+e},\ldots,\t^{k_s-k_1+e},0,\ldots,0).
$$ 
As $k_1-e$ is even, $h_{ij}(t^{e},t^{k_2-k_1+e},\ldots,t^{k_s-k_1+e},0,\ldots,0)>0$ for $(-1)^it>0$ small enough. We deduce considering its Taylor expansion at $0$ that $h_{ij}(t^e,0,\ldots,0)\geq0$ for $(-1)^it>0$, because $k_\lambda-k_1>0$ for $\lambda=2,\ldots,s$. Thus, after substituting $\alpha$ by $\alpha^*$, we can suppose $\alpha:=(\t^{e},\t^{k_2'},\ldots,\t^{k_s'},0,\ldots,0)$, where $k_\lambda':=k_\lambda-k_1+e$ and $e<k_2'<\cdots<k_s'$. In the following we denote $k_\lambda'$ with $k_\lambda$ to lighten notation.

\noindent{\sc Step 2. Second modification of the Nash arc $\alpha$.}
Let us check next: {\em After a linear change of coordinates we may assume either $s=1$ and $e=1$ or $s\geq2$ and $k_2=e+1$\em}. 

Pick any pair $(i,j)$. If $h_{ij}(t^e,0,\ldots,0)>0$ for $(-1)^it>0$, there exists $\eta\in(0,1)$ (valid for each pair $(i,j)$ in this situation) such that if $(c_2,\ldots,c_n)\in\R^{n-1}$ and each $|c_k|\leq\eta$, then $h_{ij}(t^e,c_2t^e,\ldots,c_nt^e)>0$. Otherwise, $h_{ij}(t^e,0,\ldots,0)=0$ for $(-1)^it>0$, so $h_{ij}$ is a linear form that does not depend on $\x_1$. 

Next, we distinguish two cases:

\noindent{\sc Case 1}. {\em $k_2-e$ is even}. We check first: {\em If $\alpha:\R\to\R^n,\ t\mapsto(t^{e},t^{k_2},\ldots,t^{k_s},0,\ldots,0)$ is a monomial map such that $\alpha(t)\in\Int(\pol_i)$ for $(-1)^it>0$ small enough, then
$$
\alpha^*:\R\to\R^n,\ t\mapsto(t^{e},\eta t^{e},\eta t^{k_3-k_2+e},\ldots,\eta t^{k_s-k_2+e},0,\ldots,0)
$$ 
is a monomial map such that $\alpha^*(t)\in\Int(\pol_i)$ for $(-1)^it>0$ small enough}.

Pick any pair $(i,j)$. If $h_{ij}(t^e,0,\ldots,0)>0$ for $(-1)^it>0$, then
$$
h_{ij}(t^{e},\eta t^{e},\eta t^{k_3-k_2+e},\ldots,\eta t^{k_s-k_2+e},0,\ldots,0)>0
$$
for $0<(-1)^it<\eta$. If $h_{ij}$ is a linear form that does not depend on $\x_1$, then
\begin{multline*}
0<h_{ij}(\eta t^{k_2},\ldots,\eta t^{k_s},0,\ldots,0)=\eta t^{k_2-e}h_{ij}(t^e,t^{k_3-k_2+e},\ldots,t^{k_s-k_2+e},0,\ldots,0)\\
=t^{k_2-e}h_{ij}(t^{e},\eta t^{e},\eta t^{k_3-k_2+e},\ldots,\eta t^{k_s-k_2+e},0,\ldots,0)
\end{multline*}
for $0<(-1)^it<\eta$ small enough. After substituting $\alpha$ by $\alpha^*$, we suppose 
$$
\alpha:=(\t^{e},\eta\t^e,\eta\t^{k_2'},\ldots,\eta\t^{k_{s-1}'},0,\ldots,0),
$$
where $k_\lambda':=k_{\lambda+1}-k_2+e$ for $\lambda=2,\ldots,s-1$ and $e<k_2'<\cdots<k_{s-1}'$. We denote $k_\lambda'$ with $k_\lambda$ to lighten notation. After a linear change of coordinates, we may assume $\alpha:=(\t^e,\t^{k_2},\ldots,\t^{k_{s-1}},0,\ldots,0)$. 

Now, if $k_2-e$ is again even, we repeat the procedure developed in this {\sc Case 1} and proceed recursively. After finitely many steps, either the corresponding $k_2-e$ is odd or $\alpha(\t)=(\t^e,0,\ldots,0)$ where $e=1,2$. If $e=2$, then $\alpha(t)=\alpha(-t)=(t^2,0,\ldots,0)\in\Int(\pol_1)\cap\Int(\pol_2)=\varnothing$ for $t>0$ small enough, which is a contradiction. Consequently, in this latter case $e=1$.

\noindent{\sc Case 2}. {\em $k_2-e$ is odd}. We prove first: {\em If $\alpha:\R\to\R^n,\ t\mapsto(t^{e},t^{k_2},\ldots,t^{k_s},0,\ldots,0)$ is a monomial map such that $\alpha(t)\in\Int(\pol_i)$ for $(-1)^it>0$ small enough, then $\alpha^*:\R\to\R^n,\ t\mapsto(t^{e},t^{e+1},t^{k_3-k_2+e+1},\ldots,t^{k_s-k_2+e+1},0,\ldots,0)$ is a monomial map such that $\alpha(t)\in\Int(\pol_i)$ for $(-1)^it>0$ small enough}.

We have $k_2-e-1$ is even and pick any pair $(i,j)$. If $h_{ij}(t^e,0,\ldots,0)>0$ for $(-1)^it>0$, then
$$
h_{ij}(t^{e},t^{e+1},t^{k_3-k_2+e+1},\ldots,t^{k_s-k_2+e+1},0,\ldots,0)>0
$$
for $0<(-1)^it<\eta$. If $h_{ij}$ is a linear form that does not depend on $\x_1$, then
\begin{equation*}
\begin{split}
0&<h_{ij}(t^e,t^{k_2},\ldots,t^{k_s},0,\ldots,0)=h_{ij}(t^{k_2},\ldots,t^{k_s},0,\ldots,0)\\
&=t^{k_2-e-1}h_{ij}(t^{e+1},t^{k_3-k_2+e+1},\ldots,t^{k_s-k_2+e+1},0,\ldots,0)\\
&=t^{k_2-e-1}h_{ij}(t^e,t^{e+1},t^{k_3-k_2+e+1},\ldots,t^{k_s-k_2+e+1},0,\ldots,0)
\end{split}
\end{equation*}
for $0<(-1)^it<\eta$ small enough. As $k_2-e-1$ is even, 
$$
h_{ij}(t^e,t^{e+1},t^{k_3-k_2+e+1},\ldots,t^{k_s-k_2+e+1},0,\ldots,0)>0
$$ 
for $0<(-1)^it<\eta$ small enough. Thus, we can suppose 
$$
\alpha:=(\t^{e},\t^{e+1},\t^{k_3'}\ldots,\t^{k_s'},0,\ldots,0),
$$ 
where $k_\lambda':=k_\lambda-k_2+e+1$ for $\lambda=3,\ldots,s$ and $e+1<k_3'<\cdots<k_s'$. Again, we denote $k_\lambda'$ with $k_\lambda$ to lighten notation.

\noindent{\sc Step $\ell+1$. Recursive modification of the Nash arc $\alpha$.} Suppose $\ell\geq2$ and
$$
\alpha:\R\to\R^n,\ t\mapsto(t^e,t^{e+1},\ldots,t^{e+\ell-1},t^{k_{\ell+1}},\ldots,t^{k_s},0,\ldots,0)
$$ 
is a monomial map such that $e+\ell-1<k_{\ell+1}<\ldots<k_s$ and $\alpha(t)\in\Int(\pol_i)$ for $(-1)^it>0$ small enough. Let us check: {\em After a linear change of coordinates, we may assume that either $$
\alpha:\R\to\R^n,\ (t^e,t^{e+1},\ldots,t^{e+\ell-1},0,\ldots,0)
$$ 
satisfies $\alpha(t)\in\Int(\pol_i)$ for $(-1)^it>0$ small enough or there exist $s'\leq s$ and positive integers $e+\ell<k_{\ell+2}'<\ldots<t^{k'_{s'}}$ such that
$$
\alpha:\R\to\R^n,\ (t^e,t^{e+1},\ldots,t^{e+\ell-1},t^{e+\ell},t^{k_{\ell+2}'},\ldots,t^{k'_{s'}},0,\ldots,0)
$$ 
satisfies $\alpha(t)\in\Int(\pol_i)$ for $(-1)^it>0$ small enough.} 

Fix a pair of indices $(i,j)$. We have 
$$
h_{ij}(t^e,t^{e+1},\ldots,t^{e+\ell-1},t^{k_{\ell+1}},\ldots,t^{k_s},0,\ldots,0)>0
$$
for $(-1)^it>0$ small enough. We deduce considering its Taylor expansion at $0$ that 
$$
h_{ij}(t^e,t^{e+1},\ldots,t^{e+\ell-1},0,\ldots,0)\geq0
$$ 
for $(-1)^it>0$ small enough, because $k_\lambda-(e+\ell-1)>0$ for $\lambda=\ell+1,\ldots,s$. If 
$$
h_{ij}(t^e,t^{e+1},\ldots,t^{e+\ell-1},0,\ldots,0)>0
$$ 
for $(-1)^it>0$ small enough, there exists an integer $1\leq m_{ij}\leq\ell$ such that $h_{ij}$ does not depend on $\x_1,\ldots,\x_{m_{ij}-1}$ and $h_{ij}(\t^{e+m_{ij}-1},0,\ldots,0)>0$ for $(-1)^it>0$ small enough. Thus, there exists $\eta\in(0,1)$ (valid for each pair $(i,j)$ in this situation) such that if $(c_{m_{ij}+1},\ldots,c_n)\in\R^{n-m_{ij}}$ and each $|c_k|\leq\eta$, then 
\begin{multline*}
h_{ij}(t^e,t^{e+1},\ldots,t^{e+m_{ij}-1},c_{m_{ij}+1}t^{e+m_{ij}-1},\ldots,c_nt^{e+m_{ij}-1},0,\ldots,0)\\
=h_{ij}(t^{e+m_{ij}-1},c_{m_{ij}+1}t^{e+m_{ij}-1},\ldots,c_nt^{e+m_{ij}-1},0,\ldots,0)>0
\end{multline*}
for $(-1)^it>0$ small enough. 

Otherwise, $h_{ij}(t^e,t^{e+1},\ldots,t^{e+\ell-1},0,\ldots,0)=0$ for $(-1)^it>0$ small enough, so $h_{ij}$ is a linear form that does not depend on $\x_1,\ldots,\x_\ell$. 

Next, we distinguish two cases:

\noindent{\sc Case 1}. {\em $k_{\ell+1}-(e+\ell-1)$ is even}. We check first: {\em 
$$
\alpha^*:\R\to\R^n,\ t\mapsto(t^e,t^{e+1},\ldots,t^{e+\ell-1},\eta t^{e+\ell-1},\eta t^{k_{\ell+2}-k_{\ell+1}+e+\ell-1},\ldots,\eta t^{k_s-k_{\ell+1}+e+\ell-1},0,\ldots,0)
$$ 
is a monomial map such that $\alpha(t)\in\Int(\pol_i)$ for $(-1)^it>0$ small enough}.

Pick any pair $(i,j)$. If $h_{ij}(t^e,t^{e+1},\ldots,t^{e+\ell-1},0,\ldots,0)>0$ for $(-1)^it>0$ small enough, then
$$
h_{ij}(t^e,t^{e+1},\ldots,t^{e+\ell-1},\eta t^{e+\ell-1},\eta t^{k_{\ell+2}-k_{\ell+1}+e+\ell-1},\ldots,\eta t^{k_s-k_{\ell+1}+e+\ell-1},0,\ldots,0)>0
$$
for $0<(-1)^it<\eta$. If $h_{ij}$ is a linear form that does not depend on $\x_1,\ldots,\x_\ell$, then
{\begin{equation*}
\begin{split}
0&<\eta h_{ij}(t^{k_{\ell+1}},\ldots,t^{k_s},0,\ldots,0)\\
&=\eta t^{k_{\ell+1}-(e+\ell-1)}h_{ij}(t^{e+\ell-1},t^{k_{\ell+2}-k_{\ell+1}+e+\ell-1},\ldots,t^{k_s-k_{\ell+1}+e+\ell-1},0,\ldots,0)\\
&=t^{k_{\ell+1}-(e+\ell-1)}h_{ij}(t^e,t^{e+1},\ldots,t^{e+\ell-1},\eta t^{e+\ell-1},\eta t^{k_{\ell+2}-k_{\ell+1}+e+\ell-1},\ldots,\eta t^{k_s-k_{\ell+1}+e+\ell-1},0,\ldots,0)
\end{split}
\end{equation*}}
for $0<(-1)^it<\eta$ small enough. As $k_{\ell+1}-(e+\ell-1)$ is even,
$$
h_{ij}(t^e,t^{e+1},\ldots,t^{e+\ell-1},\eta t^{e+\ell-1},\eta t^{k_{\ell+2}-k_{\ell+1}+e+\ell-1},\ldots,\eta t^{k_s-k_{\ell+1}+e+\ell-1},0,\ldots,0)>0.
$$
Thus, after substituting $\alpha$ by $\alpha^*$, we can suppose 
$$
\alpha:=(\t^e,\t^{e+1},\ldots,\t^{e+\ell-1},\eta\t^{e+\ell-1},\eta\t^{k_{\ell+1}'},\ldots,\eta\t^{k_{s-1}'},0,\ldots,0),
$$
where $k_\lambda':=k_{\lambda+1}-k_{\ell+1}+e+\ell-1$ for $\lambda=\ell+1,\ldots,s-1$ and $e+\ell-1<k_{\ell+1}'<\cdots<k_{s-1}'$. We denote $k_\lambda'$ with $k_\lambda$ to lighten notation. After a linear change of coordinates, we may assume $\alpha:=(\t^e,\t^{e+1},\ldots,\t^{e+\ell-1},\t^{k_{\ell+1}},\ldots,\t^{k_{s-1}},0,\ldots,0)$. 

Now, if $k_{\ell+1}-(e+\ell-1)$ is again even, we repeat the procedure developed in this {\sc Case 1} and proceed recursively. After finitely many steps, either the corresponding $k_{\ell+1}-(e+\ell-1)$ is odd or $\alpha(\t)=(\t^e,\t^{e+1},\ldots,\t^{e+\ell-1},0,\ldots,0)$ where $e=1,2$ and $\ell\geq2$.

\noindent{\sc Case 2}. {\em $k_{\ell+1}-(e+\ell-1)$ is odd}. We prove first: {\em 
$$
\alpha^*:\R\to\R^n,\ t\mapsto(t^{e},t^{e+1},\ldots,t^{e+\ell-1},t^{e+\ell},t^{k_{\ell+2}-k_{\ell+1}+e+\ell},\ldots,t^{k_s-k_{\ell+1}+e+\ell},0,\ldots,0)
$$ 
is a monomial map such that $\alpha(t)\in\Int(\pol_i)$ for $(-1)^it>0$ small enough}.

Observe that $k_{\ell+1}-(e+\ell)$ is even and pick any pair $(i,j)$. If $h_{ij}(t^e,t^{e+1},\ldots,t^{e+\ell-1},0,\ldots,0)>0$ for $(-1)^it>0$ small enough, then
$$
h_{ij}(t^{e},t^{e+1},\ldots,t^{e+\ell-1},t^{e+\ell},t^{k_{\ell+2}-k_{\ell+1}+e+\ell},\ldots,t^{k_s-k_{\ell+1}+e+\ell},0,\ldots,0)>0
$$
for $0<(-1)^it<\eta$. If $h_{ij}$ is a linear form that does not depend on $\x_1,\ldots,\x_\ell$, then
\begin{equation*}
\begin{split}
0&<h_{ij}(t^e,t^{e+1},\ldots,t^{e+\ell-1},t^{k_{\ell+1}},\ldots,t^{k_s},0,\ldots,0)=h_{ij}(t^{k_{\ell+1}},\ldots,t^{k_s},0,\ldots,0)\\
&=t^{k_{\ell+1}-(e+\ell)}h_{ij}(t^{e+\ell},t^{k_{\ell+2}-k_{\ell+1}+e+\ell},\ldots,t^{k_s-k_{\ell+1}+e+\ell},0,\ldots,0)\\
&=t^{k_{\ell+1}-(e+\ell)}h_{ij}(t^{e},t^{e+1},\ldots,t^{e+\ell-1},t^{e+\ell},t^{k_{\ell+2}-k_{\ell+1}+e+\ell},\ldots,t^{k_s-k_{\ell+1}+e+\ell},0,\ldots,0)
\end{split}
\end{equation*}
for $0<(-1)^it<\eta$ small enough. As $k_{\ell+1}-(e+\ell)$ is even, 
$$
h_{ij}(t^{e},t^{e+1},\ldots,t^{e+\ell-1},t^{e+\ell},t^{k_{\ell+2}-k_{\ell+1}+e+\ell},\ldots,t^{k_s-k_{\ell+1}+e+\ell},0,\ldots,0)>0
$$ 
for $0<(-1)^it<\eta$ small enough. Thus, we can suppose 
$$
\alpha:=(\t^{e},\t^{e+1},\ldots,\t^{e+\ell-1},\t^{e+\ell},\t^{k_{\ell+2}'},\ldots,\t^{k_s'},0,\ldots,0),
$$ 
where $k_\lambda':=k_\lambda-k_{\ell+1}+(e+\ell)$ and $e+\ell<k_{\ell+2}'<\cdots<k_s'$. Again, we denote $k_\lambda'$ with $k_\lambda$ to lighten notation.

\noindent{\sc Conclusion.} The process ends after finitely many steps providing the statement, as required.
\end{proof}

The following example supplies a pair of $n$-dimensional convex polyhedra in $\R^n$ with disjoint interiors and adherent to the origin for which the simplest monomial paths connecting their interiors analytically through the origin are moment paths.

\begin{examples}\label{sharp}
Denote $\x_{n+1}:=0$ and let $\x_1,\ldots,\x_n$ be variables. Consider for $\epsilon=0,1$ the convex polyhedra (Figure \ref{fig5})
\begin{align*}
\pol_1&:=\{\x_k\leq\x_{k-1},\ k=2,\ldots,n+1\}\cap\{\x_1\leq1\},\\
\pol_{2\epsilon}&:=\{(-1)^{k+\epsilon}\x_k\leq(-1)^{k-1+\epsilon}\x_{k-1},\ k=2,\ldots,n+1\}\cap\{(-1)^{1+\veps}\x_1\leq1\}.
\end{align*}

\begin{center}
\begin{figure}[ht]
\begin{tikzpicture}[scale=1]
\draw[color=black,line width=1.5pt] (0,-1) -- (2,-1) -- (0,1) -- (0,-1);
\draw[color=black,line width=1.5pt] (4,-1) -- (2,-1) -- (4,1) -- (4,-1);

\draw[fill=gray!100,opacity=0.5,draw=none] (0,-1) -- (2,-1) -- (0,1) -- (0,-1);
\draw[fill=gray!100,opacity=0.5,draw=none] (4,-1) -- (2,-1) -- (4,1) -- (4,-1);

\draw[color=red,line width=1.5pt] (1,-0.5) parabola bend (2,-1) (3,-0.5);
\draw (2,-1) node{$\bullet$};
\draw (0.5,-0.25) node{$\pol_{20}$};
\draw (3.5,-0.25) node{$\pol_1$};
\draw (2,-0.5) node{{\color{red}$\alpha$}};

\draw[dashed,color=black,line width=1pt] (5,-2) -- (5,2);

\draw[color=black,line width=1.5pt] (8,0) -- (6,0) -- (8,-2) -- (8,0);
\draw[color=black,line width=1.5pt] (8,0) -- (6,0) -- (8,2) -- (8,0);

\draw[fill=gray!100,opacity=0.5,draw=none] (8,0) -- (6,0) -- (8,-2) -- (8,0);
\draw[fill=gray!100,opacity=0.5,draw=none] (8,0) -- (6,0) -- (8,2) -- (8,0);

\draw[color=red,line width=1.5pt] (6,0) parabola bend (6,0) (7,0.5);
\draw[color=red,line width=1.5pt] (6,0) parabola bend (6,0) (7,-0.5);

\draw (6,0) node{$\bullet$};
\draw (7.5,-0.75) node{$\pol_{21}$};
\draw (7.5,0.75) node{$\pol_1$};
\draw (6,0.5) node{{\color{red}$\alpha$}};

\end{tikzpicture}
\caption{Polyhedra $\pol_1$ and $\pol_{2\epsilon}$ of Example \ref{sharp}\label{fig5}}
\vspace{-1.75em}
\end{figure}
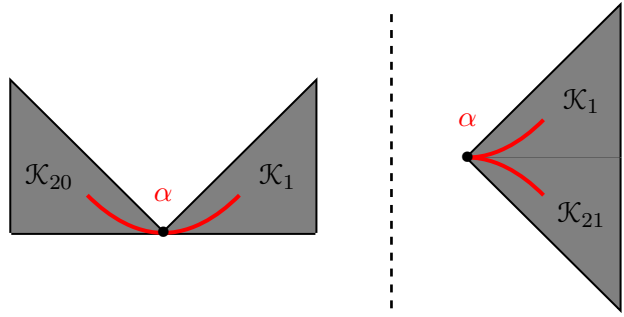
\end{center}

We have
\begin{align*}
\Int(\pol_1)&:=\{\x_k<\x_{k-1},\ k=2,\ldots,n+1\}\cap\{\x_1<1\},\\
\Int(\pol_{2\epsilon})&:=\{(-1)^{k+\epsilon}\x_k<(-1)^{k-1+\epsilon}\x_{k-1},\ k=2,\ldots,n+1\}\cap\{(-1)^{1+\veps}\x_1<1\}.
\end{align*}
One can check that 
$$
\pol_1\cap\pol_{2\epsilon}=
\begin{cases}
\{0\}&\text{if $\epsilon=0$,}\\
\{0\leq\x_1\leq1,\x_k=0:\ k=2,\ldots,n\}&\text{if $\epsilon=1$}
\end{cases}
$$ 
and $\Int(\pol_1)\cap\Int(\pol_{2\epsilon})=\varnothing$. Consider a monomial map $\alpha_\epsilon:\R\to\R^n,\ t\mapsto(a_1t^{k_1},\ldots,a_nt^{k_n})$ for some integers $k_i\geq1$ (so $\alpha_\epsilon(0)=0$) and some $a_1,\ldots,a_n\in\R$ (see Lemma \ref{nashpol}). Assume there exists $\delta>0$ such that $\alpha_\epsilon((0,\delta])\subset\Int(\pol_1)$ and $\alpha_\epsilon([-\delta,0))\subset\Int(\pol_{2\epsilon})$. Consequently,
\begin{align}
&1>a_1t^{k_1}>\cdots>a_\ell t^{k_\ell}>\cdots>a_nt^{k_n}>0\label{des1}\\
&1>(-1)^{1+\epsilon+k_1}a_1(-t)^{k_1}>\cdots>(-1)^{\ell+\epsilon+k_\ell}a_\ell(-t)^{k_\ell}>\cdots>(-1)^{n+\epsilon+k_n}a_n(-t)^{k_n}>0\label{des2}
\end{align}
where $0<t\leq\delta$ in \eqref{des1} and $0<-t\leq\delta$ in \eqref{des2}. Thus, each $a_\ell>0$, $k_\ell\leq k_{\ell+1}$ for $\ell=1,\ldots,n-1$ and $\ell+\epsilon+k_\ell$ is even for each $\ell=1,\ldots,n$, so the parity of $k_\ell$ coincides with the one of $\ell+\epsilon$ (so $k_\ell k_{\ell+1}$ is odd for $\ell=1,\ldots,n-1$). The minimal possible choice for the exponents is $k_\ell=\ell+\epsilon$ for $\ell=1,\ldots,n$ and $\epsilon=0,1$, so we obtain the moment curve $\alpha_\epsilon:\R\to\R^n,\ t\mapsto(a_1t^{1+\epsilon},a_2t^{2+\epsilon},\ldots,a_nt^{n+\epsilon})$ for some $a_1,\ldots,a_n>0$ and $\epsilon=0,1$.\hfill$\sqbullet$
\end{examples}

\subsection{Proof of Main Theorem \ref{plcase}.}
As we are working with convex polyhedra, the polynomial paths joining polynomial arcs and polynomial bridges can be chosen to be segments. For each $a\in\R^n$ and $\veps>0$ denote the open ball of center $a$ and radius $\veps>0$ with $\Bb_n(a,\veps)$. In order to compute the distance of a segment inside an $n$-dimensional convex polyhedron $\pol\subset\R^n$ to its exterior $\R^n\setminus\pol$ (or equivalently to its boundary $\partial\pol$) we present the following result.

\begin{lem}\label{segment}
Let $\Cc\subset\R^n$ be a convex set (that spans $\R^n$) and $x,y\in\Cc$. Let $\Ss$ be the segment that connects $x$ and $y$. Then
$$
\dist(\Ss,\R^n\setminus\Int(\Cc))=\min\{\dist(x,\R^n\setminus\Int(\Cc)),\dist(y,\R^n\setminus\Int(\Cc))\}.
$$
\end{lem}
\begin{proof}
If either $x$ or $y$ belong to $\partial\Cc$, then $\dist(\Ss,\R^n\setminus\Int(\Cc))=0$ and the equality in the statement holds. Assume $0<\veps:=\dist(x,\R^n\setminus\Int(\Cc))\leq\dist(y,\R^n\setminus\Int(\Cc))$ and observe that $\Bb_n(x,\veps),\Bb_n(y,\veps)\subset\Int(\Cc)$. We claim: $\bigcup_{z\in\Ss}\Bb_n(z,\veps)\subset\Int(\Cc)$. Once this is proved, the equality in the statement follows straightforwardly.

Let $z\in\Ss$ and $p\in\Bb_n(z,\veps)$. Let $t\in[0,1]$ be such that $z=tx+(1-t)y$. Consider the points $p_1:=x+(p-z)$ and $p_2:=y+(p-z)$. As $p\in\Bb_n(z,\veps)$, we have $\|p-z\|<\veps$, so $p_1\in\Bb_n(x,\veps)\subset\Int(\Cc)$ and $p_2\in\Bb_n(y,\veps)\subset\Int(\Cc)$. Thus, 
$$
p=t(x+(p-z))+(1-t)(y+(p-z))=tp_1+(1-t)p_2\in\Int(\Cc),
$$
as required.
\end{proof}

We are ready to prove Main Theorem \ref{plcase} by simplifying the proof of Main Theorem \ref{nashsmart}. The \em degree of a polynomial map \em $\alpha:\R\to\R^n$ is the maximum of the degrees of its components.

\begin{proof}[Proof of Main Theorem \em \ref{plcase}]
By Lemma \ref{cuspidal} for each $t_i$ there exist a polynomial path $\beta_i:\R\to\R^n$ of degree $e_i\leq3$ and $\delta_i>0$ such that: $\beta_i(t_i)=p_i$ and $\Lambda_i:=\beta_i([t_i-\delta_i,t_i+\delta_i])\subset\Ss_i\cup\{p_i\}$ is contained in a small enough ball centered at $p_i$. 

\begin{center}
\begin{figure}[ht]
\begin{tikzpicture}[scale=1.5]
\draw[color=black,line width=2pt] (0,0) -- (2,0) -- (0,2) -- (1,2) -- (2,3) -- (1,4) -- (0,4) -- (0,0);
\draw[color=black,line width=2pt] (4,0) -- (2,0) -- (4,2) -- (3,2) -- (2,3) -- (3,4) -- (4,4) -- (4,0);

\draw[fill=gray!100,opacity=0.5,draw=none] (0,0) -- (2,0) -- (0,2) -- (1,2) -- (2,3) -- (1,4) -- (0,4) -- (0,0);
\draw[fill=gray!100,opacity=0.5,draw=none](4,0) -- (2,0) -- (4,2) -- (3,2) -- (2,3) -- (3,4) -- (4,4) -- (4,0);

\draw[color=blue,line width=2pt] (0,0) -- (4,0) -- (4,4) -- (2,3) -- (0,4) -- (0,0);

\draw[color=red,line width=2pt] (1,0.5) parabola bend (2,0) (3,0.5);
\draw[color=red,line width=2pt] (0.5,1) .. controls (-0.15,2) and (-0.15,2) .. (0.5,3);
\draw[color=red,line width=2pt] (3.5,1) .. controls (4.15,2) and (4.15,2) .. (3.5,3);
\draw[color=red,line width=2pt] (1.5,3) -- (2.5,3);

\draw[color=red,line width=2pt] (0.25,1) .. controls (0.25,1) and (0.5,0.5) .. (0,0);
\draw[color=red,line width=2pt] (1,0.25) .. controls (1,0.25) and (0.5,0.5) .. (0,0);

\draw[color=red,line width=2pt] (3.75,1) .. controls (3.75,1) and (3.5,0.5) .. (4,0);
\draw[color=red,line width=2pt] (3,0.25) .. controls (3,0.25) and (3.5,0.5) .. (4,0);

\draw[color=red,line width=2pt] (0.25,3) .. controls (0.25,3) and (0.5,3.5) .. (0,4);
\draw[color=red,line width=2pt] (1,3.75) .. controls (1,3.75) and (0.5,3.5) .. (0,4);

\draw[color=red,line width=2pt] (3.75,3) .. controls (3.75,3) and (3.5,3.5) .. (4,4);
\draw[color=red,line width=2pt] (3,3.75) .. controls (3,3.75) and (3.5,3.5) .. (4,4);

\draw[color=red,line width=2pt] (0.2,0.15) -- (1.45,0.15);
\draw[color=red,line width=2pt] (2.55,0.15) -- (3.8,0.15);

\draw[color=red,line width=2pt] (0.15,0.2) -- (0.15,1.55);
\draw[color=red,line width=2pt] (0.15,2.45) -- (0.15,3.8);

\draw[color=red,line width=2pt] (3.85,0.2) -- (3.85,1.55);
\draw[color=red,line width=2pt] (3.85,2.45) -- (3.85,3.8);

\draw[color=red,line width=2pt] (0.2,3.85) -- (1.75,3);
\draw[color=red,line width=2pt] (2.25,3) -- (3.8,3.85);

\draw[densely dashed,color=black,line width=1pt] (0,0) -- (0.2,0.15) -- (1.45,0.15) parabola bend (2,0) (2.55,0.15) -- (3.8,0.15) -- (4,0) -- (3.85,0.2) -- (3.85,1.55) .. controls (4.035,2) and (4.035,2) .. (3.85,2.45) -- (3.85,3.8) -- (4,4) -- (3.8,3.85) -- (2.25,3) -- (1.75,3) -- (0.2,3.85) -- (0,4) --(0.15,3.8) -- (0.15,2.45) .. controls (-0.035,2) and (-0.035,2) .. (0.15,1.55) -- (0.15,0.2) -- (0,0);

\draw[color=green,line width=2pt] (0,0) .. controls (0.4,0.5) and (0.4,0.5) .. (0.25,1.6) .. controls (0.15,1.75) and (0,1.5) .. (0,2) .. controls (0,2.5) and (0.15,2.25) .. (0.25,2.4) .. controls (0.4,3.5) and (0.4,3.5) .. (0,4) parabola bend (2,3) (4,4) .. controls (3.6,3.5) and (3.6,3.5) .. (3.75,2.4) .. controls (3.85,2.25) and (4,2.5) .. (4,2) .. controls (4,1.5) and (3.85,1.75) .. (3.75,1.6) .. controls (3.6,0.5) and (3.6,0.5) .. (4,0) .. controls (3.5,0.4) and (3.5,0.4) .. (2.4,0.25) .. controls (2.25,0.15) and (2.5,0) .. (2,0) .. controls (1.5,0) and (1.75,0.15) .. (1.6,0.25) .. controls (0.5,0.4) and (0.5,0.4) .. (0,0);

\draw (2,0) node{$\bullet$};
\draw (0,2) node{$\bullet$};
\draw (4,2) node{$\bullet$};
\draw (2,3) node{$\bullet$};
\draw (0,0) node{$\bullet$};
\draw (4,0) node{$\bullet$};
\draw (4,4) node{$\bullet$};
\draw (0,4) node{$\bullet$};

\draw (0.2,0.15) node{\scriptsize$\bullet$};
\draw (1.45,0.15) node{\scriptsize$\bullet$};
\draw (2.55,0.15) node{\scriptsize$\bullet$};
\draw (3.8,0.15) node{\scriptsize$\bullet$};
\draw (0.15,0.2) node{\scriptsize$\bullet$};
\draw (0.15,1.55) node{\scriptsize$\bullet$};
\draw (0.15,2.45) node{\scriptsize$\bullet$};
\draw (0.15,3.8) node{\scriptsize$\bullet$};
\draw (0.2,3.85) node{\scriptsize$\bullet$};

\draw (3.85,0.2) node{\scriptsize$\bullet$};
\draw (3.85,1.55) node{\scriptsize$\bullet$};
\draw (3.85,2.45) node{\scriptsize$\bullet$};
\draw (3.85,3.8) node{\scriptsize$\bullet$};
\draw (3.8,3.85) node{\scriptsize$\bullet$};

\draw (1.75,3) node{\scriptsize$\bullet$};
\draw (2.25,3) node{\scriptsize$\bullet$};

\draw (0.75,0.75) node{$\Ss_1$};
\draw (3.25,0.75) node{$\Ss_2$};
\draw (0.75,3) node{$\Ss_4$};
\draw (3.25,3) node{$\Ss_3$};
\draw (2,3.25) node{{\color{blue}$\beta$}};
\draw (2,0.25) node{{\color{red}$\gamma$}};
\draw (2,-0.25) node{{\color{green}$\alpha$}};

\end{tikzpicture}
\caption{Construction of the polygonal path $\beta$ (blue), the continuous piecewise polynomial path $\gamma$ (red and dashed black) and the polynomial path $\alpha$ (green).\label{fig6}}
\vspace{-1.75em}
\end{figure}
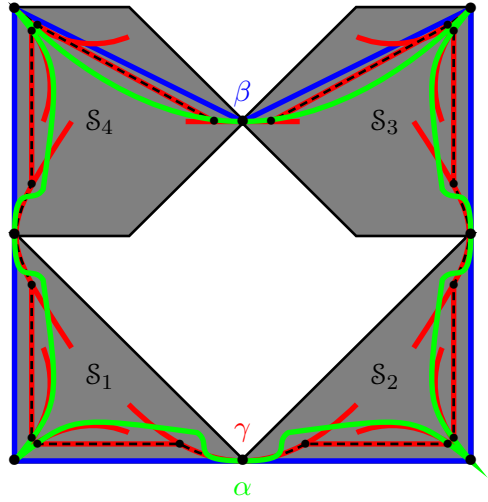
\end{center}

Fix $i=1,\ldots,r-1$ and recall that both $\Ss_i$ and $\Ss_{i+1}$ are the interiors of convex polyhedra of dimension $n$. Suppose first $\Ss_i\cap\Ss_{i+1}\neq\varnothing$. The intersection $\Ss_i\cap\Ss_{i+1}$ is the interior of a convex polyhedron of dimension $n$. By Lemma \ref{cuspidal} there exists a polynomial arc $\lambda_i:[s_i-\rho_i,s_i+\rho_i]\to\Ss_i\cap\Ss_{i+1}$ of degree $3\leq n+1$ such that $\lambda_i([s_i-\rho_i,s_i+\rho_i]\setminus\{\rho_i\})\subset(\Ss_i\cap\Ss_{i+1})\cup\{q_i\}$ and we substitute $\Gamma_i$ by the image of $\lambda_i$. Suppose next $\Ss_i\cap\Ss_{i+1}=\varnothing$. Let $\eta_i:[-1,1]\to\Ss_i\cup\Ss_{i+1}\cup\{q_i\}$ be a Nash parameterization of $\Gamma_i$ such that $\eta_i([-1,0))\subset\Ss_i$ and $\eta_i((0,1])\subset\Ss_{i+1}$. By Theorem \ref{mc} we can modify $\Gamma_i$ and after that it admits a polynomial parameterization $\lambda_i:[s_i-\rho_i,s_i+\rho_i]\to\Ss_i\cup\Ss_{i+1}\cup\{q_i\}$ of degree $d_i\leq n+1$, where $\rho_i>0$, $\lambda_i(s_i)=q_i$, $\lambda_i([s_i-\rho_i,s_i))\subset\Ss_i$ and $\lambda_i((s_i,s_i+\rho_i])\subset\Ss_{i+1}$. We choose each $\rho_i>0$ small enough to guarantee that $\Gamma_i$ is contained in a small enough ball centered at $q_i$. Denote $\tau_i:=t_i-\delta_i$, $\theta_i:=t_i+\delta_i$, $\xi_i:=s_i-\rho_i$ and $\zeta_i:=s_i+\rho_i$. We may assume
$$
0<\tau_1<t_1<\theta_1<\xi_1<s_1<\zeta_1<\tau_2<t_2<\theta_2<\cdots<\xi_{r-1}<s_{r-1}<\zeta_{r-1}<\tau_r<t_r<\theta_r<1.
$$

Let $\gamma:[0,1]\to\Ss\cup\{p_1,\ldots,p_r,q_1,\ldots,q_{r-1}\}\subset\R^n$ be a continuous piecewise polynomial path (Figure \ref{fig6}) such that: 
\begin{itemize}
\item[$\bullet$] $\gamma|_{(\tau_i,\theta_i)}=\beta_i|_{(\tau_i,\theta_i)}$ for $i=1,\ldots,r$ and $\gamma|_{(\xi_i,\zeta_i)}=\lambda_i|_{(\xi_i,\zeta_i)}$ for $i=1,\ldots,r-1$.
\item[$\bullet$] $\gamma|_{[\theta_i,\xi_i]}$ is an affine parameterization of the segment inside $\Ss_i$ that connects $\beta_i(\theta_i)$ with $\lambda_i(\xi_i)$ for $i=1,\ldots,r$.
\item[$\bullet$] $\gamma|_{[\zeta_i,\tau_{i+1}]}$ is an affine parameterization of the segment inside $\Ss_{i+1}$ that connects $\lambda_i(\zeta_i)$ with $\beta_{i+1}(\tau_{i+1})$ for $i=1,\ldots,r-1$.
\item[$\bullet$] $\gamma|_{[0,\tau_1]}$ and $\gamma|_{[\theta_r,1]}$ are an affine parameterization of segments inside $\Ss_1$ and $\Ss_r$.
\end{itemize}
Using that $\delta_i,\rho_i>0$ has been chosen small enough to guarantee that $\Lambda_i,\Gamma_i$ are contained in small balls centered in $p_i$ and $q_i$, one can check that $\gamma|_{[t_1,t_r]}$ is close to the polygonal path $\beta$ (see (iv) in the statement). In addition, each polynomial piece of $\gamma$ has degree $\leq n+1$. Define
{\small$$
\veps:=\min_i\{\dist(\beta_i(\tau_i),\R^n\setminus\Ss_i),\dist(\beta_i(\theta_i),\R^n\setminus\Ss_i),\dist(\lambda_i(\xi_i),\R^n\setminus\Ss_i),\dist(\lambda_i(\zeta_i),\R^n\setminus\Ss_{i+1})\}>0.
$$}
Denote $K:=[0,1]\setminus(\bigcup_{i=1}^r(\tau_i,\theta_i)\cup\bigcup_{i=1}^{r-1}(\xi_i,\zeta_i))$ and recall that if $I$ is a connected component of $K$, the restriction $\gamma|_I$ is an affine parameterization of a segment inside some $\Ss_i$. By Lemma \ref{segment} 
\begin{itemize}
\item[(0)] if $\gamma^*:[0,1]\to\R^n$ is a continuous semialgebraic map such that $\|\gamma-\gamma^*\|_K<\veps$, then $\gamma^*(K)\subset\Ss$ and each connect component of $\gamma^*(K)$ is contained in the required $\Ss_i$. In addition, $\gamma^*|_{K\cap[t_1,t_r]}$ is close to $\beta|_{K\cap[t_1,t_r]}$.
\end{itemize}

Write $\Ss_i:=\{h_{i1}>0,\ldots,h_{is}>0\}$ where $h_{ij}\in\R[\x]$ is a polynomial of degree $1$. As $\beta_i([t_i-\delta_i,t_i+\delta_i]\setminus\{t_i\})\subset\Ss_i=\{h_{i1}>0,\ldots,h_{is}>0\}$, the polynomial $h_{ij}\circ\beta_i$ is strictly positive on the interval $(t_i,t_i+\delta_i]$. As each $h_{ij}$ has degree $1$ and $\beta_i$ has degree $e_i\leq 3$, then $h_{ij}\circ\beta_i$ is a non-zero polynomial of degree $m_{ij}\leq e_i\leq3$. Analogously, as $\lambda_i([s_i-\rho_i,s_i))\subset\Ss_i=\{h_{i1}>0,\ldots,h_{is}>0\}$ and $\lambda_i((s_i,s_i+\rho_i])\subset\Ss_{i+1}=\{h_{i+1,1}>0,\ldots,h_{i+1,s}>0\}$, the polynomial $h_{ij}\circ\lambda_i$ is strictly positive on $[s_i-\rho_i,s_i)$ and the polynomial $h_{i+1,j}\circ\lambda_i$ is strictly positive on $(s_i,s_i+\rho_i]$. Thus, $h_{ij}\circ\lambda_i$ and $h_{i+1,j}\circ\lambda_i$ are non-zero polynomials of degrees $m_{ij}',m_{ij}''\leq d_i\leq n+1$. Consider the constants
\begin{align*}
&\mu_{ij}:=\Big|\frac{d^{m_{ij}}}{d\t^{m_{ij}}}(h_{ij}\circ\gamma|_{[\tau_i,\theta_i]})\Big|=\Big|\frac{d^{m_{ij}}}{d\t^{m_{ij}}}(h_{ij}\circ\beta_i)\Big|>0,\\
&\mu_{ij}':=\Big|\frac{d^{m_{ij}'}}{d\t^{m_{ij}'}}(h_{ij}\circ\gamma|_{[\xi_i,s_i]})\Big|=\Big|\frac{d^{m_{ij}'}}{d\t^{m_{ij}'}}(h_{ij}\circ\lambda_i)\Big|>0,\\
&\mu_{ij}'':=\Big|\frac{d^{m_{ij}''}}{d\t^{m_{ij}''}}(h_{i+1,j}\circ\gamma|_{[s_i,\zeta_i]})\Big|=\Big|\frac{d^{m_{ij}''}}{d\t^{m_{ij}''}}(h_{i+1,j}\circ\lambda_i)\Big|>0.
\end{align*}
Define 
\begin{equation}\label{ell}
\ell:=\max\{m_{ij},m_{ij}',m_{ij}'':\ 1\leq i\leq r,\ 1\leq j\leq s\}\leq n+1. 
\end{equation}
By the Remark \ref{sharpr}(i) to the proof of Lemma \ref{clue}(1) we deduce that if $\gamma^*:[0,1]\to\R^n$ is a $\Cont^{\ell+4}$ semialgebraic map such that 
\begin{itemize}
\item[(1)] $|\frac{d^{m_{ij}}}{d\t^{m_{ij}}}(h_{ij}\circ\gamma|_{[\tau_i,\theta_i]})-\frac{d^{m_{ij}}}{d\t^{m_{ij}}}(h_{ij}\circ\gamma^*|_{[\tau_i,\theta_i]})|_<\mu_{ij}$,
\item[(2)] $|\frac{d^{m_{ij}'}}{d\t^{m_{ij}'}}(h_{ij}\circ\gamma|_{[\xi_i,s_i]})-\frac{d^{m_{ij}'}}{d\t^{m_{ij}'}}(h_{ij}\circ\gamma^*|_{[\xi_i,s_i]})|<\mu_{ij}'$,
\item[(3)] $|\frac{d^{m_{ij}''}}{d\t^{m_{ij}''}}(h_{i+1,j}\circ\gamma|_{[s_i,\zeta_i]})-\frac{d^{m_{ij}''}}{d\t^{m_{ij}''}}(h_{i+1,j}\circ\gamma^*|_{[s_i,\zeta_i]})|<\mu_{ij}''$,
\item[(4)] $T_{t_i}^{e_i}\gamma=T_{t_i}^{e_i}\gamma^*$ for $i=1,\ldots,r$ and $T_{s_i}^{d_i}\gamma=T_{s_i}^{d_i}\gamma^*$ for $i=1,\ldots,r-1$,
\end{itemize}
then $\gamma^*([0,1]\setminus K)\subset\Ss\cup\{p_1,\ldots,p_r,q_1,\ldots,q_{r-1}\}$. In fact, $\gamma^*([\tau_i,\theta_i])\subset\Ss_i$, $\gamma^*([\xi_i,s_i])\subset\Ss_i$ and $\gamma^*([s_i,\zeta_i])\subset\Ss_{i+1}$. Conditions (0) to (4) concerning $\veps,\mu_{ij},\mu_{ij}',\mu_{ij}''$ and the Taylor expansions at the values $t_i$ and $s_i$ determine when a polynomial path $\alpha:\R\to\R^n$, whose restriction to $[0,1]$ is close to $\gamma$, satisfies the conditions (i) to (iv) in the statement (Figure \ref{fig6}). Finally, such a polynomial path $\alpha$ exists by Lemma \ref{swdp}, as required.
\end{proof}

\subsection{Degree of the polynomial approximation in the PL case.}\label{bound}
We maintain all the notations introduced in the proof of Main Theorem \ref{plcase}. Recall that the polynomials $h_{ij}$ have degree $1$. To simplify the presentation we assume $m_{ij}=e_i$, $m_{ij}'=m_{ij}''=d_i$ for each couple $(i,j)$ and we take a smaller $0<\veps'<\veps$ such that if $\|\alpha-\gamma\|_{K}<\veps'$, $\|\alpha^{(e_i)}-\gamma^{(e_i)}\|_{[\tau_i,\theta_i]}<\veps'$ and $\|\alpha^{(d_i)}-\gamma^{(d_i)}\|_{[\xi_i,\zeta_i]}<\veps'$, then conditions $(0)$ to $(3)$ are satisfied. As the polynomials $h_{ij}\in\R[\x]$ have degree $1$, the computation of $\veps'$ from $\veps$ seems feasible without too much effort. To have in addition condition (4) we review the proof of Lemma \ref{swdp} and need to add a linear combination of suitable polynomials (see Equations \eqref{poly} and \eqref{ell}) of degrees $\leq\ell+(r-1)(\ell+1)^2\leq n+1+(r-1)(n+2)^2$, which possibly forces us to take a smaller $\veps'>0$ (see the proof of Lemma \ref{swdp}). Due to the high degree of the latter polynomials, the effective computation of the new $\veps'$ seems cumbersome, because it involves bounds of several derivatives of such polynomials on the interval $[0,1]$, see \eqref{boundm} and \eqref{bounddelta}. However, such polynomials are quite standard and the bounds for its derivatives on the interval $[0,1]$ can be computed once and then used repeatedly when needed.

To estimate the degree $\nu$ of the polynomial path $\alpha:\R\to\R^n$ we use Theorem \ref{cotasi}. In view of such result there exist constants $C,C_i,L_i>0$ such that if $\gamma:=(\gamma_1,\ldots,\gamma_n)$ and $\alpha:=(B_\nu(\gamma_1),\ldots,B_\nu(\gamma_n))$ for an integer $\nu\geq1$, then 
\begin{align*}
&\|\alpha-\gamma\|_{K}\leq\frac{C}{\nu^2},\\
&\|\alpha^{(e_i)}-\gamma^{(e_i)}\|_{[\tau_i,\theta_i]}<\frac{e_i(e_i-1)}{2\nu}\|\beta_i^{(e_i)}\|+\frac{C_i}{\nu^2},\quad\text{}\\
&\|\alpha^{(d_i)}-\gamma^{(d_i)}\|_{[\xi_i,\zeta_i]}<\frac{d_i(d_i-1)}{2\nu}\|\lambda_i^{(d_i)}\|+\frac{L_i}{\nu^2}.
\end{align*}
The effective computation of the constants $C,C_i,L_i>0$ requires to follow the proof of Theorem \ref{cotasi} applied to $\gamma$. The proof of Theorem \ref{cotasi} is constructive enough to make the effective computation of the constants possible, but patience is mandatory. 

We have used $\gamma|_{[\tau_i,\theta_i]}=\beta_i$ and $\gamma|_{[\xi_i,\zeta_i]}=\lambda_i$ and the fact that $\beta_i$ and $\lambda_i$ are polynomial tuples of respective degrees $e_i$ and $d_i$. In particular, $\|\beta_i^{(e_i)}\|$ and $\|\lambda_i^{(d_i)}\|$ are constants. Thus, to compute the degree $\nu$ of $\alpha$ we need
\begin{equation}\label{error1}
\min_i\Big\{\frac{C}{\nu^2},\frac{e_i(e_i-1)}{2\nu}\|\beta_i^{(e_i)}\|+\frac{2C_i}{2\nu^2},\frac{d_i(d_i-1)}{2\nu}\|\lambda_i^{(d_i)}\|+\frac{2L_i}{2\nu^2}\Big\}<\veps'.
\end{equation}
For instance, we may take
\begin{equation}\label{error2}
\nu_0:=\Big\lceil\max_i\Big\{\frac{\sqrt{C}}{\sqrt{\veps'}},\frac{\sqrt{2C_i}}{\sqrt{\veps'}},\frac{\sqrt{2L_i}}{\sqrt{\veps'}},\frac{e_i(e_i-1)}{\veps'}\|\beta_i^{(e_i)}\|,\frac{d_i(d_i-1)}{\veps'}\|\lambda_i^{(d_i)}\|\Big\}\Big\rceil+1.
\end{equation}
Then, $\nu:=\max\{n+1+(r-1)(n+2)^2,\nu_0\}$ is the degree of the searched polynomial path $\alpha:\R\to\R^n$.
\qed

\begin{remark}\label{plcasec}
In \cite{fu} we study the problem of representing (compact) semialgebraic sets $\Ss\subset\R^n$ (that are connected by analytic paths) as polynomial images of a closed unit ball $\ol{\Bb}_m(0,1)$. A relevant case is the representation of a finite union $\Ss\subset\R^n$ of $n$-dimensional convex polyhedra $\pol_\ell$ (such that $\Ss$ is connected by analytic paths) as a polynomial image of either the $(n+1)$-dimensional closed unit ball $\ol{\Bb}_{n+1}(0,1)$ or the $n$-dimensional closed unit ball $\ol{\Bb}_n(0,1)$. 

If the reader follows the proofs of \cite[Thm.1.2 \& Thm.1.3]{fu}, he realizes that the complexity of the construction concentrates on finitely many polynomial paths that can be constructed using Main Theorem \ref{plcase} (the PL version of Main Theorem \ref{nashsmart}). The polynomial maps constructed to prove \cite[Thm.1.2 \& Thm.1.3]{fu} are the composition of a polynomial map of degree $6$ (see \cite[Lem.2.5 \& Lem.2.7]{fu}) that transforms the closed unit ball $\ol{\Bb}_m(0,1)$ onto the symplicial prism $\Delta_m:=\{0\leq\x_1,\ldots,0\leq\x_m,\x_1+\cdots+\x_m\leq1\}\times[0,1]$ (for either $m=n$ or $n-1$) with polynomial maps 
$$
\varphi_m:\Delta_m\times[0,1]\to\Ss,\ (\lambda_1,\ldots,\lambda_m,t)\to\Big(1-\sum_{k=1}^m\lambda_\ell\Big)\alpha_0(t)
+\sum_{k=1}^m\lambda_\ell\alpha_k(t)
$$
where each $\alpha_k:[0,1]\to\Ss$ is a polynomial path inside $\Ss$ that passes through the vertices of the simplices of a suitable triangulation of the $n$-dimensional compact convex polyhedra $\pol_\ell$, whose union constitutes the semialgebraic set $\Ss$. As $\varphi_m$ has degree $1$ with respect to $\lambda_1,\ldots,\lambda_m$, the complexity of the involved polynomials concentrates on the construction of the mentioned polynomial paths $\alpha_k$ and one would like to estimate the degree of such polynomial paths. This can be done using Main Theorem \ref{plcase} (the PL version of Main Theorem \ref{nashsmart}).

In Main Theorem \ref{plcase} we have provided a simplified proof and consequently an estimation of the degree of such polynomial paths (see Equations \eqref{error1} and \eqref{error2}) in terms of the formulas provided in Theorem \ref{cotasi}. Using formulas \eqref{error1} or \eqref{error2}, the reader can bound the degree of the polynomial paths mentioned above. Thus, one can estimate for each $n$-dimensional PL semialgebraic set $\Ss\subset\R^n$ (connected by analytic paths) the degree of the polynomials maps from either the $(n+1)$-dimensional closed unit ball $\ol{\Bb}_{n+1}(0,1)$ or the $n$-dimensional closed unit ball $\ol{\Bb}_n(0,1)$ to $\R^n$ that represent $\Ss$.\hfill$\sqbullet$
\end{remark}

\section{Convergence of derivatives of Bernstein's polynomials on compact subsets}\label{s5}

The purpose of this section is to prove Theorem \ref{cotasi}. We recall for the sake of completeness some notation, terminology and preliminary statements from \cite{f}. Let $f:[0,1]\to\R$ be a continuous function. 

\subsection{Derivatives of divided differences of a continuous function.}
For each pair of integers $s,t\geq0$ define
$$
B_{\nu,s,t}(f)(x):=\sum_{k=0}^{\nu-s}\Big(\Big[\frac{k}{\nu},\ldots,\frac{k+s}{\nu},\underbrace{x,\ldots,x}_{\text{$t$ times}}\Big]f\Big)B_{k,\nu-s}(x)
$$
where $[x_0,\ldots,x_k]f$ denotes the \em $k$th order divided difference of $f$ at the points $x_0,\ldots,x_k\in[0,1]$\em. Write $\ell:=s+t$. If $f$ is a $\Cont^{\ell}$-function, there exists by \cite[Cor.3.4.2]{d} a value $\xi_k$ in the smallest interval that contains the points $\frac{k}{\nu},\ldots,\frac{k+s}{\nu},x$ such that
$$
\Big[\frac{k}{\nu},\ldots,\frac{k+s}{\nu},\underbrace{x,\ldots,x}_{\text{$t$ times}}\Big]f=\frac{f^{(\ell)}(\xi_k)}{\ell!}.
$$
Thus, if $x\in[0,1]$, we have by \S\ref{bpbp},
\begin{equation}\label{a1}
|B_{\nu,s,t}(f)(x)|\leq\sum_{k=0}^{\nu-s}\Big|\frac{f^{(\ell)}(\xi_k)}{\ell!}\Big|B_{k,\nu-s}\leq\frac{\|f^{(\ell)}\|_{[0,1]}}{\ell!}=\frac{\|f^{(s+t)}\|_{[0,1]}}{(s+t)!}.
\end{equation}

We have $B_{\nu,0,0}(f)=B_\nu(f)$ and by \cite[pag.133]{f}
\begin{equation}\label{bff}
B_\nu(f)(x)-f(x)=\frac{1}{\nu}x(1-x)B_{\nu,1,1}(f)(x).
\end{equation}
Differentiating \eqref{bff} at a point $x\in[0,1]$ where $f$ is differentiable, we obtain
$$
B_\nu(f)'(x)-f'(x)=\frac{1}{\nu}((1-2x)B_{\nu,1,1}(f)(x)+x(1-x)(B_{\nu,1,1}(f))'(x)).
$$
Using Leibniz rules and differentiating $\ell$ times equation \eqref{bff} (at a point $x\in[0,1]$ where $f$ is $\ell$ times differentiable), we obtain \cite[Eq.(3.2)]{f}
\begin{multline}\label{32}
(B_\nu(f))^{(\ell)}(x)-f^{(\ell)}(x)=\frac{1}{\nu}(-\ell(\ell-1)(B_{\nu,1,1}(f))^{(\ell-2)}(x)\\
+\ell(1-2x)(B_{\nu,1,1}(f))^{(\ell-1)}(x)+x(1-x)(B_{\nu,1,1}(f))^{(\ell)}(x)).
\end{multline}

Let $x\in[0,1]$ be a point such that $f$ is a $\Cont^{\ell+2}$-function on a neighborhood of $x$. By \cite[Lem.1]{f} one deduces
\begin{equation}\label{33}
(B_{\nu,1,1}(f))^{(\ell)}(x)=\ell!\sum_{k=1}^{\ell+1}k\frac{\nu-1}{\nu}\cdots\frac{\nu-k+1}{\nu}B_{\nu,k,\ell-k+2}(f)(x).
\end{equation}
Thus, if $f$ is a $\Cont^{\ell+2}$-function on $[0,1]$, we have by \eqref{a1} and the equality $\sum_{k=1}^{\ell+1}k=\frac{(\ell+2)(\ell+1)}{2}$
\begin{multline}\label{31}
|(B_{\nu,1,1}(f))^{(\ell)}(x)|\leq\ell!\sum_{k=1}^{\ell+1}k\frac{\nu-1}{\nu}\cdots\frac{\nu-k+1}{\nu}|B_{\nu,k,\ell-k+2}(f)(x)|\\
\leq\ell!\sum_{k=1}^{\ell+1}k\frac{\|f^{(\ell+2)}\|_{[0,1]}}{(\ell+2)!}=\frac{\|f^{(\ell+2)}\|_{[0,1]}}{2}.
\end{multline}

\subsection{Comparison of derivatives of Bernstein's polynomials.}
In the following result we compare on a compact subset $K$ of an open subset $\Omega\subset(0,1)$ the higher order derivatives of the corresponding Bernstein's polynomials of degree $\nu$ of two continuous functions on $[0,1]$ that coincide on $\Omega$.

\begin{lem}[Comparison]\label{comparison}
Let $f_1,f_2:[0,1]\to\R$ be continuous functions that coincide on an open set $\Omega\subset(0,1)$ and let $\ell\geq0$. Then for each compact set $K\subset\Omega$ there exists a constant $M_{K,\ell}>0$ (depending only on $K$ and $\ell$) such that
$$
|B_\nu(f_1)^{(\ell)}(x)-B_\nu(f_2)^{(\ell)}(x)|\leq\frac{M_{K,\ell}}{\nu^2}\|f_1-f_2\|_{[0,1]}
$$
for each $x\in K$.
\end{lem}
\begin{proof}
Let $i,j,\ell\geq0$ be such that $2i+j\leq\ell$. By \cite[Ch.4.Prop.4.4]{dl} there exist polynomials $q_{ij\ell}\in\R[\x]$ that do not depend on $\nu,k$ such that
$$
\frac{d^\ell}{d\x^\ell}(\x^k(1-\x)^{\nu-k})=\x^{k-\ell}(1-\x)^{\nu-k-\ell}\sum_{2i+j\leq\ell}\nu^i(k-\nu\x)^jq_{ij\ell}(\x).
$$
Write $f:=f_1-f_2$, which is identically $0$ on $K$. Observe that
\begin{equation}\label{key}
\begin{split}
B_\nu^{(\ell)}(f)&=\sum_{k=0}^\nu f\Big(\frac{k}{\nu}\Big)\binom{\nu}{k}\frac{d^\ell}{d\x^\ell}(\x^k(1-\x)^{\nu-k})\\
&=\sum_{k=0}^\nu f\Big(\frac{k}{\nu}\Big)\binom{\nu}{k}\x^{k-\ell}(1-\x)^{\nu-k-\ell}\sum_{2i+j\leq\ell}\nu^i(k-\nu\x)^jq_{ij\ell}(\x)\\
&=\frac{1}{\x^\ell(1-\x)^{\ell}}\sum_{k=0}^\nu f\Big(\frac{k}{\nu}\Big)B_{k,\nu}(\x)\sum_{2i+j\leq\ell}\nu^i(k-\nu\x)^jq_{ij\ell}(\x)\\
&=\frac{1}{\x^\ell(1-\x)^{\ell}}\sum_{2i+j\leq\ell}q_{ij\ell}(\x)\nu^{i+j}\sum_{k=0}^\nu f\Big(\frac{k}{\nu}\Big)\Big(\frac{k}{\nu}-\x\Big)^jB_{k,\nu}(\x).
\end{split}
\end{equation}
Let $\delta:=\dist(K,[0,1]\setminus\Omega)>0$ and observe that if $x\in K$ and $|\frac{k}{\nu}-x|\leq\delta$, then $f(\frac{k}{\nu})=0$. By \cite[Ch.10.\S1.(1.6), pag. 304]{dl} there exists a constant $C(\delta,i+j+2)$ such that
\begin{equation}\label{ineq}
\sum_{|\frac{k}{\nu}-x|>\delta}B_{k,\nu}(x)\leq C(\delta,i+j+2)\frac{1}{\nu^{i+j+2}}.
\end{equation}
Thus, by \eqref{key}, \eqref{ineq} and as $B_\nu(f_1)^{(\ell)}(x)-B_\nu(f_2)^{(\ell)}(x)=B_\nu^{(\ell)}(f)(x)$ and $|\frac{k}{\nu}-x|\leq1$,
\begin{multline*}
|B_\nu(f_1)^{(\ell)}(x)-B_\nu(f_2)^{(\ell)}(x)|\leq\frac{1}{x^\ell(1-x)^{\ell}}\sum_{2i+j\leq\ell}|q_{ij\ell}(x)|\nu^{i+j}\sum_{|\frac{k}{\nu}-x|>\delta}\Big|f\Big(\frac{k}{\nu}\Big)\Big|B_{k,\nu}(x)\\
\leq\Big(\frac{1}{x^\ell(1-x)^{\ell}}\sum_{2i+j\leq\ell}|q_{ij\ell}(x)|C(\delta,i+j+2)\Big)\frac{1}{\nu^2}\|f\|_{[0,1]}
\end{multline*}
for each $x\in K$. Now, the statement follows readily.
\end{proof}

\subsection{Some bounds for derivatives of Taylor polynomials.}
Let $f:[0,1]\to\R$ be a continuous function that is $\Cont^\ell$ on an open subset $\Omega\subset[0,1]$. Define
$$
T^\ell f:\Omega\times[0,1]\to\R,\ (y,x)\mapsto\sum_{k=0}^\ell\frac{f^{(k)}(y)}{k!}(x-y)^k.
$$
We have
$$
\frac{\partial^m}{\partial\x^m}T^\ell f=\sum_{k=m}^\ell\frac{f^{(k)}(\y)}{(k-m)!}(\x-\y)^{k-m}.
$$
If $K\subset\Omega$ is a compact set,
\begin{align*}
&\|T^\ell f\|_{K\times[0,1]}:=\max\{T^\ell f(y,x):\ (y,x)\in K\times[0,1]\},\\
&\|(T^\ell f)^{(m)}\|_{K\times[0,1]}:=\Big\|\frac{\partial^m}{\partial\x^m}T^\ell f\Big\|_{K\times[0,1]}:=\max\Big\{\frac{\partial^m}{\partial\x^m}T^\ell f(y,x):\ (y,x)\in K\times[0,1]\Big\}.
\end{align*}
As the points $x,y\in[0,1]$, we deduce
\begin{equation}\label{34}
\|(T^\ell f)^{(m)}\|_{K\times[0,1]}\leq\sum_{k=m}^\ell\frac{\|f^{(k)}\|_K}{(k-m)!}.
\end{equation}
In particular, $\|(T^\ell f)^{(\ell)}\|_{K\times[0,1]}\leq\|f^{(\ell)}\|_K$. 

\subsection{Proof of Theorem \ref{cotasi}}
The proof is conducted in several steps:

\noindent{\sc Step 1. Initial preparation.}
Define
$$
P:=T^{\ell+3}f:\Omega\times[0,1]\to\R, (y,x)\mapsto\sum_{k=0}^{\ell+3}\frac{f^{(k)}(y)}{k!}(x-y)^k.
$$
We claim: {\em There exists a function $g:\Omega\times[0,1]\to\R$ such that $h(y,x):=f(x)-P(y,x)=g(y,x)(x-y)^{\ell+4}$ on $\Omega\times[0,1]$ and for each compact set $K\subset\Omega$ there exists a constant $N_{f,K}>0$ such that $|g(y,x)|<N_{f,K}$ for each $(y,x)\in K\times[0,1]$} (see also Remark \ref{cotasir}).

Define
$$
g:\Omega\times[0,1]\to\R,\ (y,x)\mapsto
\begin{cases}
\frac{h(y,x)}{(x-y)^{\ell+4}}&\text{if $x\neq y$,}\\
0&\text{otherwise.}
\end{cases}
$$
Observe that $g$ is continuous on $(\Omega\times[0,1])\setminus\Delta$ where $\Delta:=\{(x,x)\in\Omega\times[0,1],\ x\in\Omega\}$.

Fix a compact set $K\subset\Omega$. For each $x\in K$ choose $\veps_x>0$ such that $[x-2\veps_x,x+2\veps_x]\subset\Omega$. As $K$ is a compact set, there exist $x_1,\ldots,x_k\in K$ such that $K\subset K':=\bigcup_{j=1}^k[x_j-\veps_{x_j},x_j+\veps_{x_j}]$. As $f^{(\ell+4)}$ is continuous in $\Omega$, there exists a constant $N_{1,f,K''}>0$ such that $|f^{(\ell+4)}(z)|\leq N_{1,f,K''}(\ell+4)!$ for each $z\in K'':=\bigcup_{j=1}^k[x_j-2\veps_{x_j},x_j+2\veps_{x_j}]$. Define $L_j:=[0,1]\setminus(x_j-2\veps_{x_j},x_j+2\veps_{x_j})$ and observe that 
$$
K'\times[0,1]=\bigcup_{j=1}^k([x_j-\veps_{x_j},x_j+\veps_{x_j}]\times[x_j-2\veps_{x_j},x_j+2\veps_{x_j}])\cup\bigcup_{j=1}^k([x_j-\veps_{x_j},x_j+\veps_{x_j}]\times L_j)
$$
As $\Delta\cap(\bigcup_{j=1}^k[x_j-\veps_{x_j},x_j+\veps_{x_j}]\times L_j)=\varnothing$, the function $g$ is continuous on the compact set $\bigcup_{j=1}^k[x_j-\veps_{x_j},x_j+\veps_{x_j}]\times L_j$, so there exists $N_{2,f,K'}>0$ such that $|g(y,x)|<N_{2,f,K'}$ for each $(y,x)\in\bigcup_{j=1}^k[x_j-\veps_{x_j},x_j+\veps_{x_j}]\times L_j$.

As $f$ is $\Cont^{\ell+4}$ on $\Omega$, for each $(y,x)\in[x_j-\veps_{x_j},x_j+\veps_{x_j}]\times[x_j-2\veps_{x_j},x_j+2\veps_{x_j}]$ there exists by Lagrange form of the remainder of Taylor's theorem $\zeta_{(y,x)}\in[x_j-2\veps_{x_j},x_j+2\veps_{x_j}]\subset K''$ such that
$$
h(y,x)=\frac{f^{(\ell+4)}(\zeta_{(y,x)})}{(\ell+4)!}(x-y)^{\ell+4},
$$ 
so $g(y,x)=\frac{f^{(\ell+4)}(\zeta_{(y,x)})}{(\ell+4)!}$ and $|g(y,x)|<N_{1,f,K''}$ for each $(y,x)\in\bigcup_{j=1}^k[x_j-\veps_{x_j},x_j+\veps_{x_j}]\times[x_j-2\veps_{x_j},x_j+2\veps_{x_j}]$. Thus, if we define $N_{f,K}:=\max\{N_{1,f,K''},N_{2,f,K'}\}$, the claim is proved.

Define $P_y:=P(y,\cdot)$ and $h_y:=h(y,\cdot)$ and fix a compact set $K\subset\Omega$. Observe that $P_y^{(k)}(y)=f^{(k)}(y)$ for each $y\in K$ and each $k=0,\ldots,\ell+3$. We have $B_{\nu}(f)=B_{\nu}(P_y)+B_{\nu}(h_y)$. Consequently, for each $y\in K$
\begin{equation}\label{bernstein}
B_{\nu}^{(\ell)}(f)=B_{\nu}^{(\ell)}(P_y)+B_{\nu}^{(\ell)}(h_y)
\end{equation}
(we are considering derivatives and Bernstein's polynomials with respect to the variable $\x$). 

\noindent{\sc Step 2. Uniform control of the error for the Taylor polynomials.}
We claim: \em there exists a constant $C_{f,K,\ell}>0$ such that $|B_\nu^{(\ell)}(h_y)(y)|<\frac{C_{f,K,\ell}}{\nu^2}$ for each $y\in K$\em.

Let $i,j,\ell\geq0$ be such that $2i+j\leq\ell$. By \cite[Ch.4.Prop.4.4]{dl} there exist polynomials $q_{ij\ell}\in\R[\x]$ that do not depend on $\nu,k$ such that
$$
\frac{d^\ell}{d\x^\ell}(\x^k(1-\x)^{\nu-k})=\x^{k-\ell}(1-\x)^{\nu-k-\ell}\sum_{2i+j\leq\ell}\nu^i(k-\nu\x)^jq_{ij\ell}(\x).
$$
As $h_y(\frac{k}{\nu})=g(y,\frac{k}{\nu})\cdot(\frac{k}{\nu}-y)^{\ell+4}$, we have
\begin{equation*}
\begin{split}
B_\nu^{(\ell)}(h_y)&=\sum_{k=0}^\nu h_y\Big(\frac{k}{\nu}\Big)\binom{\nu}{k}\frac{d^\ell}{d\x^\ell}(\x^k(1-\x)^{\nu-k})\\
&=\sum_{k=0}^\nu h_y\Big(\frac{k}{\nu}\Big)\binom{\nu}{k}\x^{k-\ell}(1-\x)^{\nu-k-\ell}\sum_{2i+j\leq\ell}\nu^i(k-\nu\x)^jq_{ij\ell}(\x)\\
&=\frac{1}{\x^\ell(1-\x)^{\ell}}\sum_{k=0}^\nu h_y\Big(\frac{k}{\nu}\Big)B_{k,\nu}(\x)\sum_{2i+j\leq\ell}\nu^{i+j}\Big(\frac{k}{\nu}-\x\Big)^jq_{ij\ell}(\x)\\
&=\frac{1}{\x^\ell(1-\x)^{\ell}}\sum_{2i+j\leq\ell}q_{ij\ell}(\x)\nu^{i+j}\sum_{k=0}^\nu g\Big(y,\frac{k}{\nu}\Big)\Big(\frac{k}{\nu}-y\Big)^{\ell+4}\Big(\frac{k}{\nu}-\x\Big)^jB_{k,\nu}(\x).
\end{split}
\end{equation*}
We have proved above that there exists a constant $N_{f,K}>0$ such that $|g(y,x)|<N_{f,K}$ for each $(y,x)\in K\times[0,1]$. Recall that $2(i+j)\leq\ell+j$ and $|\frac{k}{\nu}-y|\leq1$. If we set $\x=y$, we have
\begin{multline}\label{just}
\Big|\sum_{k=0}^\nu g\Big(y,\frac{k}{\nu}\Big)\Big(\frac{k}{\nu}-y\Big)^{\ell+j+4}B_{k,\nu}(y)\Big|\leq N_{f,K}\frac{1}{\nu^{2(i+j)+4}}\sum_{k=0}^\nu(k-\nu y)^{2(i+j)+4}B_{k,\nu}(y)\\
\leq N_{f,K}\frac{1}{\nu^{2(i+j)+4}}A_{i+j+2}\nu^{i+j+2}
\leq N_{f,K}A_{i+j+2}\frac{1}{\nu^{i+j+2}} 
\end{multline}
for a constant $A_{i+j+2}>0$ (see \cite[Ch.10.\S1.(1.5), pag. 304]{dl}). Consequently, 
$$
|B_\nu^{(\ell)}(h_y)(y)|\leq\Big(\frac{1}{y^\ell(1-y)^{\ell}}\sum_{2i+j\leq\ell}|q_{ij\ell}(y)|N_{f,K}A_{i+j+2}\Big)\frac{1}{\nu^2}
$$
and the claim follows if we take $C_{f,K,\ell}:=\|\frac{1}{\y^\ell(1-\y)^{\ell}}\|_K\sum_{2i+j\leq\ell}\|q_{ij\ell}\|_KN_{f,K}A_{i+j+2}$. 

\noindent{\sc Step 3. Proof of the first part of the statement.}
If $x\in K$, we have using {\sc Step 2} (because $P_x^{(k)}(x)=f^{(k)}(x)$ for each $x\in K$ and each $k=0,\ldots,\ell+3$)
\begin{multline}\label{35}
|B_\nu^{(\ell)}(f)(x)-f^{(\ell)}(x)|\leq|B_\nu^{(\ell)}(P_x)(x)-f^{(\ell)}(x)|+|B_\nu^{(\ell)}(h_x)(x)|\\
\leq|B_\nu^{(\ell)}(P_x)(x)-P_x^{(\ell)}(x)|+\frac{C_{f,K,\ell}}{\nu^2}.
\end{multline}
By \eqref{32} and \eqref{31} applied to $P_x$ we obtain
\begin{multline*}
|B_\nu^{(\ell)}(P_x)(x)-P_x^{(\ell)}(x)|\leq\frac{1}{2\nu}(\ell(\ell-1)\|P_x^{(\ell)}\|_{[0,1]}\\
+\ell|1-2x|\|P_x^{(\ell+1)}\|_{[0,1]}+x(1-x)\|P_x^{(\ell+2)}\|_{[0,1]}).
\end{multline*}
By \eqref{34} we deduce
\begin{multline*}
|B_\nu^{(\ell)}(P_x)(x)-P_x^{(\ell)}(x)|\leq\frac{1}{2\nu}\Big(\ell(\ell-1)\sum_{k=\ell}^{\ell+3}\frac{\|f^{(k)}\|_K}{(k-\ell)!}\\
+\ell|1-2x|\sum_{k=\ell+1}^{\ell+3}\frac{\|f^{(k)}\|_K}{(k-\ell-1)!}+x(1-x)\sum_{k=\ell+2}^{\ell+3}\frac{\|f^{(k)}\|_K}{(k-\ell-2)!}\Big)
\end{multline*}
for each $x\in K$ and the first part of the statement holds.

\noindent{\sc Step 4. Bound of the error.} {\em For each $\veps>0$ and each pair of integers $s,t\geq0$ such that $\lambda:=s+t\leq\ell+2$ there exists a constant $C_{f,K,\lambda,\veps}>0$ such that
\begin{equation}\label{st4}
\Big|B_{\nu,s,t}(P_y)(x)-\frac{P_y^{(\lambda)}}{\lambda!}(x)\Big|<\frac{\veps}{\lambda!}+\frac{C_{f,K,\lambda,\veps}}{\nu}
\end{equation}
for each $(y,x)\in K\times[0,1]$ and each $\nu>s$.} In particular, {\em $B_{\nu,s,t}(P_y)$ converges to $\frac{P_y^{(\lambda)}}{\lambda!}$ uniformly on $K\times[0,1]$ when $\nu\to\infty$.}

We will follow the proof of \cite[Lem.2]{f} making the suitable needed changes. Fix integers $s,t\geq0$ and denote $\lambda:=s+t$. Next fix $\nu\geq s$ and $0\leq k\leq\nu-s$. Fix $\veps>0$ and let $\delta>0$ be such that if $(y,x),(y',x')\in K\times[0,1]$ satisfy $|x-x'|<\delta$ and $|y-y'|<\delta$, then $|P_y^{(\lambda)}(x)-P_{y'}^{(\lambda)}(x')|<\veps$ (recall that $P_z$ is $\Cont^{\ell+4}$ on $[0,1]$ for each $z\in K$). Fix $x\in[0,1]$ and let 
$$
I_\nu:=\Big\{k\in\{0,\ldots,\nu-s\}:\ x-\delta<\frac{k}{\nu}<\frac{k+s}{\nu}<x+\delta\Big\}.
$$
Fix $y\in K$ and pick $\xi_k$ in the smallest interval that contains the points $x,\frac{k}{\nu},\ldots,\frac{k+s}{\nu}$ such that
$$
\Big[\frac{k}{\nu},\ldots,\frac{k+s}{\nu},\underbrace{x,\ldots,x}_{\text{$t$ times}}\Big]P_y=\frac{P_y^{(\lambda)}(\xi_k)}{\lambda!}.
$$
Consequently,
$$
B_{\nu,s,t}(P_y)(x)=\sum_{k=0}^{\nu-s}\Big(\Big[\frac{k}{\nu},\ldots,\frac{k+s}{\nu},\underbrace{x,\ldots,x}_{\text{$t$ times}}\Big]P_y\Big)B_{k,\nu-s}(x)=\frac{1}{\lambda!}\sum_{k=0}^{\nu-s}P_y^{(\lambda)}(\xi_k)B_{k,\nu-s}(x).
$$
Define
$$
S_\nu:=\lambda!B_{\nu,s,t}(P_y)(x)-P_y^{(\lambda)}(x)=\sum_{k=0}^{\nu-s}(P_y^{(\lambda)}(\xi_k)-P_y^{(\lambda)}(x))B_{k,\nu-s}(x).
$$
Write $S_\nu=C_\nu+D_\nu$ where
\begin{align*}
C_\nu&:=\sum_{k\in I_\nu}(P_y^{(\lambda)}(\xi_k)-P_y^{(\lambda)}(x))B_{k,\nu-s}(x),\\
D_\nu&:=\sum_{k\not\in I_\nu}(P_y^{(\lambda)}(\xi_k)-P_y^{(\lambda)}(x))B_{k,\nu-s}(x).
\end{align*}
If $k\in I_\nu$, we have $|\xi_k-x|<\delta$, so
$$
|C_\nu|\leq\sum_{k\in I_\nu}\veps B_{k,\nu-s}(x)\leq\veps.
$$

Regarding $D_\nu$, define 
\begin{equation}\label{boundmf}
M_{f,K,\lambda}:=\max\Big\{\Big|\frac{\partial^\lambda P}{\partial\x^\lambda}(y,x)\Big|:\ (y,x)\in K\times[0,1]\Big\}
\end{equation}
for $\lambda=0,\ldots,\ell+2$. If $0\leq k\leq\nu-s$, we have
{\small\begin{align*}
\Big|\frac{k}{\nu}-x\Big|&\leq\Big|\frac{k}{\nu-s}-x\Big|+\Big|\frac{k}{\nu}-\frac{k}{\nu-s}\Big|\leq\Big|\frac{k}{\nu-s}-x\Big|+\frac{s}{\nu}\frac{k}{\nu-s}\leq\Big|\frac{k}{\nu-s}-x\Big|+\frac{s}{\nu},\\
\Big|\frac{k+s}{\nu}-x\Big|&\leq\Big|\frac{k}{\nu-s}-x\Big|+\Big|\frac{k+s}{\nu}-\frac{k}{\nu-s}\Big|\leq\Big|\frac{k}{\nu-s}-x\Big|+\frac{s}{\nu}\Big(1-\frac{k}{\nu-s}\Big)\leq\Big|\frac{k}{\nu-s}-x\Big|+\frac{s}{\nu}.
\end{align*}}
Consequently,
$$
\max\Big\{\Big(\frac{k}{\nu}-x\Big)^2,\Big(\frac{k+s}{\nu}-x\Big)^2\Big\}\leq\Big(\frac{k}{\nu-s}-x\Big)^2+2\frac{s}{\nu}+\frac{s^2}{\nu^2}.
$$
For each $k\not\in I_\nu$ we have
$$
\delta^2\leq\max\Big\{\Big(\frac{k}{\nu}-x\Big)^2,\Big(\frac{k+s}{\nu}-x\Big)^2\Big\}\leq\Big(\frac{k}{\nu-s}-x\Big)^2+2\frac{s}{\nu}+\frac{s^2}{\nu^2}.
$$
We deduce $1\leq\frac{1}{\delta^2}((\frac{k}{\nu-s}-x)^2+2\frac{s}{\nu}+\frac{s^2}{\nu^2})$ and as $|P_y^{(\lambda)}(\xi_k)-P_y^{(\lambda)}(x)|\leq 2 M_{f,K,\lambda}$, we conclude using \S\ref{bpbp} (concretely the property of the variance of a binomial distribution)
\begin{multline*}
|D_\nu|\leq\frac{2}{\delta^2}M_{f,K,\lambda}\sum_{k\not\in I_\nu}\Big(\Big(\frac{k}{\nu-s}-x\Big)^2+2\frac{s}{\nu}+\frac{s^2}{\nu^2}\Big)B_{k,\nu-s}(x)\\
\leq\frac{2}{\delta^2}M_{f,K,\lambda}\Big(2\frac{s}{\nu}+\frac{s^2}{\nu^2}+\sum_{k=0}^{\nu-s}\Big(\frac{k}{\nu-s}-x\Big)^2B_{k,\nu-s}(x)\Big)=\frac{2}{\delta^2}M_{f,K,\lambda}\Big(2\frac{s}{\nu}+\frac{s^2}{\nu^2}+\frac{x(1-x)}{\nu-s}\Big).
\end{multline*}
As $s<\nu$, we obtain $0<\frac{s}{\nu},\frac{1}{\nu-s}<1$, so
$$
2\frac{s}{\nu}+\frac{s^2}{\nu^2}+\frac{x(1-x)}{\nu-s}\leq\frac{3s}{\nu}+\frac{1}{2\nu}\Big(1+\frac{s}{\nu-s}\Big)<\frac{1+7s}{2\nu}\leq\frac{1+7\lambda}{2\nu}.
$$
Thus, if $\nu>s$, then
$$
|S_\nu|\leq|C_\nu|+|D_\nu|<\veps+\frac{2}{\delta^2}M_{f,K,\lambda}\frac{1+7\lambda}{2\nu}=\veps+\frac{C_{f,K,\lambda,\veps}\lambda!}{\nu}
$$
for the constant $C_{f,K,\lambda,\veps}:=\frac{1+7\lambda}{\delta^2\lambda!}M_{f,K,\lambda}>0$.

\noindent{\sc Step 5. Proof of the second part of the statement.} Fix $\nu>\ell$. By Remark \ref{derf} and \eqref{32} we have
\begin{multline}\label{36}
\nu((B_\nu(P_y))^{(\ell)}(x)-P_y^{(\ell)}(x))-\frac{1}{2}\frac{\partial^\ell}{\partial x^\ell}(x(1-x)P_y''(x))\\
=-\ell(\ell-1)((B_{\nu,1,1}(P_y))^{(\ell-2)}(x)-\tfrac{1}{2}P_y^{(\ell)}(x))
+\ell(1-2x)((B_{\nu,1,1}(P_y))^{(\ell-1)}(x)-\tfrac{1}{2}P_y^{(\ell+1)}(x))\\
+x(1-x)((B_{\nu,1,1}(P_y))^{(\ell)}(x))-\tfrac{1}{2}P_y^{(\ell+2)}(x)).
\end{multline}
Write $m:=\ell-2,\ell-1,\ell$. Using that $m!\sum_{k=1}^{m+1}k=\frac{(m+2)!}2{}$ we get by \eqref{33}
\begin{multline}\label{36a}
(B_{\nu,1,1}(P_y))^{(m)}(x)-\tfrac{1}{2}P_y^{(m+2)}(x)\\
=m!\sum_{k=1}^{m+1}k\frac{\nu-1}{\nu}\cdots\frac{\nu-k+1}{\nu}\Big(B_{\nu,k,m+2-k}(P_y)(x)-\frac{P_y^{(m+2)}(x)}{(m+2)!}\Big)\\
+m!\sum_{k=1}^{m+1}k\Big(\Big(1-\frac{1}{\nu}\Big)\cdots\Big(1-\frac{k-1}{\nu}\Big)-1\Big)\frac{P_y^{(m+2)}(x)}{(m+2)!}.
\end{multline}

As $m\leq\ell<\nu$, we have for $k=1,\ldots,m+1$
$$
0<\Big(1-\frac{m}{\nu}\Big)^m\leq\Big(1-\frac{1}{\nu}\Big)\cdots\Big(1-\frac{m}{\nu}\Big)\leq\Big(1-\frac{1}{\nu}\Big)\cdots\Big(1-\frac{k-1}{\nu}\Big)<1.
$$
Consequently, as $\frac{m}{\nu}<1$, we deduce
$$
0<1-\Big(1-\frac{1}{\nu}\Big)\cdots\Big(1-\frac{k-1}{\nu}\Big)<1-\Big(1-\frac{m}{\nu}\Big)^m=\sum_{q=1}^m\binom{m}{q}(-1)^{q+1}\Big(\frac{m}{\nu}\Big)^q\leq\frac{1}{\nu}m\sum_{q=1}^m\binom{m}{q}
$$
Thus, $L_m:=m\sum_{q=1}^m\binom{m}{q}>0$ satisfies
$$
\Big|\Big(1-\frac{1}{\nu}\Big)\cdots\Big(1-\frac{k-1}{\nu}\Big)-1\Big|<\frac{L_m}{\nu}
$$ 
for $k=1,\ldots,m+1$. 

Recall that $m!\sum_{k=1}^{m+1}k=\frac{(m+2)!}{2}$ and $|\frac{\nu-1}{\nu}\cdots\frac{\nu-k+1}{\nu}|<1$. As $m\leq\ell$, we have $m+2\leq\ell+2$. Consequently, by \eqref{st4} (in {\sc Step 4}), \eqref{boundmf} and \eqref{36a} we have
\begin{equation}\label{36b}
\begin{split}
|(B_{\nu,1,1}(P_y))^{(m)}(x)&-\tfrac{1}{2}P_y^{(m+2)}(x)|\\
&\leq m!\sum_{k=1}^{m+1}k\Big|\frac{\nu-1}{\nu}\cdots\frac{\nu-k+1}{\nu}\Big|\Big|B_{\nu,k,m+2-k}(P_y)(x)-\frac{P_y^{(m+2)}(x)}{(m+2)!}\Big|\\
&+m!\sum_{k=1}^{m+1}k\Big|\Big(1-\frac{1}{\nu}\Big)\cdots\Big(1-\frac{k-1}{\nu}\Big)-1\Big|\frac{|P_y^{(m+2)}(x)|}{(m+2)!}\\
&<\frac{(m+2)!}{2}\Big(\frac{\veps}{(m+2)!}+\frac{C_{f,K,m+2,\veps}}{\nu}\Big)+\frac{L_m}{2\nu}M_{f,K,m+2}\\
&=\frac{\veps}{2}+\frac{C_{f,K,m+2,\veps}(m+2)!+L_mM_{f,K,m+2}}{2\nu}
\end{split}
\end{equation}
for each $(y,x)\in K\times[0,1]$. We conclude from \eqref{36} and \eqref{36b}
\begin{multline*}
\Big|\nu((B_\nu(P_y))^{(\ell)}(x)-P_y^{(\ell)}(x))-\frac{1}{2}\frac{\partial^\ell}{\partial x^\ell}(x(1-x)P_y''(x))\Big|\\
\leq\ell(\ell-1)\Big(\frac{\veps}{2}+\frac{(C_{f,K,\ell,\veps}\ell!+L_{\ell-2}M_{f,K,\ell})}{2\nu}\Big)+\ell|1-2x|\Big(\frac{\veps}{2}+\frac{(C_{f,K,\ell+1,\veps}(\ell+1)!+L_{\ell-1}M_{f,K,\ell+1})}{2\nu}\Big)\\
+x(1-x)\Big(\frac{\veps}{2}+\frac{(C_{f,K,\ell+2,\veps}(\ell+2)!+L_\ell M_{f,K,\ell+2})}{2\nu}\Big).
\end{multline*}
If we write
\begin{align*}
\veps':=\,&(\ell^2+\tfrac{1}{4})\tfrac{\veps}{2},\\
C^\bullet_{f,K,\ell,\veps'}:=\,&\tfrac{1}{2}(\ell(\ell-1)(C_{f,K,\ell,\veps}\ell!+L_{\ell-2}M_{f,K,\ell})+\ell(C_{f,K,\ell+1,\veps}(\ell+1)!+L_{\ell-1}M_{f,K,\ell+1})\\
&+\tfrac{1}{4}(C_{f,K,\ell+2,\veps}(\ell+2)!+L_\ell M_{f,K,\ell+2})),
\end{align*}
we conclude, using that $|1-2x|\leq1$ and $x(1-x)\leq\frac{1}{4}$,
$$
\Big|\nu((B_\nu(P_y))^{(\ell)}(x)-P_y^{(\ell)}(x))-\frac{1}{2}\frac{\partial^\ell}{\partial x^\ell}(x(1-x)P_y''(x))\Big|\leq\veps'+\frac{C^\bullet_{f,K,\ell,\veps'}}{\nu}
$$
for each $(y,x)\in K\times[0,1]$. If $x\in K$ and we set $y=x$, we deduce (using \eqref{bernstein}, that is, $(B_\nu(f))^{(\ell)}=(B_\nu(P_x))^{(\ell)}+(B_\nu(h_x))^{(\ell)}$, and {\sc Step 2})
\begin{multline*}
\Big|\nu((B_\nu(f))^{(\ell)}(x)-f^{(\ell)}(x))-\frac{1}{2}\frac{\partial^\ell}{\partial x^\ell}(x(1-x)f''(x))\Big|\\
=\Big|\nu((B_\nu(P_x))^{(\ell)}(x)-P_x^{(\ell)}(x))-\frac{1}{2}\frac{\partial^\ell}{\partial x^\ell}(x(1-x)P_x''(x))\Big|+|\nu B_\nu(h_x))^{(\ell)}(x)|\\
\leq\veps'+\frac{C^\bullet_{f,K,\ell,\veps'}}{\nu}+\frac{C_{f,K,\ell}}{\nu}.
\end{multline*}
To finish it is enough to define $C_{f,K,\ell,\veps'}^*:=C^\bullet_{f,K,K,\ell,\veps'}+C_{f,K,\ell}$ (and to adjust $\veps>0$).
\qed
\begin{remark}\label{cotasir}
In the previous proof we have only used that the function $g$ introduced in the {\sc Step 1} is bounded over the sets of the form $K\times[0,1]$ where $K\subset\Omega$ is a compact set. However, it is natural to wonder about a sufficient condition to guarantee that $g$ is in addition continuous: {\em The function $g:\Omega\times[0,1]\to\R$ is continuous if $f:\Omega\to\R$ is a $\Cont^{\ell+5}$ function}. 

We have proved in this {\sc Step 1} (adapted to the case when $f$ is $\Cont^{\ell+5}$) that there exists a function $g_0:\Omega\times[0,1]\to\R$ such that $f(x)-\sum_{k=0}^{\ell+4}\frac{f^{(k)}(y)}{k!}(x-y)^k=g_0(y,x)(x-y)^{\ell+5}$ on $\Omega\times[0,1]$ and for each compact set $K\subset\Omega$ there exists a constant $N_{0,f,K}>0$ such that $|g_0(y,x)|<N_{0,f,K}$ for each $(y,x)\in K\times[0,1]$.

Define $g(y,x):=\frac{f^{(\ell+4)}}{(\ell+4)!}(y)+(x-y)g_0(y,x)$ for each $(y,x)\in\Omega\times[0,1]$. Observe that $g_0$ is continuous outside $\Delta:=\{(x,x)\in\Omega\times[0,1]:\ x\in\Omega\}$ and it is bounded on any compact neighborhood of each point of $\Delta$ inside $\Omega\times[0,1]$. Thus, $h(y,x):=(x-y)g_0(y,x)$ is continuous on $\Omega\times[0,1]$. Consequently, $g$ is continuous on $\Omega\times[0,1]$, as required.\hfill$\sqbullet$
\end{remark}

\appendix
\section{Modification of continuous semialgebraic paths.}\label{A}
In the proof of Main Theorem \ref{nashsmart} we needed to slightly modify continuous semialgebraic paths to avoid certain algebraic sets (except for finitely many points), but keeping essentially their behavior. In order to make the proof of such result more intuitive, we have postponed such modification until now. The reader can find by himself many other ways to modify continuous semialgebraic paths in the needed way. However, we include the precise technicalities for the sake of completeness here.

\begin{lem}[Modification of continuous semialgebraic paths]\label{modification}
Let $\Ss\subset\R^n$ be a pure dimensional semialgebraic set and $\Ss_1,\ldots,\Ss_r$ open connected semialgebraic subsets of $\Reg(\Ss)$ (non-necessarily pairwise different). Pick control points $p_i\in\cl(\Ss_i)$ for $i=1,\ldots,r$ and $q_i\in\cl(\Ss_i)\cap\cl(\Ss_{i+1})$ for $i=1,\ldots,r-1$. Fix control times $s_0:=0<t_1<\cdots<t_r<1=:s_r$ and $s_i\in(t_i,t_{i+1})$ for $i=1,\ldots,r-1$. Let $Y\subset\R^n$ be a (proper) algebraic set that does not contain any of the $\Ss_i$ and let $\beta:[0,1]\to\R^n$ be a continuous semialgebraic path such that: 
\begin{itemize}
\item[(i)] $\beta([0,1])\subset\bigcup_{i=1}^r\Ss_i\cup\{p_1,\ldots,p_r,q_1,\ldots,q_{r-1}\}$,
\item[(ii)] $\beta(t_i)=p_i$ for $i=1,\ldots,r$ and $\beta(s_i)=q_i$ for $i=1,\ldots,r-1$,
\item[(iii)] $\beta((t_i,s_i))\subset\Ss_i$ for $i=1,\ldots,r$ and $\beta((s_i,t_{i+1}))\subset\Ss_{i+1}$ for $i=1,\ldots,r-1$,
\item[(iv)] $\eta(\beta)\subset(0,1)\setminus\{t_1,\ldots,t_r,s_1,\ldots,s_{r-1}\}$ and $\beta(\eta(\beta))\subset\bigcup_{i=1}^r\Ss_i$.
\end{itemize}
Then, for each $\veps>0$ there exists a continuous semialgebraic path $\beta^*:[0,1]\to\R^n$ satisfying conditions {\em(i)}, {\em(ii)}, {\em(iii)} and {\em(iv)} above and such that $(\beta^*)^{-1}(Y)$ is a finite set, $\eta(\beta^*)\cap(\beta^*)^{-1}(Y)=\varnothing$, $\beta^*(\eta(\beta^*))\subset\bigcup_{i=1}^r\Ss_i$ and $\|\beta-\beta^*\|<\veps$.
\end{lem}

\begin{center}
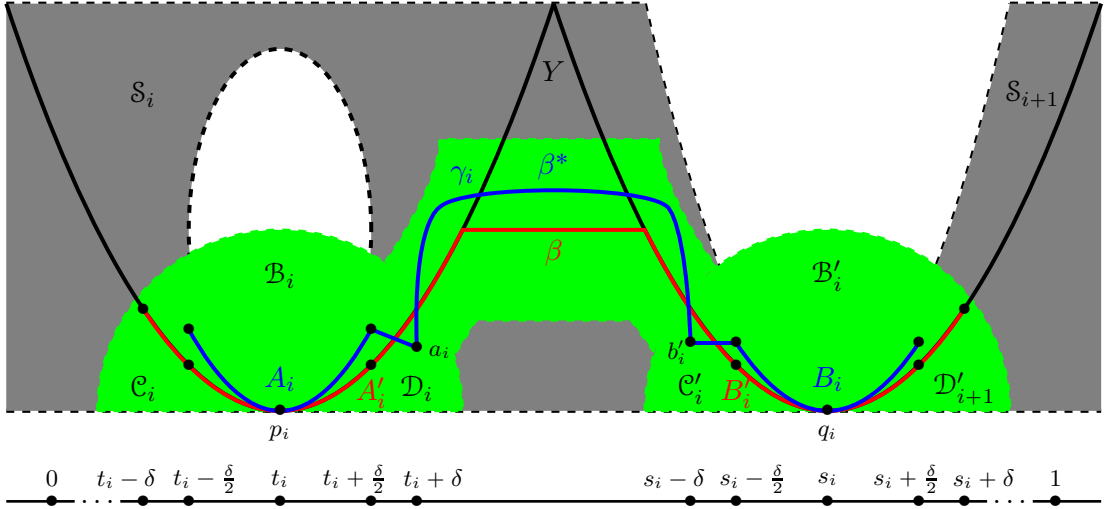
\begin{figure}[ht]
\begin{tikzpicture}[scale=1.2]

\draw[dashed,line width=1.5pt] (0,-0.5) -- (12,-0.5);
\draw[dashed,line width=1.5pt] (11,4) -- (12,4);
\draw[dashed,line width=1.5pt] (0,4) -- (7,4);

\draw[dashed,line width=1.5pt] (7,4) parabola bend (9,-0.5) (11,4);

\draw[fill=gray!100,opacity=0.5,draw=none] (0,-0.5) -- (12,-0.5) -- (12,4) -- (11,4) parabola bend (9,-0.5) (7,4) -- (0,4);
\draw[fill=white!100,opacity=1,draw=none] (3,1.5) ellipse (1cm and 2cm);
\draw[dashed,line width=1.5pt] (3,1.5) ellipse (1cm and 2cm);

\draw[color=green,dashed,line width=1pt] (1,-0.5) -- (5,-0.5) arc(0:180:2);
\draw[color=green,dashed,line width=1pt] (7,-0.5) -- (11,-0.5) arc(0:180:2) --cycle;

\draw[color=green,dashed,line width=1pt] (4.15,1.15) .. controls (4.15,1.125) and (4.75,2) .. (4.75,2.5) -- (7.15,2.5) .. controls (7.15,2) and (7.8,1) .. (7.8,1.1);
\draw[color=green,dashed,line width=1pt] (7.1,0.15) -- (6.85,0.5) -- (5.15,0.5) -- (4.9,0.1);

\draw[fill=green!100,opacity=0.25,draw=none] (1,-0.5) -- (5,-0.5) arc(0:180:2) --cycle;
\draw[fill=green!100,opacity=0.25,draw=none] (7,-0.5) -- (11,-0.5) arc(0:180:2) --cycle;

\draw[fill=green!100,opacity=0.25,draw=none] (4.15,1.15) .. controls (4.15,1.125) and (4.75,2) .. (4.75,2.5) -- (7.15,2.5) .. controls (7.15,2) and (7.8,1) .. (7.8,1.1) .. controls (7.6,1) and (7.2,0.4) .. (7.1,0.15) -- (6.85,0.5) -- (5.15,0.5) -- (4.9,0.1) .. controls (4.7,0.75) and (4.2,1.025) .. (4.15,1.15);

\draw[color=black,line width=1.5pt] (0,4) parabola bend (3,-0.5) (6,4);
\draw[color=black,line width=1.5pt] (6,4) parabola bend (9,-0.5) (12,4);
\draw (6,3.25) node{$Y$};
\draw (6,2.2) node{{\color{blue}$\beta^*$}};
\draw (6,1.25) node{{\color{red}$\beta$}};
\draw (5,2.1) node{{\color{blue}$\gamma_i$}};

\draw[color=red,line width=1.5pt] (1.5,0.625) parabola bend (3,-0.5) (5,1.5) -- (7,1.5) parabola bend (9,-0.5) (10.5,0.625);

\draw[color=blue,line width=1.5pt] (2,0.4) parabola bend (3,-0.5) (4,0.4);
\draw[color=blue,line width=1.5pt] (8,0.25) parabola bend (9,-0.5) (10,0.25);

\draw[color=blue,line width=1.5pt] (4,0.4) -- (4.5,0.2) .. controls (4.45,1.6) and (4.75,1.75) .. (4.75,1.75) .. controls (5,2) and (7,2) .. (7.25,1.75) .. controls (7.25,1.75) and (7.45,1.6) .. (7.5,0.25) -- (8,0.25);

\draw (1.5,0.625) node{$\bullet$};
\draw (3,-0.5) node{$\bullet$};
\draw (9,-0.5) node{$\bullet$};
\draw (2,0) node{$\bullet$};
\draw (2,0.4) node{$\bullet$};
\draw (4,0) node{$\bullet$};
\draw (4,0.4) node{$\bullet$};
\draw (4.5,0.2) node{$\bullet$};
\draw (7.5,0.25) node{$\bullet$};
\draw (8,0) node{$\bullet$};
\draw (8,0.25) node{$\bullet$};
\draw (10,0) node{$\bullet$};
\draw (10,0.25) node{$\bullet$};
\draw (10.5,0.625) node{$\bullet$};

\draw (3,1) node{$\Bb_i$};
\draw (9,1) node{$\Bb_i'$};
\draw (3,-0.15) node{{\color{blue}$A_i$}};
\draw (9,-0.15) node{{\color{blue}$B_i$}};
\draw (4,-0.3) node{{\color{red}$A_i'$}};
\draw (8,-0.3) node{{\color{red}$B_i'$}};

\draw (3,-0.75) node{\footnotesize$p_i$};
\draw (9,-0.75) node{\footnotesize$q_i$};
\draw (1.5,-0.25) node{$\Cc_i$};
\draw (4.5,-0.25) node{$\Dd_i$};
\draw (7.5,-0.25) node{$\Cc_i'$};
\draw (10.5,-0.25) node{$\Dd_{i+1}'$};
\draw (4.75,0.15) node{\footnotesize$a_i$};
\draw (7.35,0.15) node{\footnotesize$b_i'$};

\draw[color=black,line width=1pt] (1.25,-1.5) -- (10.75,-1.5);
\draw (1,-1.5) node{$\ldots$};
\draw (11,-1.5) node{$\ldots$};

\draw[color=black,line width=1pt](0,-1.5) -- (0.75,-1.5);
\draw[color=black,line width=1pt](11.25,-1.5) -- (12,-1.5);
\draw (0.5,-1.5) node{$\bullet$};
\draw (0.5,-1.25) node{\footnotesize$0$};
\draw (11.5,-1.5) node{$\bullet$};
\draw (11.5,-1.25) node{\footnotesize$1$};

\draw (1.5,-1.5) node{$\bullet$};
\draw (2,-1.5) node{$\bullet$};
\draw (3,-1.5) node{$\bullet$};
\draw (4,-1.5) node{$\bullet$};
\draw (4.5,-1.5) node{$\bullet$};
\draw (7.5,-1.5) node{$\bullet$};
\draw (8,-1.5) node{$\bullet$};
\draw (9,-1.5) node{$\bullet$};
\draw (10,-1.5) node{$\bullet$};
\draw (10.5,-1.5) node{$\bullet$};

\draw (1.5,-1.25) node{\footnotesize$t_i-\delta\ \ \ \ $};
\draw (2,-1.25) node{\footnotesize$\ \ \ \ t_i-\frac{\delta}{2}$};
\draw (3,-1.25) node{\footnotesize$t_i$};
\draw (4,-1.25) node{\footnotesize$t_i+\frac{\delta}{2}\ \ \ \ $};
\draw (4.5,-1.25) node{\footnotesize$\ \ \ \ t_i+\delta$};
\draw (7.5,-1.25) node{\footnotesize$s_i-\delta\ \ \ \ $};
\draw (8,-1.25) node{\footnotesize$\ \ \ \ s_i-\frac{\delta}{2}$};
\draw (9,-1.25) node{\footnotesize$s_i$};
\draw (10,-1.25) node{\footnotesize$s_i+\frac{\delta}{2}$\ \ \ \ };
\draw (10.5,-1.25) node{\ \ \ \ \footnotesize$s_i+\delta$};

\draw (1.5,3) node{$\Ss_i$};
\draw (11.25,3) node{$\Ss_{i+1}$};

\end{tikzpicture}
\caption{Construction of the Nash path $\gamma_i$ and the corresponding part of $\beta^*$.\label{fig7}}
\vspace{-1.75em}
\end{figure}
\end{center} 

\begin{center}
\begin{figure}[ht]
\begin{tikzpicture}[scale=1.2]

\draw[dashed,line width=1.5pt] (0,-0.5) -- (12,-0.5);
\draw[dashed,line width=1.5pt] (12,4) -- (5,4) parabola bend (3,-0.5) (1,4) -- (0,4);

\draw[fill=gray!100,opacity=0.5,draw=none] (0,-0.5) -- (12,-0.5) -- (12,4) -- (5,4) parabola bend (3,-0.5) (1,4) -- (0,4);
\draw[fill=white!100,opacity=1,draw=none] (9,1.5) ellipse (1cm and 2cm);
\draw[dashed,line width=1.5pt] (9,1.5) ellipse (1cm and 2cm);

\draw[color=green,dashed,line width=1pt] (1,-0.5) -- (5,-0.5) arc(0:180:2);
\draw[color=green,dashed,line width=1pt] (7,-0.5) -- (11,-0.5) arc(0:180:2) --cycle;

\draw[color=green,dashed,line width=1pt] (4.2,1.1) .. controls (4.2,1.125) and (4.75,2) .. (4.75,2.5) -- (7.15,2.5) .. controls (7.15,2) and (7.8,1) .. (7.8,1.1);
\draw[color=green,dashed,line width=1pt] (7.1,0.15) -- (6.85,0.5) -- (5.15,0.5) -- (4.9,0.1);

\draw[fill=green!100,opacity=0.25,draw=none] (1,-0.5) -- (5,-0.5) arc(0:180:2) --cycle;
\draw[fill=green!100,opacity=0.25,draw=none] (7,-0.5) -- (11,-0.5) arc(0:180:2) --cycle;

\draw[fill=green!100,opacity=0.25,draw=none] (4.2,1.1) .. controls (4.2,1.125) and (4.75,2) .. (4.75,2.5) -- (7.15,2.5) .. controls (7.15,2) and (7.8,1) .. (7.8,1.1) .. controls (7.6,1) and (7.2,0.4) .. (7.1,0.15) -- (6.85,0.5) -- (5.15,0.5) -- (4.9,0.1) .. controls (4.7,0.75) and (4.2,1.025) .. (4.15,1.15);

\draw[color=black,line width=1.5pt] (0,4) parabola bend (3,-0.5) (6,4);
\draw[color=black,line width=1.5pt] (6,4) parabola bend (9,-0.5) (12,4);
\draw (6,3.25) node{$Y$};
\draw (6,2.2) node{{\color{blue}$\beta^*$}};
\draw (6,1.25) node{{\color{red}$\beta$}};
\draw (5,2.1) node{{\color{blue}$\sigma_i$}};

\draw[color=red,line width=1.5pt] (1.5,0.625) parabola bend (3,-0.5) (5,1.5) -- (7,1.5) parabola bend (9,-0.5) (10.5,0.625);

\draw[color=blue,line width=1.5pt] (2,0.25) parabola bend (3,-0.5) (4,0.25);
\draw[color=blue,line width=1.5pt] (8,0.4) parabola bend (9,-0.5) (10,0.4);

\draw[color=blue,line width=1.5pt] (4,0.25) -- (4.5,0.2) .. controls (4.45,1.6) and (4.75,1.75) .. (4.75,1.75) .. controls (5,2) and (7,2) .. (7.25,1.75) .. controls (7.25,1.75) and (7.45,1.6) .. (7.5,0.25) -- (8,0.4);

\draw (1.5,0.625) node{$\bullet$};
\draw (3,-0.5) node{$\bullet$};
\draw (9,-0.5) node{$\bullet$};
\draw (2,0) node{$\bullet$};
\draw (2,0.25) node{$\bullet$};
\draw (4,0) node{$\bullet$};
\draw (4,0.25) node{$\bullet$};
\draw (4.5,0.2) node{$\bullet$};
\draw (7.5,0.25) node{$\bullet$};
\draw (8,0.4) node{$\bullet$};
\draw (10,0) node{$\bullet$};
\draw (10,0.4) node{$\bullet$};
\draw (10.5,0.625) node{$\bullet$};

\draw (3,1) node{$\Bb_i'$};
\draw (9,1) node{$\Bb_{i+1}$};
\draw (3,-0.15) node{{\color{blue}$B_i$}};
\draw (9,-0.15) node{{\color{blue}$A_{i+1}$}};
\draw (4,-0.3) node{{\color{red}$B_i'$}};
\draw (8,-0.3) node{{\color{red}$A_{i+1}'$}};

\draw (3,-0.75) node{\footnotesize$q_i$};
\draw (9,-0.75) node{\footnotesize$p_{i+1}$};
\draw (1.5,-0.25) node{$\Cc_i'$};
\draw (4.65,-0.25) node{$\Dd_{i+1}'$};
\draw (7.35,-0.25) node{$\Cc_{i+1}$};
\draw (10.5,-0.25) node{$\Dd_{i+1}$};
\draw (4.85,0.15) node{\footnotesize$a_{i+1}'$};
\draw (7.35,0.15) node{\footnotesize$b_{i+1}$};

\draw[color=black,line width=1pt] (1.25,-1.5) -- (10.75,-1.5);
\draw (1,-1.5) node{$\ldots$};
\draw (11,-1.5) node{$\ldots$};

\draw[color=black,line width=1pt](0,-1.5) -- (0.75,-1.5);
\draw[color=black,line width=1pt](11.25,-1.5) -- (12,-1.5);
\draw (0.5,-1.5) node{$\bullet$};
\draw (0.5,-1.25) node{\footnotesize$0$};
\draw (11.5,-1.5) node{$\bullet$};
\draw (11.5,-1.25) node{\footnotesize$1$};

\draw (1.5,-1.5) node{$\bullet$};
\draw (2,-1.5) node{$\bullet$};
\draw (3,-1.5) node{$\bullet$};
\draw (4,-1.5) node{$\bullet$};
\draw (4.5,-1.5) node{$\bullet$};
\draw (7.5,-1.5) node{$\bullet$};
\draw (8,-1.5) node{$\bullet$};
\draw (9,-1.5) node{$\bullet$};
\draw (10,-1.5) node{$\bullet$};
\draw (10.5,-1.5) node{$\bullet$};

\draw (1.5,-1.25) node{\footnotesize$s_i-\delta\ \ \ \ $};
\draw (2,-1.25) node{\footnotesize$\ \ \ \ s_i-\frac{\delta}{2}$};
\draw (3,-1.25) node{\footnotesize$s_i$};
\draw (4,-1.25) node{\footnotesize$s_i+\frac{\delta}{2}\ \ \ \ $};
\draw (4.5,-1.25) node{\footnotesize$\ \ \ \ s_i+\delta$};
\draw (7.5,-1.25) node{\scriptsize$t_{i+1}-\delta\ \ \ \ \ \ \ \ $};
\draw (8,-1.25) node{\scriptsize$\ \ \ \ \ t_{i+1}-\frac{\delta}{2}$};
\draw (9,-1.25) node{\scriptsize$t_{i+1}$};
\draw (10,-1.25) node{\scriptsize$t_{i+1}+\frac{\delta}{2}$\ \ \ \ \ };
\draw (10.5,-1.25) node{\ \ \ \ \ \ \ \scriptsize$t_{i+1}+\delta$};

\draw (0.75,3) node{$\Ss_i$};
\draw (7.5,3) node{$\Ss_{i+1}$};

\end{tikzpicture}
\caption{Construction of the Nash path $\sigma_i$ and the corresponding part of $\beta^*$.\label{fig8}}
\vspace{-1.75em}
\end{figure}
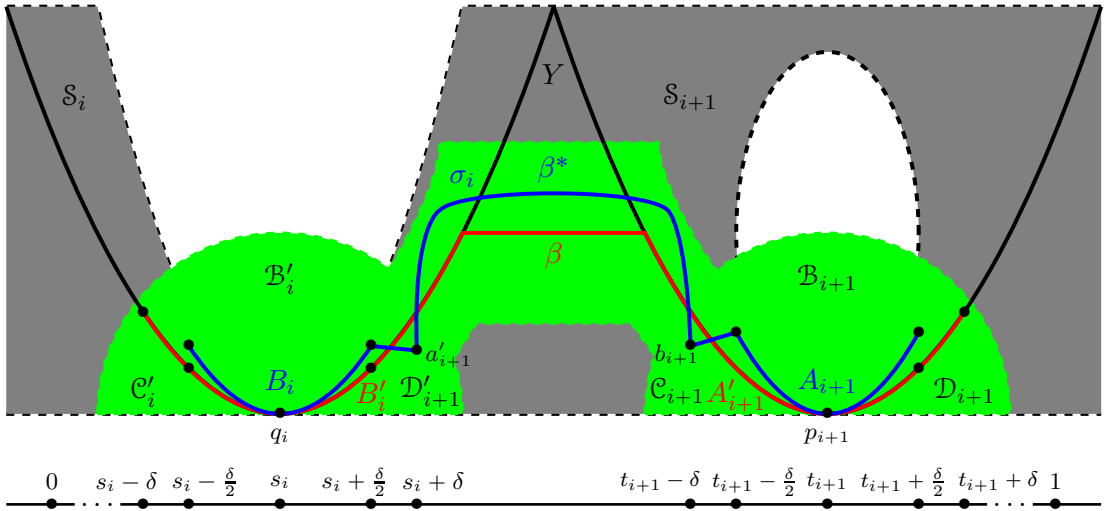
\end{center} 

\begin{proof}
We fix $\veps>0$ and conduct the proof of this result in several steps: 

\noindent{\sc Step 1. (Local) modification of $\beta$ around the points $p_i$.}
Fix an index $i=1,\ldots,r$. {\em We modify $\beta$ in a neighborhood of $t_i$ so that the new $\beta$ is a Nash map around $t_i$ and $\beta([t_i-\delta,t_i+\delta]\cap Y\subset\{p_i\}$ if $\delta>0$ is small enough}.

Consider the open ball $\Bb_i$ of center $p_i$ and radius $\frac{\veps}{3}$. Let $\delta_0>0$ be such that $\beta|_{[t_i-\delta_0,t_i+\delta_0]}$ is a Nash path whose image is contained in $(\Ss_i\cap\Bb_i)\cup\{p_i\}$. Let $\Cc_i$ and $\Dd_i$ be the connected components of $\Ss_i\cap\Bb_i$ (maybe the same) such that $A_i':=\beta([t_i-\delta,t_i+\delta])\subset\Cc_i\cup\Dd_i\cup\{p_i\}$ (for some $0<\delta<\delta_0$ small enough). By \cite[Main Thm.1.4]{f1} the semialgebraic set $\Cc_i\cup\Dd_i\cup\{p_i\}$ is a Nash image of $\R^d$ (where $d:=\dim(\Ss)$) and it is connected by analytic paths. By either \cite[Prop.7.8 \& Cor.7.9]{f1} or Lemma \ref{doublecurve} (the first reference if $\Cc_i\neq\Dd_i$ and the second reference if $\Cc_i=\Dd_i$) we may find a Nash bridge (or Nash arc) $A_i\subset\Cc_i\cup\Dd_i\cup\{p_i\}$ such that $A_i\cap Y\subset\{p_i\}$. As $\Ss_i\cap\Bb_i$ is a Nash manifold, both $\Cc_i$ and $\Dd_i$ are connected Nash manifolds.   
 
\noindent{\sc Step 2. (Local) modification of $\beta$ around the points $q_i$.}
Fix an index $i=1,\ldots,r-1$. {\em We modify $\beta$ in a neighborhood of $s_i$ so that the new $\beta$ is a Nash map around $s_i$ and $\beta([s_i-\delta,s_i+\delta])\cap Y\subset\{q_i\}$ if $\delta>0$ is small enough}.

Let $\Bb_i'$ be the ball of center $q_i$ and radius $\frac{\veps}{3}$. Let $\delta_0>0$ be such that $\beta|_{[s_i-\delta_0,s_i+\delta_0]}$ is a Nash path whose image is contained in $((\Ss_i\cap\Ss_{i+1})\cap\Bb_i')\cup\{q_i\}$. Let $\Cc_i'$ and $\Dd_{i+1}'$ be the respective connected components of $\Ss_i\cap\Bb_i'$ and $\Ss_{i+1}\cap\Bb_i'$ (maybe the same if $\Ss_i=\Ss_{i+1}$) such that $B_i':=\beta([s_i-\delta,s_i+\delta])\subset\Cc_i'\cup\Dd_{i+1}'\cup\{q_i\}$ (for some $0<\delta<\delta_0$ small enough). By \cite[Main Thm.1.4]{f1} the semialgebraic set $\Cc_i'\cup\Dd_{i+1}'\cup\{q_i\}$ is a Nash image of $\R^d$ (where $d:=\dim(\Ss)$) and it is connected by analytic paths. By \cite[Prop.7.8 \& Cor.7.9]{f1} or Lemma \ref{doublecurve} (the first reference if $\Cc_i'\neq\Dd_{i+1}'$ and the second reference if $\Cc_i'=\Dd_{i+1}'$) we may find a Nash bridge (or a Nash arc) $B_i\subset\Cc_i'\cup\Dd_{i+1}'\cup\{q_i\}$ such that $B_i\cap Y\subset\{q_i\}$. As $\Ss_i\cap\Bb_i'$ is a Nash manifold, both $\Cc_i'$ and $\Dd_{i+1}'$ are connected Nash manifolds. 
 
\noindent{\sc Step 3. Modification of $\beta$ outside a neighborhood of $\{p_1,\ldots,p_r,q_1,\ldots,q_{r-1}\}$.}
Taking a smaller $\delta>0$ if necessary, we may assume $b_{i0}:=\beta(t_i-\delta)\in\Cc_i$, $a_{i0}:=\beta(t_i+\delta)\in\Dd_i$ for $i=1,\ldots,r$ and $b_{i0}':=\beta(s_i-\delta)\in\Cc_i'$, $a_{i+1,0}':=\beta(s_i+\delta)\in\Dd_{i+1}'$ for $i=1,\ldots,r-1$. If $\beta([t_i+\delta,s_i-\delta])\cap Y$ is a finite set, we do nothing with this semialgebraic set. If $\beta([s_i+\delta,t_{i+1}-\delta])\cap Y$ is a finite set, we also do nothing. Let us modify $\beta([t_i+\delta,s_i-\delta])$ if the intersection $\beta([t_i+\delta,s_i-\delta])\cap Y$ has dimension $1$ (Figure \ref{fig7}). 

Pick points $a_{i1}\in\Dd_i\setminus Y$ and $b_{i1}'\in\Cc_i'\setminus Y$ and let 
$$
\beta_i:[t_i+\delta/2,s_i-\delta/2]\to\Dd_i\cup\beta([t_i+\delta,s_i-\delta])\cup\Cc_i'\subset\Ss_i
$$ 
be a continuous semialgebraic path such that $\beta_i|_{[t_i+\delta,s_i-\delta]}=\beta|_{[t_i+\delta,s_i-\delta]}$, $\beta_i(t_i+\delta/2)=a_{i1}$, $\beta_i([t_i+\delta/2,t_i+\delta])\subset\Dd_i$, $\beta_i(s_i-\delta/2)=b_{i1}'$ and $\beta_i([s_i-\delta,s_i-\delta/2])\subset\Cc_i'$. 

Define 
$$
\veps':=\min\{\veps,\dist(a_{i0},\Ss_i\setminus\Dd_i),\dist(a_{i1},\Ss_i\setminus(\Dd_i\setminus Y)),\dist(b_{i0}',\Ss_i\setminus\Cc_i'),\dist(b_{i1}',\Ss_i\setminus(\Cc_i'\setminus Y))\}>0. 
$$
By \cite[Cor.8.9.6]{bcr} there exists a Nash path $\gamma_i:[t_i+\delta/2,s_i-\delta/2]\to\Ss_i$ such that $\|\beta_i-\gamma_i\|<\frac{\veps'}{3}$. We have $\gamma_i(t_i+\delta/2)\in\Dd_i\setminus Y$, $a_i:=\gamma_i(t_i+\delta)\in\Dd_i$, $b_i':=\gamma_i(s_i-\delta)\in\Cc_i'$ and $\gamma_i(s_i-\delta/2)\in\Cc_i'\setminus Y$. By \cite[Lem.7.7]{f1} we deduce $\gamma_i^{-1}(Y)$ is a finite set. As $\gamma_i$ is Nash, $\eta(\gamma_i)=\varnothing$. 

Analogously, if $\beta([s_i+\delta,t_{i+1}-\delta])\cap Y$ has dimension $1$, one constructs (as before) a Nash path $\sigma_i:[s_i+\delta/2,t_{i+1}-\delta/2]\to\Ss_{i+1}$ such that $\|\beta|_{[s_i+\delta/2,t_{i+1}-\delta/2]}-\sigma_i\|<\frac{\veps}{3}$, $\sigma_i(s_i+\delta/2)\in\Dd_{i+1}'\setminus Y$, $a_{i+1}':=\sigma_i(s_i+\delta)\in\Dd_{i+1}'$, $b_{i+1}:=\sigma_i(t_{i+1}-\delta)\in\Cc_{i+1}$, $\sigma_i(t_{i+1}-\delta/2)\in\Cc_{i+1}\setminus Y$ and $\sigma_i^{-1}(Y)$ is a finite set (Figure \ref{fig8}). Again, as $\sigma_i$ is Nash, $\eta(\sigma_i)=\varnothing$. 
 
\noindent{\sc Step 4. Full modification of $\beta$.}
Recall that if $x,y\in\Bb_i$ (or $x,y\in\Bb_i'$), then $\|x-y\|<\frac{2\veps}{3}$. In addition, $\Cc_i\subset\Ss_i\cap\Bb_i$, $\Dd_i\subset\Ss_i\cap\Bb_i$, $\Cc_i'\subset\Ss_i\cap\Bb_i'$ and $\Dd_{i+1}'\subset\Ss_{i+1}\cap\Bb_i'$ are connected Nash manifolds. By \cite[Thm.1.5]{f1} each connected Nash manifold is connected by Nash paths. Thus, we can construct a continuous semialgebraic path $\beta^*:[0,1]\to\Ss$ that connects, using additional Nash paths that avoid $Y$ except for perhaps finitely many points, the already constructed Nash arcs (in {\sc Step} 1), Nash bridges (in {\sc Step} 2) and Nash paths (in {\sc Step} 3) and satisfies the following conditions:
\begin{itemize}
\item $\beta^*|_{[0,t_1-\delta]}=\beta|_{[0,t_1-\delta]}$ and $\beta^*|_{[t_r+\delta,1]}=\beta|_{[t_r+\delta,1]}$.
\item $\beta^*|_{[t_i-\frac{\delta}{2},t_i+\frac{\delta}{2}]}:[t_i-\frac{\delta}{2},t_i+\frac{\delta}{2}]\to A_i\subset\Cc_i\cup\Dd_i\cup\{p_i\}\subset\Ss_i\cap\Bb_i$ is a Nash parameterization of $A_i$ around $p_i=\beta^*(t_i)$.
\item $\beta^*|_{[s_i-\frac{\delta}{2},s_i+\frac{\delta}{2}]}:[s_i-\frac{\delta}{2},s_i+\frac{\delta}{2}]\to B_i\subset\Cc_i'\cup\Dd_{i+1}'\cup\{q_i\}\subset(\Ss_i\cup\Ss_{i+1})\cap\Bb_i'$ is a Nash parameterization of $B_i$ around $q_i=\beta^*(s_i)$.
\item $\beta^*|_{[t_i+\delta,s_i-\delta]}=\gamma_i|_{[t_i+\delta,s_i-\delta]}$ and $\beta^*|_{[s_i+\delta,t_{i+1}-\delta]}=\sigma_i|_{[s_i+\delta,t_{i+1}-\delta]}$,
\item $\beta^*([t_i+\frac{\delta}{2},t_i+\delta])\subset\Dd_i\subset\Ss_i\cap\Bb_i$ and $\beta^*([s_i-\delta,s_i-\frac{\delta}{2}])\subset\Cc_i'\subset\Ss_i\cap\Bb_i'$,
\item $\beta^*([s_i+\frac{\delta}{2},s_i+\delta])\subset\Dd_{i+1}'\subset\Ss_{i+1}\cap\Bb_i'$ and $\beta^*([t_{i+1}-\delta,t_{i+1}-\frac{\delta}{2}])\subset\Cc_{i+1}\subset\Ss_{i+1}\cap\Bb_{i+1}$,
\item $\eta(\beta^*)\subset\bigcup_{i=1}^r\{t_i-\delta,t_i-\frac{\delta}{2},t_i+\frac{\delta}{2},t_i+\delta\}\cup\bigcup_{i=1}^{r-1}\{s_i-\delta,s_i-\frac{\delta}{2},s_i+\frac{\delta}{2},s_i+\delta\}$ and $\beta^*(\eta(\beta^*))\subset\bigcup_{i=1}^r\Ss_i$,
\item $(\beta^*)^{-1}(Y)$ is a finite set, $\eta(\beta^*)\cap(\beta^*)^{-1}(Y)=\varnothing$ and $\eta(\beta^*)\cap\{t_1,\ldots,t_r,s_1,\ldots,s_{r-1}\}=\varnothing$.
\end{itemize}

Following the construction of $\beta^*$ we have done, one deduces that $\|\beta^*-\beta\|<\veps$. Thus, $\beta^*:[0,1]\to\R^n$ is a semialgebraic path close to $\beta$ that satisfies the required conditions (i), (ii) and (iii) in the statement. In addition, $\beta^*([0,1])\cap Y$ is a finite set, $\eta(\beta^*)\cap(\beta^*)^{-1} (Y)=\varnothing$, $\eta(\beta^*)\cap\{t_1,\ldots,t_r,s_1,\ldots,s_{r-1}\}=\varnothing$ and $\beta^*(\eta(\beta^*))\subset\bigcup_{i=1}^r\Ss_i$, as required.
\end{proof}

\bibliographystyle{amsalpha}

\end{document}